\documentclass[onefignum,onetabnum]{siamonline190516}

\usepackage{lipsum}
\usepackage{amsfonts}
\usepackage{graphicx}
\usepackage{epstopdf}

\usepackage{enumitem}
\usepackage{latexsym,mathrsfs}
\usepackage{amsmath,amssymb} 
\usepackage{url}
\usepackage{pifont}
\usepackage{mathtools}
\usepackage{makecell}
\usepackage[normalem]{ulem}
\usepackage{hyperref}

\usepackage{pgfplots,latexsym}
\usepackage{amsmath,amssymb,floatflt,amsfonts,amscd,graphics,mathrsfs} 
\usepackage{etoolbox,verbatim,color}
\usepackage{url}

\usepackage{hyperref}

\usepackage[referable]{threeparttablex}
\usepackage{footnote}
\usepackage{mathtools}
\usepackage{lipsum}
\usepackage{algorithmic}
\usepackage{algorithm}
\usepackage{soul}
\newtheorem{assumption}{Assumption}

\newcommand{\mbfr}{\mathbf{r}}
\newcommand{\mbfm}{\mathbf{m}}
\newcommand{\mbfn}{\mathbf{n}}

\definecolor{brightpink}{rgb}{1.0, 0.0, 0.5}

\ifpdf
\DeclareGraphicsExtensions{.eps,.pdf,.png,.jpg}
\else
\DeclareGraphicsExtensions{.eps}
\fi

\setlist[enumerate]{leftmargin=.5in}
\setlist[itemize]{leftmargin=.5in}

\newsiamremark{remark}{Remark}
\newsiamremark{hypothesis}{Hypothesis}
\crefname{hypothesis}{Hypothesis}{Hypotheses}
\newsiamthm{claim}{Claim}
\newsiamremark{example}{Example}

\usepackage{multicol}
\usepackage{multirow}
\catcode`~=11 \def\UrlSpecials{\do\~{\kern -.15em\lower .7ex\hbox{~}\kern .04em}} \catcode`~=13 

\newcommand{\iprod}[2]{\left\langle {#1}, {#2} \right\rangle}

\newcommand{\lrbrace}[1]{\left\{#1\right\}}

\newcommand{\dom}{\mathbf{dom}}

\newcommand{\inter}{\mathbf{int}\,}

\newcommand{\bbE}{\mathbb{E}}

\newcommand{\bbR}{\mathbb{R}}

\DeclareMathAlphabet{\mathbsf}{OT1}{cmss}{bx}{n}
\DeclareMathAlphabet{\mathssf}{OT1}{cmss}{m}{sl}

\DeclareSymbolFont{bsfletters}{OT1}{cmss}{bx}{n}  
\DeclareSymbolFont{ssfletters}{OT1}{cmss}{m}{n}
\DeclareMathSymbol{\bsfGamma}{0}{bsfletters}{'000}
\DeclareMathSymbol{\ssfGamma}{0}{ssfletters}{'000}
\DeclareMathSymbol{\bsfDelta}{0}{bsfletters}{'001}
\DeclareMathSymbol{\ssfDelta}{0}{ssfletters}{'001}
\DeclareMathSymbol{\bsfTheta}{0}{bsfletters}{'002}
\DeclareMathSymbol{\ssfTheta}{0}{ssfletters}{'002}
\DeclareMathSymbol{\bsfLambda}{0}{bsfletters}{'003}
\DeclareMathSymbol{\ssfLambda}{0}{ssfletters}{'003}
\DeclareMathSymbol{\bsfXi}{0}{bsfletters}{'004}
\DeclareMathSymbol{\ssfXi}{0}{ssfletters}{'004}
\DeclareMathSymbol{\bsfPi}{0}{bsfletters}{'005}
\DeclareMathSymbol{\ssfPi}{0}{ssfletters}{'005}
\DeclareMathSymbol{\bsfSigma}{0}{bsfletters}{'006}
\DeclareMathSymbol{\ssfSigma}{0}{ssfletters}{'006}
\DeclareMathSymbol{\bsfUpsilon}{0}{bsfletters}{'007}
\DeclareMathSymbol{\ssfUpsilon}{0}{ssfletters}{'007}
\DeclareMathSymbol{\bsfPhi}{0}{bsfletters}{'010}
\DeclareMathSymbol{\ssfPhi}{0}{ssfletters}{'010}
\DeclareMathSymbol{\bsfPsi}{0}{bsfletters}{'011}
\DeclareMathSymbol{\ssfPsi}{0}{ssfletters}{'011}
\DeclareMathSymbol{\bsfOmega}{0}{bsfletters}{'012}
\DeclareMathSymbol{\ssfOmega}{0}{ssfletters}{'012}

\DeclareMathOperator*{\argmin}{argmin} 
\DeclareMathOperator*{\dist}{dist}

\DeclareMathOperator{\minimize}{minimize}

\headers{Block Alternating Bregman Majorization Minimization with Extrapolation}{L.T.K. Hien, D.N. Phan, N. Gillis, M. Ahookhosh, P. Patrinos}

\title{Block Bregman Majorization Minimization with Extrapolation\thanks{LTK Hien and DN Phan contributed equally to this work.
		\funding{The authors acknowledge the support by the European Research Council (ERC starting grant no 679515), and by the Fonds de la Recherche Scientifique - FNRS and the Fonds Wetenschappelijk Onderzoek - Vlaanderen (FWO) under EOS Project no O005318F-RG47.}}}

\author{Le~Thi~Khanh~Hien\thanks{Department of Mathematics and Operational Research, 
		Facult\'e Polytechnique, Universit\'e de Mons, 
		Rue de Houdain 9, 7000 Mons, Belgium 
		(\email{thikhanhhien.le@umons.ac.be}, \email{nicolas.gillis@umons.ac.be}).}
  \and Duy~Nhat~Phan\thanks{Dynamic Decision Making Laboratory, Carnegie Mellon University, USA (\email{dnphan@andrew.cmu.edu}).}   
	\and Nicolas~Gillis\footnotemark[2] 
	\and Masoud~Ahookhosh\thanks{Department of Mathematics, University of Antwerp, Belgium (\email{masoud.ahookhosh@uantwerp.be}).} 
             \and 
              Panagiotis~Patrinos\thanks{Department of Electrical Engineering (ESAT-STADIUS), KU Leuven, Belgium 
               (\email{panos.patrinos@esat.kuleuven.}).}
	}

\usepackage{amsopn}

\ifpdf
\hypersetup{
  pdftitle={An Example Article},
  pdfauthor={D. Doe, P. T. Frank, and J. E. Smith}
}
\fi

\begin{document}

\maketitle

\begin{abstract}
In this paper, we consider a class of nonsmooth nonconvex optimization problems whose objective is the sum of a block relative smooth function and a proper and lower semicontinuous block separable function. Although the analysis of block proximal gradient (BPG) methods  for the class of block $L$-smooth functions have been successfully extended to Bregman BPG methods that deal with the class of block relative smooth functions, accelerated Bregman BPG methods are  scarce and challenging to design. Taking our inspiration from Nesterov-type acceleration and the majorization-minimization scheme, we propose a block alternating Bregman Majorization-Minimization framework with Extrapolation (BMME). We prove subsequential convergence of BMME to a first-order stationary point under mild assumptions, and study its  global convergence under stronger conditions. We illustrate the effectiveness of BMME on the penalized orthogonal nonnegative matrix factorization problem. 
\end{abstract}



\section{Introduction}
\label{sec:intro}
In this paper, we consider the following nonsmooth nonconvex optimization problem
\begin{equation}
\label{model}
\begin{split}
  \displaystyle \minimize_{\substack{x  = (x_1,\ldots,x_m)}} \quad   
    & F(x) := f(x) + \sum_{i=1}^m g_i(x_i) \\
     \text{ subject to } \quad 
     &  x_i \in \mathcal X_i \text{ for } i=1,\ldots,m,
\end{split}
\end{equation}
where $\mathcal X_i$ is a closed convex set of a finite dimensional real linear space  $\mathbb E_i$ for $i \in [m]:=\{1,2,\ldots,m\}$, $x$ can be decomposed into $m$ blocks $ x=(x_1,\ldots,x_m)$ with $x_i\in \mathcal X_i$,  $f$ is a continuously differentiable function, and $g_i$ is a proper and lower semicontinuous function (possibly with extended values), and $\mathcal X_i \cap \dom\, g_i \ne \emptyset$. We denote $\mathcal X:=\prod_{i=1}^m \mathcal X_i$. We assume $F$ is bounded from below throughout the paper.

\subsection{Related works} 

The composite separable optimization  problem (CSOP)~\eqref{model} has been widely studied. It covers many applications including compressed sensing~\cite{Attouch2010}, sparse dictionary learning~\cite{aharon2006k,XuYin2016}, 
nonnegative tensor 
factorization~\cite{Xu2013,Hien_ICML2020}, 
and regularized sparse regression problems~\cite{BLUMENSATH2009,Natarajan1995}. 
When $f$ has the block Lipschitz smooth property (that is, for all $i\in [m]$ and fixing the values of  $x_j$ for $j\ne i$, the block function $x_i\mapsto f(x)$ admits an $L_i$-Lipschitz continuous gradient), then the nonconvex CSOP can be efficiently solved by  block proximal gradient (BPG) methods \cite{Beck2013,Bolte2014,Razaviyayn2013,Tseng2009}. These methods update each block $i$,  while fixing the value of 
blocks $x_j$ for $j\ne i$,  by minimizing over $x_i\in \mathcal X_i$ 
the block Lipschitz gradient surrogate function (see \cite[Section 4]{Titan2020}) as follows
\begin{equation}
\label{lipschitz-surrogate}
x_i^{k+1} \in\argmin_{x_i\in\mathcal X_i } \iprod{ \nabla f_i^k(x_i^{k})}{x_i-x_i^{k}}  + \frac{1}{2\gamma_i^{k}} \|x_i-x_i^{k}\|^2 +g_i(x_i),
\end{equation} 
where $x_i^k$ is the value of block $i$ at iteration $k$, $f_i^k(\cdot)$ denotes the value of the block function $x_i\to f(x_1^{k+1},\ldots,x_{i-1}^{k+1},x_i,x_{i+1}^k,\ldots,x_m^k)$, and $\gamma_i^k$ is a step-size. To accelerate the BPG methods, several inertial versions have been proposed, including 
\begin{itemize}
    \item[(i)] the heavy ball type acceleration methods in~\cite{Ochs2019} that calculate an extrapolation point $\bar x_i^k=x_i^k+\beta_i^k(x_i^k-x_i^{k-1})$,  then solving~\eqref{lipschitz-surrogate} by replacing the proximal term $\frac{1}{2\gamma_i^{k}} \|x_i-x_i^{k}\|^2$ by $\frac{1}{2\gamma_i^{k}} \|x_i-\bar x_i^{k}\|^2$, 
    
    \item[(ii)] the Nesterov-type acceleration methods in \cite{Xu2013,Xu2017} that takes the same step as the heavy ball acceleration but also replace $\nabla f_i^k(x_i^{k})$ in \eqref{lipschitz-surrogate} by $\nabla f_i^k(\bar x_i^{k})$, and 
    
    \item[(iii)] the acceleration methods using two extrapolation points in~\cite{Hien_ICML2020,Pock2016} that evaluate the gradient $\nabla f_i^k$ in \eqref{lipschitz-surrogate} at an extrapolation point different from $\bar x_i^k$. 
\end{itemize}
These methods were proved to have convergence guarantees when solving the nonconvex CSOP. 
 The analysis of BPG methods has been extended to Bregman BPG methods  \cite{ahookhosh2019multi,HienNicolas_KLNMF,Teboulle2020} that replace the proximal term $\frac12\|x_i-x_i^{k}\|^2$ in~\eqref{lipschitz-surrogate} by a Bregman divergence $D_{\psi_i}(x_i,x_i^k)$ associated with a kernel  function $\psi_i$ (see Definition~\ref{def:Bregman_divergence}) as follows 
 \begin{equation}
\label{Bregman-surrogate}
\min_{x_i\in\mathcal X_i } \iprod{ \nabla f_i^k(x_i^{k})}{x_i-x_i^{k}}  + \frac{1}{\gamma_i^{k}} D_{\psi_i}(x_i,x_i^k) +g_i(x_i).
\end{equation}
The Bregman BPG methods can deal with a larger class of nonconvex CSOP in which the block function $x_i \mapsto f(x)$ 
may not have a $L_i$-Lipschitz continuous gradient, but is a relative smooth function  (also known as a smooth adaptable function)~\cite{Bauschke2017,lu2018relatively,Bolte2018}. 
Although the convergence analysis of BPG methods has been successfully extended to Bregman BPG methods, the convergence guarantees of their inertial versions for solving the nonconvex CSOP have not been studied much. 
In fact, to the best of our knowledge, there are only two papers addressing the convergence of inertial versions of Bregman BPG methods for solving~\eqref{model},   namely \cite{ahookhosh2020_inertial} and \cite{Titan2020}. In \cite{ahookhosh2020_inertial}, the authors consider an inertial Bregman BPG method that adds to $\nabla f_i^k(x_i^{k})$ in~\eqref{Bregman-surrogate} a \emph{weak inertial force}, $\alpha_i^k(x_i^k - x_i^{prev})$,  where $\alpha_i^k$ is some extrapolation parameter and $x_i^{prev}$ is the previous value of $x_i^k$. In~\cite[Section 4.3]{Titan2020}, the authors introduce a heavy ball type acceleration with backtracking. The analysis of this method can be extended to a Nesterov-type acceleration with backtracking; however, the back-tracking procedure in~\cite[Section 4.3]{Titan2020} for the Nesterov-type acceleration would be quite expensive since the computation of $f_i^k(\bar x_i^{k})$ and $\nabla  f_i^k(\bar x_i^{k})$ would be required in the back-tracking process. 
Furthermore, there are no experiments in \cite{ahookhosh2020_inertial} and \cite{Titan2020} to justify the efficacy of the inertial versions for Bregman BPG methods. 

BPG and Bregman BPG methods belong to the block majorization minimization framework \cite{Titan2020,Razaviyayn2013} that updates one block $x_i$ of $x$ by minimizing
a block surrogate function of the objective function. 
In~\cite[Section 6.2]{Titan2020}, 
the matrix completion problem (MCP), which also has the form of Problem~\eqref{model}, illustrates the advantage of using suitable block surrogate functions and the efficacy of TITAN, the inertial block majorization minimization framework proposed in~\cite{Titan2020}. 
Specifically, each subproblem that minimizes the block surrogate function used in TITAN has a closed-form solution while each proximal gradient step in the BPG method does not. 
Furthermore, TITAN outperforms BPG for the MCP. 
This motivates us to design an algorithm that allows using surrogate functions of $g_i$ to replace $g_i$, in contrast to the current Bregman BPG methods which do not change $g_i$ in the sub-problems; see \eqref{Bregman-surrogate}.

\subsection{Contribution and organization of the paper} 

After having introduced some preliminary notions of the Bregman distances and block relative smooth functions in Section~\ref{sec:Bregman}, 
 we propose in Section~\ref{sec:multi_iMM} a block alternating \textbf{B}regman \textbf{M}ajorization \textbf{M}inimization framework with \textbf{E}xtrapolation (BMME) that uses Nesterov-type acceleration to solve Problem~\eqref{model} in which $f$ is assumed to be a block relative smooth function with respect to $(\varphi_1,\ldots,\varphi_m)$; see Definition~\ref{def:block_relative_smooth}. 
 This means that the gradient and the Bregman divergence  in~\eqref{Bregman-surrogate} are replaced with $\nabla f_i^k(\bar x_i^k)$ and $D_{\varphi^k_i}(x_i,\bar x_i^k)$, respectively; see Algorithm~\ref{algo:iMM_multiblock}.  We use a line-search strategy proposed in~\cite{Mukkamala2020} to determine the extrapolation point $\bar x_i^k$. We remark that the inertial Bregman BPG method proposed in~\cite{Mukkamala2020}, named CoCaIn, is for solving the CSOP with $m=1$ while BMME is for solving~\eqref{model} with multiple blocks. Furthermore, CoCaIn requires its subproblem, which is Problem~\eqref{Bregman-surrogate} with the gradient and the Bregman divergence being replaced by $\nabla f_i^k(\bar x_i^k)$ and $D_{\varphi^k_i}(x_i,\bar x_i^k)$ (note that we can omit the index $i$ as $m=1$ for CoCaIn), to be solved exactly (in other words, to have a closed-form solution). This requirement would be restrictive in applications where the nonsmooth part $g_i$ is nonconvex and does not allow a closed form solution for the subproblem. In contrast, BMME employs surrogate functions for $g_i$, $i \in [m]$, that may lead to closed-form solutions for its subproblem, see an example in \Cref{appendix_mcp}.  We note that CoCaIn requires $g_i(\cdot) + \alpha/2 \|\cdot\|^2$ to be convex for some constant $\alpha\geq 0$ (see \cite[Assumption C]{Mukkamala2020}) while BMME requires $x_i\mapsto u_i(x_i,y_i)$ to be convex for any $y_i\in \mathcal X_i$, where $u_i(\cdot,\cdot)$ is a surrogate function of $g_i$ (see Definition~\ref{def:surrogate}). And as such, our analysis may allow a larger class of $g_i$ than CoCaIn since $u_i$ with $u_i(x_i,y_i)= g_i(x_i) + \alpha/2 \|x_i-y_i\|^2$ is a surrogate function of $g_i$. It is important noting that the convexity assumption for the surrogate of $g_i$ allows BMME to use stepsizes that only depend on the relative smooth constants of $f$. In contrast, CoCaIn needs to start with an initial relative smooth constant that linearly depends on the value of $\alpha$ that makes $g_i(\cdot) + \alpha/2 \|\cdot\|^2$ convex. This initial relative smooth constant could be very large and lead to a very small stepsizes which results in a slow convergence. To illustrate this fact, we provide an experiment in \Cref{appendix_mcp} to compare the performance of BMME and CoCaIn on the matrix completion problem. 

In Section~\ref{sec:convergence}, we prove subsequential convergence of the sequence generated by BMME to a first-order stationary point of~\eqref{model} under mild assumptions, 
and prove the global convergence under stronger conditions. Furthermore, the analysis in~\cite{Mukkamala2020} does not consider the subsequential convergence but only proves the global convergence for $F$ satisfying the  Kurdyka-{\L}ojasiewicz (KL) property \cite{Kurdyka1998},  and under the assumption that the domains of the kernel functions are the full space. 
In our convergence analysis, we assume that every limit point $x^*$ of the generated sequence by BMME satisfying the condition that $x_i^*$ lies in the interior of the domain of $x_i\mapsto \varphi_i(x_1^*,\ldots,x_{i-1}^*,x_i,x_{i+1}^*,x_m^*)$ for $i \in [m]$. This assumption is naturally satisfied when the $\varphi_i$'s have a full domain or $\mathcal X\subset {\rm int \, dom\,} \varphi_i$. For example, the feasible set $\mathcal X=\{x: x_i\in\mathbb R^{d_i}, x_i\geq \varepsilon>0\}$ (that is, each component of $x_i$ is lower bounded by a positive constant $\varepsilon$) and the Burg entropy $\varphi_i(x)=-\sum_{j=1}^{d_i}\log x_{ij}$ satisfy our assumption; see for example the perturbed Kullback-Leibler nonnegative matrix factorization in~\cite{HienNicolas_KLNMF}. We then prove subsequential convergence without the assumption that $F$ satisfies the KL property, and prove global convergence with this assumption.  

In Section~\ref{sec:experiment}, we apply BMME to solve a penalized orthogonal nonnegative matrix factorization problem (ONMF). 
We conclude the paper in Section~\ref{sec:conclusion}.

\section{Preliminaries: Bregman distances and relative smoothness} 
\label{sec:Bregman}

In this section, we present preliminaries of Bregman distances and relative smoothness. 
We adopt~\cite[Definition 2.1]{Bolte2018} to define a kernel generating distance which, for simplicity, we refer to as  ``kernel function''. 


\begin{definition}[Kernel generating distance] Let $C$ be a nonempty, convex and open subset of $\mathbb E_i$. A function $\psi:\mathbb E_i \to \bar{\mathbb R}:=(-\infty, +\infty]$ associated with $C$ is called a kernel generating distance if it satisfies the following:
  \begin{description}
\item[(i)] $\psi$ is proper, lower semicontinuous and convex with ${\rm dom\,} \psi \subset \bar C$, where $\bar C$ is the closure of $C$, and ${\rm dom\,} \partial \psi = C$.
\item[(ii)] $\psi$ is continuously differentiable on ${\rm int\, dom\,} \psi \equiv C$.
\end{description}
Let us denote the class of kernel generating distances by $\mathcal G(C)$.
\end{definition}

\begin{definition}
\label{def:Bregman_divergence}
Given $\psi\in \mathcal G(C)$, we define $D_\psi: {\rm dom\,} \psi \times {\rm int\, dom\,} \psi \to \mathbb R_+$ as the \emph{Bregman divergence} associated with the kernel function $\psi$ as follows	
\begin{align*}
D_\psi(x_i,y_i):= \psi(x_i) - \psi(y_i) -\iprod{\nabla \psi(y_i)}{x_i-y_i}. 
\end{align*}
\end{definition}

\begin{definition}[$(L,l)$-relative smooth function]  
\label{def:relative_smooth}
Given $\psi\in \mathcal G(C)$, let $\phi: \mathbb E_i \to (-\infty, +\infty] $ be a proper and lower semicontinuous function with ${\rm dom\,} \psi \subset {\rm dom\,}\phi$, which is continuously differentiable on $C={\rm int\, dom\,} \psi $. 
 We say $\phi$ is 
$(L,l)$-relative smooth to $\psi$ if there exist $L > 0$ and $l\geq 0$ such that for any $x_i, y_i\in C$, 
\begin{equation}
\phi(x_i) - \phi(y_i) -\iprod{\nabla \phi(y_i)}{x_i-y_i} \leq LD_\psi(x_i, y_i), 
\end{equation}
and 
\begin{equation}
\label{eq:lD}
 -l D_\psi(x_i, y_i) 
 \leq \phi(x_i) 
 - \phi(y_i) -\iprod{\nabla \phi(y_i)}{x_i-y_i}.
\end{equation}
\end{definition}
Whenever $\phi$ is convex, we may take $l=0$ and Definition~\ref{def:relative_smooth} recovers~\cite[Definition 1.1]{lu2018relatively}. 
In the case $l=L$,  Definition~\ref{def:relative_smooth} recovers~\cite[Definition 2.2]{Bolte2018}.

 Given a function $f: \mathbb E \to (-\infty, +\infty]$, for each $i\in [m]$ and any fixed $y_j$ for $j\ne i$, we define a block function $f(\cdot,y_{\ne i}):\mathbb E_i \to (-\infty, +\infty]$ by 
\begin{equation}
   x_i\mapsto f(x_i,y_{\ne i}):= f(y_1,\ldots,y_{i-1},x_i,y_{i+1},\ldots,y_m).
\end{equation}

\begin{definition}[Block relative smooth function]  
\label{def:block_relative_smooth}
We say that 
$f:  \mathbb E \to (-\infty, +\infty]$ is a block relative smooth function with respect to $(\varphi_1,\ldots,\varphi_m)$, where $f$ is continuously differentiable on $\mathcal C={\rm int\, dom\,} \varphi_1=\cdots= {\rm int\, dom\,} \varphi_m$ and ${\rm dom\,} \varphi_1 = \ldots = {\rm dom\,} \varphi_m \subset {\rm dom\,} f$, if, for any $y \in {\rm dom\,} \varphi_i$ we have 
$\varphi_i(\cdot,y_{\ne i})$ is a kernel generating function 
and the function $f(\cdot,y_{\ne i})$
is a $(L^{y_{\ne i}}_i,l^{y_{\ne i}}_i)$-relative smooth to $\varphi_i(\cdot,y_{\ne i})$, where $(L^{y_{\ne i}}_i,l^{y_{\ne i}}_i)$  may depend on $y_j$, $j\ne i$.
\end{definition}

Throughout this paper we will assume the following. 
\begin{assumption}\label{assumption:f}
We suppose $\mathcal C= \inter\, \dom\, \varphi_1=\cdots= \inter\, \dom\, \varphi_m$, $\dom\, \varphi_1 = \ldots = \dom\, \varphi_m \subset \dom\, f$, $\mathcal X \cap \dom\, \varphi_1 \ne\emptyset$,  the function $f$ in~\eqref{model} is a block relative smooth function with respect to $(\varphi_1,...,\varphi_m)$.
\end{assumption}

Let us make an important remark regarding   Definition~\ref{def:block_relative_smooth}. 


 \paragraph{Flexibility of Definition~\ref{def:block_relative_smooth}.} \label{paragraph:flex}
 Let us consider the notion of  block relative smoothness in   Definition~\ref{def:block_relative_smooth} without $l^{y_{\ne i}}_i$, that is, the condition~\eqref{eq:lD} is discarded. Similar definitions have been considered in~\cite{ahookhosh2019multi} and~\cite{ahookhosh2020_inertial}. 
In~\cite{ahookhosh2019multi}, the authors first define a multi-block kernel function $\psi: \mathbb E_1 \times \ldots \mathbb \times E_m \to \bar{\mathbb R}$ \cite[Definition 3.1]{ahookhosh2019multi}, and then define multi-block relative smoothness of $f$ with respect to this multi-block kernel function with the relative smooth constants  $(L_1,\ldots,L_m)$ \cite[Definition 3.4]{ahookhosh2019multi}.  In~\cite{ahookhosh2020_inertial}, the authors define the block relative smoothness of $f$ with respect to $(\psi_1,\ldots,\psi_m)$, where $\psi_i: \mathbb E_1 \times \ldots \mathbb E_m \to \bar{\mathbb R}$ is an $i$-th block kernel function \cite[Definition 2.1]{ahookhosh2020_inertial}, with the relative smooth constants $(L_1,\ldots,L_m)$ \cite[Definition 2.2]{ahookhosh2020_inertial}. 
It is crucial to note that $L_1,\ldots,L_m$ in these definitions are \emph{constants} and the stepsize used in the algorithms proposed in~\cite{ahookhosh2019multi} and~\cite{ahookhosh2020_inertial} to update each block $i$ is \emph{strictly less} than $1/L_i$. 
In contrast, our Definition~\ref{def:block_relative_smooth} allows the block $i$ relative smooth constant to change in the iterative process, that is, $L_i^{y_{\ne i}}$ and $l_i^{y_{\ne i}}$ are not constants but vary with respect to the values of the other blocks $y_j$ for $j\ne i$. 
This flexibility in Definition~\ref{def:block_relative_smooth} will lead to more flexible choices for the block kernel functions, and also leads to variable step-sizes in designing Bregman BPG algorithms for solving the multi-block CSOP. In fact, as we will see in Algorithm~\ref{algo:iMM_multiblock}, the stepsize to update block $i$ is $1/L_i^k$ which changes in the course of the iterative process. 
We will illustrate this crucial advantage of Algorithm~\ref{algo:iMM_multiblock} for solving the penalized ONMF problem in Section~\ref{sec:experiment}. Furthermore, it is important noting that if $f$ satisfies \cite[Definition 2.1]{ahookhosh2020_inertial} or \cite[Definition 3.4]{ahookhosh2019multi}, then $f$ satisfies  Definition~\ref{def:block_relative_smooth} with the corresponding $L_i^{y_{\ne i}}$ being the constant $L_i$ for all $i \in [m]$. 
However, the converse does not hold; see an example in Section~\ref{kernelONMF}. Hence Algorithm~\ref{algo:iMM_multiblock} applies to a broader class of problems, while allowing a more flexible choice of the step-sizes which will lead to faster convergence; see Section~\ref{exp:synthreal}.

\section{Block Alternating Majorization Minimization with Extrapolation}
\label{sec:multi_iMM}

Before introducing BMME, let us first recall the definition of a surrogate function as follows.
\begin{definition}
\label{def:surrogate}
A function $u_i:\mathcal X_i \times \mathcal X_i \to \bar{\mathbb{ R}}  $ is called a surrogate function of $g_i:\mathcal X_i\to \bar{\mathbb {R}}$ 
if the following conditions are satisfied:
\begin{description}
    \item[(a)] $u_i(y_i,y_i) = g_i(y_i)$ for all $y_i\in \mathcal X_i$, 
    \item[(b)] $u_i(x_i,y_i) \geq g_i(x_i)$ for all $x_i,y_i\in\mathcal X_i$.
    \end{description}
The approximation error is defined as $h_i(x_i,y_i):=u_i(x_i,y_i) - g_i(x_i)$.
\end{definition}
For example, $u_i(x_i,y_i) = g_i(x_i) + \frac{\alpha}{2} \|x_i-y_i\|^2$, where $\alpha$ is a nonnegative constant, is always a surrogate function of $g_i$. In this case, $h_i(x_i,y_i)=\frac{\alpha}{2} \|x_i-y_i\|^2$. We refer the readers to \cite{Titan2020,Razaviyayn2013,Mairal_ICML13} for more examples.

Denote $x^{k,0} = x^k$ and 
$$x^{k,i}=(x^{k+1}_1, \ldots,x^{k+1}_{i},x^{k}_{i+1},\ldots,x^{k}_m).
$$
For notation succinctness, we denote $\varphi_i^k(\cdot) := \varphi_i(\cdot,x^{k,i-1}_{\neq i})$,  $L_i^k: = L_i^{x^{k,i-1}_{\neq i}}$, and $l_i^k: = l_i^{x^{k,i-1}_{\neq i}}$.

We can now introduce our BMME algorithm; see Algorithm \ref{algo:iMM_multiblock}. In particular, at iteration $k$, for each block $i$, BMME chooses a surrogate function $u_i$ of $g_i$ such that $x_i\mapsto u_i(x_i,y_i)$ is convex (as mentioned in the introduction, this condition is satisfied by the requirement that $g_i(\cdot) + \alpha/2 \|\cdot\|^2$ is convex for some constant $\alpha\geq 0$ of \cite{Mukkamala2020}) and computes an extrapolated point $\bar{x}^k_i = x_i^k + \beta^k_i(x^k_i - x^{k-1}_i) \in \inter\,\dom\, \varphi_i^k$, where $\beta^k_i$ is an extrapolation parameter satisfying
\begin{equation*}
     D_{\varphi^k_i}(x^k_i,\bar{x}^k_i) \leq \frac{\delta_i L^{k-1}_i}{L^k_i + l^k_i} D_{\varphi^{k-1}_i}(x^{k-1}_i,x^k_i),
\end{equation*}
for some $\delta_i\in(0,1)$. BMME then updates $x^{k,i}$ by
   \begin{equation*}
       \begin{split}
           x^{k,i}_i&\in\argmin_{x_i\in\mathcal X_i}\biggl\{L^k_iD_{\varphi^k_i}(x_i,\bar{x}^k_i) + \biggl\langle \nabla_i f(\bar{x}^k_i,x^{k,i-1}_{\neq i}),x_i\biggr\rangle + u_i(x_i,x_i^k) \biggr\}\\
           &=\argmin_{x_i\in\mathbb E_i }\biggl\{L^k_iD_{\varphi^k_i}(x_i,\bar{x}^k_i) + \biggl\langle \nabla_i f(\bar{x}^k_i,x^{k,i-1}_{\neq i}),x_i\biggr\rangle + \big(u_i(x_i,x_i^k) + I_{\mathcal X_i} (x_i)\big)\biggr\},
       \end{split}
   \end{equation*}
where $I_{\mathcal X_i}$ is the indicator function of $\mathcal X_i$.
We make the following standard assumption for $\{x^{k}\}$, see for example~\cite[Assumption C]{Bolte2018}. Note that the initial points $x^{-1}$ and $x^0$ are chosen in the interior domain of $\varphi_i$, $i\in [m]$. 
 \begin{assumption}
 \label{assump:interior}
We have $x^{k}\in {\rm int \, dom\,} \varphi_i$, $i\in [m]$.
 \end{assumption}
Assumption~\ref{assump:interior} is naturally satisfied when the domain of $\varphi_i$ is full space.  See~\cite[Lemma 3.1]{Bolte2018} and ~\cite[Remark 3.1]{Bolte2018} for a sufficient condition that ensures~\eqref{eq:iMM_update} to produce $x_i^{k+1}\in \inter\,\dom\, \varphi_i^k$, which implies that Assumption~\ref{assump:interior} holds.

\begin{algorithm}[ht!]
\caption{BMME}
\begin{algorithmic}[1]
\label{algo:iMM_multiblock}
\STATE Choose $x^{-1},x^0\in {\rm int \, dom\,} \varphi_i$, $\delta_i\in(0,1)$, and set $k = 0$. Let $u_i$ be a surrogate function of $g_i$ such that $x_i\mapsto u_i(x_i,y_i)$ is convex for any $y_i \in \mathcal X_i$.
   \medskip
   \REPEAT 
   \FOR{ $i = 1,...,m$}
   \STATE  Compute an extrapolation parameter $\beta^k_i$ such that
   \begin{equation}\label{eq:beta}
        D_{\varphi^k_i}(x^k_i,\bar{x}^k_i) \leq \frac{\delta_i L^{k-1}_i}{L^k_i + l^k_i} D_{\varphi^{k-1}_i}(x^{k-1}_i,x^k_i),
   \end{equation}
   where $\bar{x}^k_i = x_i^k + \beta^k_i(x^k_i - x^{k-1}_i) \in {\rm int \, dom\,} \varphi_i^k$.
   \STATE Update $x^{k,i}$ by
   \begin{equation}
   \label{eq:iMM_update}
       \begin{split}
           x^{k+1}_i\in\argmin_{x_i\in\mathcal X_i}\biggl\{L^k_iD_{\varphi^k_i}(x_i,\bar{x}^k_i) + \biggl\langle \nabla_i f(\bar{x}^k_i,x^{k,i-1}_{\neq i}),x_i\biggr\rangle + u_i(x_i,x_i^k) \biggr\}.
       \end{split}
   \end{equation}
   \ENDFOR
   \STATE $k\leftarrow k+ 1$.
    \UNTIL{Stopping criterion.}
\end{algorithmic}
\end{algorithm}

\paragraph{Choice of the extrapolation parameters.}  

BMME needs to adequately choose the extrapolation parameters $\beta^k_i$'s. 
Let us mention some special choices. 

When $x_i \mapsto f(x_i,y_{\ne i})$ admits an $L^{y_{\ne i}}_i$-Lipschitz continuous gradient, that is, $\varphi(\cdot,y_{\ne i}) = \frac{1}{2}\|\cdot\|^2$, 
Condition~\eqref{eq:beta} becomes
\begin{equation*}
    (\beta^k_i)^2\|x^k_i-x^{k-1}_i\|^2 \leq \frac{\delta_i L_i^{k-1}}{L_i^k+l_i^k}\|x^k_i-x^{k-1}_i\|^2.
\end{equation*}
Therefore, we can choose any $\beta^k_i$ such that $\beta^k_i \leq \sqrt{\frac{\delta_i L_i^{k-1}}{L_i^k+l_i^k}}$.
Moreover, if $f(\cdot,y_{\ne i})$ is convex, we can take $l_i^k=0$ and hence we can choose any $\beta^k_i \leq \sqrt{\frac{\delta_i L_i^{k-1}}{L_i^k}}$.  

In general, \cite[Lemma 4.2]{Mukkamala2020} showed that if the symmetry coefficient of $\varphi_i^k$, which is defined by $\inf \Big \{\frac{D_{\varphi_i^k}(x_i,y_i)}{D_{\varphi_i^k}(y_i,x_i)}: x_i,y_i\in \inter\,\dom\,\varphi_i^k \Big\}$, is positive then, for a given 
\[
\kappa = \frac{\delta_i L^{k-1}_i}{L^k_i + l^k_i}D_{\varphi^{k-1}_i}(x^{k-1}_i,x^k_i)/D_{\varphi^k_i}(x^{k-1}_i,x^k_i)>0, 
\] 
there always exists $\gamma^k_i>0$ such that the following condition is satisfied for all $\beta^k_i\in[0,\gamma_i^k]$
\begin{equation}
    D_{\varphi^k_i}(x^k_i,\bar{x}^k_i) \leq \kappa D_{\varphi^k_i}(x^{k-1}_i,x^k_i),
\end{equation}
which is equivalent to the condition \eqref{eq:beta}. Therefore, $\beta_i^k$ can be determined by a line search as follows. At each iteration, we initialize $\beta_i^k = \frac{\nu_i^{k-1}-1}{\nu_i^k}$, where $\nu_i^k = \frac{1}{2}\left(1+\sqrt{1+4(\nu_i^{k-1})^2}\right)$ and $\nu_i^0=1$ as in Nesterov~\cite{Nesterov1983}, 
and, while the inequality~\eqref{eq:beta} does not hold, we decrease $\beta_i^k$ by a constant factor $\eta_i \in(0,1)$, 
that is,  $\beta_i^k \leftarrow \beta_i^k\eta_i$. 

Before proceeding to the convergence analysis, we make an important remark: the relative smoothness constants $L_i^k$ and $l_i^k$ in Algorithm~\ref{algo:iMM_multiblock} are assumed to be known at the moment of updating $x^k_i$. 
In case these values are unknown (or their known lower/upper bounds are too loose), we can employ the convex-concave backtracking strategy as in the algorithm CoCaIn BPG proposed in \cite[Section 3.1]{Mukkamala2020} to determine these values as well as the extrapolation parameter $\beta_i^k$. 
The upcoming convergence analysis would be similar in that case. 
In \Cref{appendix_mcp}, we consider the matrix completion problem (MCP) which has the form of Problem~\eqref{model} with $m=1$, and we illustrate BMME with the backtracking strategy on this problem when the values of the relative smooth constants are too small/large. 
The experiment presented in \Cref{appendix_mcp} on the MCP shows the  backtracking strategy significantly  improves the performance  of BMME, and also outperforms CoCaIn BPG. 
However, to simplify the presentation, we will only consider the convergence analysis of BMME for solving the multi-block Problem~\eqref{model} when the relative smooth constants are assumed to be known.  

\section{Convergence analysis}
\label{sec:convergence}
In this section, we study the subsequential convergence as well as the global convergence of BMME. For our upcoming analysis, we need the following first-order optimality condition of \eqref{model}: $x^*$ is a first-order stationary point of \eqref{model} if
\begin{equation}
\label{eq:opt_cond}
\iprod{p(x^*)}{x-x^*}\geq 0 \, \text{ for all } \, x\in \mathcal X  \text{ and for some } \, p(x^*) \in \partial F(x^*). 
\end{equation}
As $f$ is continuously differentiable folowing Assumption \ref{assumption:f}, $\partial F(x^*)=\lrbrace{\partial_{x_1} F(x^*)}\times\ldots\times\lrbrace{\partial_{x_m} F(x^*)}$, where $\partial F(x^*)$ is the limiting-subdifferential of $F$ at $x^*$, see Definition \ref{def:dd} in \Cref{sec:prelnnopt}. Therefore, 
\eqref{eq:opt_cond} is equivalent to 
\begin{equation}
\label{eq:opt_cond2}
\iprod{p_i(x^*)}{x_i-x_i^*}\geq 0 \text{ for all } x_i\in \mathcal X_i, \text{for some } p_i(x^*) \in \partial_{x_i} F(x^*) \text{ for } i\in [m]. 
\end{equation} 
 If $x^*$ is in the interior of $\mathcal X$ or $\mathcal X_i=\mathbb E_i$ then \eqref{eq:opt_cond} reduces to the condition $0\in \partial F(x^*)$, that is, $x^*$ is a critical point of $F$.

\subsection{Subsequential convergence}

The following theorem presents the subsequential convergence of the sequence generated by Algorithm \ref{algo:iMM_multiblock} under an additional assumption on the surrogate function of $g_i$.

\begin{assumption}
\label{assump:surrogate_assum} 
\begin{itemize}
\item[(A)] For $i\in [m]$, the surrogate function $u_i(\cdot,\cdot)$ of $g_i$ used in \eqref{eq:iMM_update} in Algorithm~\ref{algo:iMM_multiblock} satisfies that $x_i\mapsto u_i(x_i,y_i)$ is convex. 

\item[(B)]  
For $i\in [m]$, $u_i(x_i,y_i)$ is continuous in $y_i$ and lower semicontinuous in $x_i$.

\item[(C)] For $i\in [m]$, given $y_i\in \mathcal X_i$, there exists a function $x_i\mapsto \bar h_i(x_i,y_i)$ such that $ \bar h_i(\cdot,y_i)$ is continuously differentiable at $y_i$ and $\nabla_{x_i} \bar h_i(y_i,y_i)=0$, and the approximation error $x_i\mapsto h_i(x_i,y_i): = u_i(x_i,y_i)- g_i(x_i)$  satisfies  
\begin{equation}
\label{lemma:h_property} 
h_i(x_i,y_i) \leq \bar h_i(x_i,y_i) \;  \text{ for all } \;  x_i \in \mathcal X_i.
\end{equation}
\end{itemize}
\end{assumption}  
For example, if $g_i(\cdot) + \frac{\alpha}{2} \|\cdot\|^2$ is convex for some constant $\alpha\geq 0$ then the surrogate $u_i(x_i,y_i)=g_i(x_i) + \frac{\alpha}{2} \|x_i-y_i\|^2$ satisfies Assumption~\ref{assump:surrogate_assum} with $\bar h_i(x_i,y_i) = h_i(x_i,y_i)=\frac{\alpha}{2} \|x_i-y_i\|^2$. More examples can be found in~\cite{Titan2020}.

\begin{theorem}
\label{theorem:subsconvergence} 
Let $\{x^k\}$ be the sequence generated by Algorithm~\ref{algo:iMM_multiblock}, and let Assumptions \ref{assumption:f}-\ref{assump:surrogate_assum} be satisfied. The following statements hold.
\begin{itemize}
    \item[A)] For $k = 0,1,...$ we have \begin{equation}\label{eq:descent}
       F(x^{k,i}) \leq F(x^{k,i-1}) 
        - L^k_iD_{\varphi^k_i}(x^k_i,x_i^{k+1}) + \delta_iL^{k-1}_i D_{\varphi^{k-1}_i}(x_i^{k-1},x^k_i).
    \end{equation}
    
    \item[B)] If there exists a positive number $\underline{L}$ such that $\min_{k,i}L^k_i \geq \underline{L}$\footnote{This is a standard assumption in analysing inertial block coordinate methods, see e.g., \cite[Assumption 2]{Xu2013},  \cite[Assumption 2]{Xu2017},  \cite[Assumption 3]{Hien_ICML2020} for similar assumptions when $f$ is a block Lipschitz smooth function.}, we have \begin{equation}
        \label{eq:sumD}
        \sum_{k=0}^{+\infty}\sum_{i=1}^mD_{\varphi^k_i}(x_i^k,x^{k+1}_i) < +\infty.
    \end{equation}
    \item[C)] Assume that $\nabla_{x_i}\varphi_i(\cdot,y_{\ne i})$ for $i\in[m]$ is continuous in $y_{\ne i}$, $\{L_i^k\}$ for $i\in[m]$ and $\{x^k\}$ are bounded\footnote{It follows from Inequality~\eqref{eq:descent} that if $F$ has bounded level sets then $\{x^k\}$ is bounded.}, and $\{\rho^k_i\}$ for $i\in[m]$ is bounded from below by $\rho>0$, where $\rho^k_i$ is the modulus of the strong convexity of $\varphi_i^k$. If $x^*$ is a limit point of $\{x^k\}$ and\footnote{As mentioned in the introduction, this condition is satisfied when $\varphi_i$ has full domain or $\mathcal X \subset\inter\,\dom\, \varphi_i$} $x^*_i\in {\rm int\, dom\,} \varphi_i(\cdot,x_{\ne i}^*)$, then $x^*$ is a first-order stationary point of Problem~\eqref{model}.
\end{itemize}
\end{theorem}

\begin{proof}
A) Since $x^{k+1}_i$ is a solution to the convex problem \eqref{eq:iMM_update}, it follows from \cite[Theorem 3.1.23]{Nesterov2018} that for every $x_i\in\mathcal X_i$ we have
\begin{equation}\label{eq:1123}
    \begin{split}
        \biggl\langle L^k_i(\nabla\varphi^k_i(x^{k+1}_i) - \nabla\varphi^k_i(\bar{x}_i^k))
        + \nabla f(\bar{x}^k_i,x^{k,i-1}_{\neq i}), x_i- x_i^{k+1}\biggr\rangle\\ + u_i(x_i,x^k_i)
        \geq u_i(x_i^{k+1},x^k_i).
    \end{split}
\end{equation}
By choosing $x_i = x_i^{k}$, we obtain  
\begin{equation}\label{eq:1122}
    \begin{split}
        \biggl\langle L^k_i(\nabla\varphi^k_i(x^{k+1}_i) - \nabla\varphi^k_i(\bar{x}_i^k))
        + \nabla f(\bar{x}^k_i,x^{k,i-1}_{\neq i}), x_i^{k} - x_i^{k+1}\biggr\rangle\\ + u_i(x_i^{k},x^k_i)
        \geq u_i(x_i^{k+1},x^k_i).
    \end{split}
\end{equation}
Substituting $u_i(x_i^{k+1},x^k_i) \geq g_i(x^{k+1}_i)$ and $u_i(x_i^{k},x^k_i) = g_i(x^k_i)$ into this inequality gives 
\begin{equation*}
    \begin{split}
        \biggl\langle L^k_i(\nabla\varphi^k_i(x^{k+1}_i) - \nabla\varphi^k_i(\bar{x}_i^k))
        + \nabla f(\bar{x}^k_i,x^{k,i-1}_{\neq i}), x_i^{k} - x_i^{k+1}\biggr\rangle + g_i(x^k_i) 
        \geq g_i(x^{k+1}_i).
    \end{split}
\end{equation*}
On the other hand, since $f$ is a block relative smooth function, we have
\begin{equation*}
    \begin{split}
        f(x^{k,i}) \leq f(\bar{x}^k_i,x^{k,i-1}_{\neq i}) + \langle \nabla f(\bar{x}^k_i,x^{k,i-1}_{\neq i}), x_i^{k+1} - \bar{x}^k_i\rangle + L^k_iD_{\varphi^k_i}(x_i^{k+1},\bar{x}^k_i),
    \end{split}
\end{equation*}
and
\begin{equation*}
    \begin{split}
        f(\bar{x}^k_i,x^{k,i-1}_{\neq i}) + \langle \nabla f(\bar{x}^k_i,x^{k,i-1}_{\neq i}), x_i^{k} - \bar{x}^k_i\rangle \leq  f(x^{k,i-1}) + l^k_iD_{\varphi^k_i}(x_i^{k},\bar{x}^k_i),
    \end{split}
\end{equation*}
By summing the three inequalities above, we obtain
\begin{equation}\label{eq:1121}
\begin{split}
    F(x^{k,i}) \leq F(x^{k,i-1}) + L^k_i\biggl\langle \nabla\varphi^k_i(x^{k+1}_i) - \nabla\varphi^k_i(\bar{x}_i^k)
        , x_i^{k}- x_i^{k+1}\biggr\rangle\\
        + L^k_iD_{\varphi^k_i}(x_i^{k+1},\bar{x}^k_i) + l^k_iD_{\varphi^k_i}(x_i^{k},\bar{x}^k_i).
\end{split}
\end{equation}
Moreover, we have
\begin{equation}
    \begin{split}
        L^k_i\biggl\langle \nabla\varphi^k_i(x^{k+1}_i) - \nabla\varphi^k_i(\bar{x}_i^k), x_i^{k}- x_i^{k+1}\biggr\rangle + L^k_iD_{\varphi^k_i}(x_i^{k+1},\bar{x}^k_i)\\
       = - L^k_iD_{\varphi^k_i}(x^{k}_i,x_i^{k+1}) + L^k_iD_{\varphi^k_i}(x_i^{k},\bar{x}^k_i),
    \end{split}
\end{equation}
Therefore, we obtain
\begin{equation}
\begin{split}
    F(x^{k,i}) \leq F(x^{k,i-1})  
        - L^k_iD_{\varphi^k_i}(x^{k}_i,x_i^{k+1}) + \biggl( L^k_i + l^k_i\biggr)D_{\varphi^k_i}(x_i^{k},\bar{x}^k_i)\\
        \leq F(x^{k,i-1}) 
        - L^k_iD_{\varphi^k_i}(x^k_i,x_i^{k+1}) + \delta_iL^{k-1}_i D_{\varphi^{k-1}_i}(x_i^{k-1},x^k_i),
\end{split}
\end{equation}
where the second inequality holds by \eqref{eq:beta}. This implies A). 

B) Summing \eqref{eq:descent} over $i = 1,..,m$ gives 
\begin{equation}\label{eq:1125}
    F(x^{k+1}) \leq F(x^k) - \sum_{i=1}^mL^k_iD_{\varphi^k_i}(x^k_i,x_i^{k+1}) + \sum_{i=1}^m\delta_iL^{k-1}_i D_{\varphi^{k-1}_i}(x_i^{k-1},x^k_i).
\end{equation}
By summing up this inequality from $k = 0$ to $K - 1$, we obtain
\begin{equation*}
\begin{split}
    F(x^K) + \sum_{i=1}^m\delta_iL^{K-1}_i D_{\varphi^{K-1}_i}(x_i^{K-1},x^K_i) + \sum_{k=0}^{K-1}\sum_{i=1}^m(1-\delta_i)L^k_iD_{\varphi^k_i}(x^k_i,x_i^{k+1})\\
    \leq F(x^0) + \sum_{i=1}^m\delta L^{-1}_iD_{\varphi^{-1}_i}(x_i^{-1},x^0_i),
\end{split}
\end{equation*}
which gives the result.

C) Let $x^*$ be a limit point of $\{x^k\}$. There exists a subsequence $\{x^{k_n}\}$ of $\{x^k\}$ converging to $x^*$. We have $D_{\varphi_i^k}(x_i^k,x_i^{k+1}) \geq \frac{\rho_i^k}{2} \|x_i^k-x_i^{k+1}\|^2$ since $\varphi_i^k$ is $\rho_i^k$-strongly convex. Together with the assumption $\rho_i^k\geq \rho>0$ and \eqref{eq:sumD} we have $\|x^k-x^{k+1}\|$ converges to 0. Hence, $\{x^{k_n+1}\}$ and $\{x^{k_n-1}\}$  converge to $x^*$. Substituting $x_i = x^*_i$ and $k = k_n$ into \eqref{eq:1123} gives 
\begin{equation}
    \begin{split}
        \biggl\langle L^{k_n}_i(\nabla\varphi^{k_n}_i(x^{k_n+1}_i) - \nabla\varphi^{k_n}_i(\bar{x}_i^{k_n}))
        + \nabla f(\bar{x}^{k_n}_i,x^{k_n,i-1}_{\neq i}), x_i^*- x_i^{k_n+1}\biggr\rangle\\ + u_i(x_i^*,x^{k_n}_i) 
        \geq u_i(x_i^{k_n+1},x^{k_n}_i).
    \end{split}
\end{equation}
By taking $n\to+\infty$,
we have
\begin{equation}
    \limsup_{n\to+\infty}u_i(x_i^{k_n+1},x^{k_n}_i) \leq g_i(x_i^*),
\end{equation}
where we have used the boundedness of $L^{k_n}_i$, the continuity of $u_i(x_i,\cdot)$, $\nabla \varphi_i$, and $\nabla f$, $x^*_i\in {\rm int\, dom\,} \varphi_i(\cdot,x_{\ne i}^*)$, and the fact that $x^{k_n+1} \to x^*$ as $n\to+\infty$. From this and the lower semi-continuity of $u_i(x_i,y_i)$, we have
\begin{equation}
    \lim_{n\to+\infty}u_i(x_i^{k_n+1},x^{k_n}_i) = g_i(x_i^*).
\end{equation}
Choosing $k = k_n$ in \eqref{eq:1123} and letting $n\to+\infty$ implies that, for all $x_i\in \mathcal X_i$, 
\begin{equation}
\label{temp1}
    g_i(x^*_i) \leq u_i(x_i,x^*_i)+ \biggl\langle  \nabla f(x^*_i,x^*_{\neq i}), x_i- x_i^*\biggr\rangle.
\end{equation}
Note that $u_i(x_i,x^*_i)=g_i(x_i) + h_i(x_i, x_i^*)$ and $f(\cdot,x^*_{\ne i})$ is $(L_i^*,l_i^*)$-relative smooth to $\varphi_i^*(\cdot)=\varphi_i(\cdot,x^*_{\ne i})$ for some constant $L_i^*,l_i^*$. Therefore, from~\eqref{temp1}  we have for all $x_i\in\mathcal X_i$ that 
\begin{equation}
    \begin{split}
        F(x^*) \leq F(x_i,x^*_{\neq i}) + l_i^*D_{\varphi^*_i}(x_i,x^*_i) + h_i(x_i,x^*_i)\\
        \leq F(x_i,x^*_{\neq i}) + l_i^*D_{\varphi^*_i}(x_i,x^*_i) + \bar{h}_i(x_i,x^*_i),
    \end{split}
\end{equation}
where $\bar{h}_i$ satisfies Assumption B \eqref{assump:surrogate_assum}. This implies that $x^*_i$ is a minimizer of the following problem
\begin{equation}\label{eq:1124}
    \min_{x_i\in\mathcal X_i} F(x_i,x^*_{\neq i}) + l_i^*D_{\varphi^*_i}(x_i,x^*_i) + \bar{h}_i(x_i,x^*_i).
\end{equation}
The result follows the optimality condition of \eqref{eq:1124} and $\nabla \bar{h}_i(x_i^*,x^*) = 0$.
\end{proof}

\subsection{Global convergence}
In order to prove the global convergence of Algorithm \ref{algo:iMM_multiblock}, we need to make an additional assumption. 

\begin{assumption}\label{assumption:global}
For every iteration $k$ of Algorithm~\ref{algo:iMM_multiblock}, $f(\cdot,x_{\neq i}^{k,i-1})$ is relative smooth with respect to $\varphi_i^k$ with constants $(L^k_i,l_i^k)$ for $i\in [m]$. We will assume the following:
\begin{description}
    \item[(A)] There exist a positive integer number $N$, and $\underline{L}_i, \bar{L}_i >0$ such that 
    \begin{itemize}
            \item[$\bullet$] $\underline{L}_i \leq \min_{k\geq N}L^k_i \leq \max_{k\geq N}L^k_i \leq \bar{L}_i$ and $\delta_i < \underline{L}_i/\bar{L}_i$;
    \item[$\bullet$] for $i\in [m]$, $\varphi_i^k$ is $\rho_i^k$- strongly convex and there exists $\rho>0$ such that $\min_{k\geq N}\rho^k_i \geq \rho$.       \end{itemize}
    \item[(B)] $\nabla f$ and $\nabla \varphi_i$, for $i\in [m]$, are Lipschitz continuous on any bounded subsets of $\mathbb E$.
\end{description}
\end{assumption}
We remark that Assumption~\ref{assumption:global} (A) on the boundedness of $L_i^k$ is considered to be standard in the literature of inertial block coordinate methods, see \cite[Assumption 2]{Xu2013},  \cite[Assumption 2]{Xu2017},  \cite[Assumption 3]{Hien_ICML2020} for similar assumptions when considering block Lipschitz smooth problems. Assumptions~\ref{assumption:global} (B) is naturally satisfied when $f$ and $\varphi_i$ are twice continuously differentiable.

The global convergence of Algorithm~\ref{algo:iMM_multiblock} now can be stated for $F$ satisfying the KL property, see Definition \ref{def:KL} in \Cref{sec:prelnnopt}.

\begin{theorem}\label{theorem:globalconvergence} 
Assume that Assumptions \ref{assumption:f} to \ref{assumption:global} hold. Let $\{x^k\}$ be the sequence generated by Algorithm~\ref{algo:iMM_multiblock}. We further assume that (i) $\{x^k\}$ is bounded, (ii) for any $x_i,y_i$ in a bounded subset of $\mathcal X_i$ if $s_i\in \partial_{x_i} (I_{\mathcal X_i}(x_i) + u_i(x_i,y_i))$, there exists $\xi_i\in\partial (I_{\mathcal X_i}(x_i)+g_i(x_i))$ such that $\|\xi_i - s_i\|\leq A_i\|x_i-y_i\|$ for some constant\footnote{This assumption is naturally satisfied if $u_i(x_i,y_i)=g_i(x_i)$ (that is, we use $g_i$ itself as its surrogate). It is also satisfied if $g_i$ and $u_i$ are continuously differentiable, $\nabla_{x_i} u_i(x_i,x_i)=\nabla g_i(x_i)$, and $(x_i,y_i)\mapsto \nabla_{x_i} u_i(x_i,y_i)$ is Lipschitz continuous on any bounded subsets of  $\mathcal X_i \times \mathcal X_i$ since we then have $\nabla x_i u_i(x_i,y_i)-\nabla g_i(x_i) = \nabla_{x_i} (u_i(x_i,y_i)-u_i(x_i,x_i))$, see~\cite{Titan2020} for some examples that satisfy these conditions. } $A_i$, and (iii) $F$ satisfies the KL property at any point $x^* \in \rm{dom\,} \partial F$. Then the whole sequence $\{x^k\}$ converges to a critical point $\Phi(x) = F(x) +\sum_{i=1}^mI_{\mathcal X_i}(x_i)$. 
\end{theorem}

\begin{proof}

Consider the following auxiliary function
\begin{equation}
\begin{split}
    &\Phi^\gamma(x,y)=\\ 
    & \Phi(x) + \sum_{i=1}^m\gamma_i D_{\varphi_i}\big((x_1,\ldots,x_{i-1},y_{i},\ldots,y_m),(x_1,\ldots,x_{i},y_{i+1},\ldots,y_m)\big)
    \end{split}
\end{equation}
where $\gamma_i = (\underline{L}_i + \delta_i\bar{L}_i)/2$, and let us denote $z^k = (x^k,x^{k-1})$. Then we have $\Phi^\gamma(z^k)= \Phi(x^k)+\sum_{i=1}^m\gamma_i D_{\varphi^{k-1}_i} (x^{k-1}_i,x^{k}_i)$. Here we only need to prove that the sequence $\{z^k\}$ satisfies the three conditions $\mathbf H1, \mathbf H2$, and $\mathbf H3$ in \cite{Attouch2013} since the result can be derived by using these conditions and the same arguments of the proof for \cite[Theorem 2.9]{Attouch2013}.

($\mathbf H1$) \emph{Sufficient decrease condition}.
It follows from \eqref{eq:1125} that for all $k\geq N+1$
\begin{equation}
   \sum_{i=1}^m\underline{L}_i D_{\varphi_i^k}(x^k_i,x_i^{k+1}) +  F(x^{k+1}) \leq F(x^k) + \sum_{i=1}^m\delta_i\bar{L}_i D_{\varphi^{k-1}_i}(x_i^{k-1},x^k_i).
\end{equation}
Therefore, we have
\begin{equation*}
    \begin{split}
        \Phi^\gamma(z^k) - \Phi^\gamma(z^{k+1}) &\geq \sum_{i=1}^m\frac{\underline{L}_i - \delta_i\bar{L}_i}{2}\biggl(D_{\varphi_i^{k-1}}(x_i^{k-1},x^k_i) + D_{\varphi_i^k}(x^k_i,x_i^{k+1})\biggr)\\
       & \geq \sum_{i=1}^m\frac{(\underline{L}_i - \delta_i\bar{L}_i)\rho_i}{4}\biggl(\|x_i^{k-1}-x^k_i\|^2 + \|x^k_i - x_i^{k+1}\|^2\biggr)\\
        &\geq \tau \|z^{k+1}-z^k\|^2,
    \end{split} 
\end{equation*}
where $\tau = \min_i (\underline{L}_i - \delta_i\bar{L}_i)\rho_i/4 >0$ due to Assumption \ref{assumption:global}.

($\mathbf H2$) \emph{Relative error condition}.
By using the optimal condition of the subproblem \eqref{eq:iMM_update} in BMME, we have for all $k\geq N+1$
\begin{equation*}
    s_i^{k+1}: = L_i^k(\nabla \varphi^k_i(\bar x_i^k) - \nabla \varphi_i^k(x_i^{k+1}))  -  \nabla_i f(\bar x_i^k,x^{k,i-1}_{\neq i})\in \partial (I_{\mathcal X_i}(x_i^{k+1}) + u_i(x_i^{k+1},x_i^k)).
\end{equation*}
Hence, there exists $\xi_i^{k+1}\in\partial (I_{\mathcal X_i}(x_i^{k+1}) + g_i(x^{k+1}_i))$ such that 
\begin{equation}\label{eq--1}
    \|\xi_i^{k+1} - s_i^{k+1}\|\leq A_i\|x^{k+1}_i - x^k_i\|,
\end{equation}
for some $A_i$. Therefore, we have $d^{k+1}_i: = \nabla_{x_i}f(x^{k+1}) + \xi_i^{k+1} \in \partial_{x_i} \Phi(x^{k+1})$ and
\begin{equation*}
    \begin{split}
        \left\|d^{k+1}_i\right\| = 
         \left\| \nabla_{x_i}f(x^{k+1}) + s_i^{k+1} + \xi_i^{k+1} - s_i^{k+1}\right\|
        \leq  \left\|L_i^k(\nabla \varphi^k_i(\bar x_i^k) - \nabla \varphi^k_i(x_i^{k+1}))\right\| \\ +  \left\|\nabla_{x_i}f(x^{k+1})- \nabla_if(\bar x_i^k,x^{k,i-1}_{\neq i})\right\|
        +  \left\|\xi_i^{k+1} - s_i^{k+1}\right\| \\
        \leq \left(\left(\bar L_iL^{\varphi_i} + L_i^f\right)\left(1+\beta_i^k\right) + A_i\right) \left( \left\|x^{k+1}_i-x^k_i \right\| +  \left\|x^k_i-x^{k-1}_i \right\|\right),
    \end{split}
\end{equation*}
where the second inequality holds by \eqref{eq--1}, the boundedness of $\{x^k\}$, and the local Lipschitz continuity of $\nabla\varphi_i$ and $\nabla f$.
On the other hand, $ \partial\Phi^\gamma(z^{k+1}) = \Big(\partial_x\Phi^\gamma(z^{k+1}) , \partial_y\Phi^\gamma(z^{k+1})\Big)$, where 
\begin{equation}
    \begin{split}
        \partial_{x_i}\Phi^\gamma(z^{k+1})& = 
         \partial_{x_i} \Phi(x^{k+1}) + \sum_{j=i+1}^m\gamma_j \big(\nabla_i\varphi_j(x^{k,j-1})-\nabla_i\varphi_j(x^{k,j})\big)\\
         &\qquad-\sum_{j=i}^m\gamma_j\nabla^2_{i j}\varphi_j(x^{k,j})(x^{k}_j-x^{k+1}_j) 
    \end{split}
\end{equation}
and
\begin{equation}
    \begin{split}
        \partial_{y_i}\Phi^\gamma(z^{k+1})& =\sum_{j=1}^{i-1}\gamma_j \big(\nabla_i\varphi_j(x^{k,j-1}) - \nabla_i\varphi_j(x^{k,j})\big)-\sum_{j=1}^i\gamma_j\nabla^2_{i j}\varphi_j(x^{k,j})(x^{k}_j-x^{k+1}_j)  .
    \end{split}
\end{equation}
Therefore, we can deduce the relative error condition from the results above.

($\mathbf H3$) \emph{Continuity condition}. 
Let $x^*$ be a limit point of $\{x^k\}$. Since $\{x^k\}$ is bounded, there exists a subsequence $\{x^{k_n}\}$ of $\{x^k\}$ converging to $x^*$. Similarly to the proof of Theorem \ref{theorem:subsconvergence} (C), we can show that $u_i(x_i^{k_n},x^{k_n-1}_i) \to g_i(x_i^*)$ as $n\to+\infty$. Therefore, we have
\begin{equation}
    \begin{split}
        \limsup_{n\to+\infty}F(x^{k_n}) \leq \limsup_{n\to+\infty}f(x^{k_n}) +\sum_{i=1}^mu_i(x_i^{k_n},x^{k_n-1}_i)\\
        = f(x^*) + \sum_{i=1}^mg_i(x^*_i).
    \end{split}
\end{equation}
On the other hand, since $\{\Phi^\gamma(x^k,x^{k-1})\}$ is non-increasing and bounded below, there exists $F^* = \lim_{k\to+\infty}\Phi^\gamma(x^k,x^{k-1})$. Moreover,  $\lim_{k\to+\infty}D_{\varphi^{k-1}_i}(x_i^{k-1},x^k_i) = 0$. This implies that $\lim_{k\to+\infty}F(x^k) = \lim_{k\to+\infty}\Phi^\gamma(x^k,x^{k-1}) = F^*$. By the uniqueness of the limit, we have $F^* = F(x^*)$.

By \cite[Theorem 2.9]{Attouch2013} we have $z^k$ converges to a critical point $(x^*,x^*)$ of $\Phi^\gamma$. Note that $\partial\Phi^\gamma(x^*,x^*)=(\partial \Phi(x^*),0)$. The result follows then.
\end{proof}

\paragraph{Convergence rate.} We end this section by a remark on the convergence rate of BMME. By using the same arguments of the proof for~\cite[Theorem~2]{Attouch2009} we can derive a convergence rate for the generated sequence  of BMME (see also  \cite[Theorem 3.14]{ahookhosh2020_inertial}, \cite[Theorem 4.7]{ahookhosh2019multi}, \cite[Theorem~3]{Hien_ICML2020} and \cite[Theorem~2.9]{Xu2013}). We note that the convergence rate appears to be the same in different papers using the technique in \cite{Attouch2009}. Specifically, suppose  $\mathbf a$ be a constant such that $\xi(s)=c s^{1-\mathbf a}$, where $c$ is a constant, see   Definition~\ref{def:KL}. Then if $\mathbf a=0$, BMME converges after a finite number of  steps; if $\mathbf a\in (0,1/2]$, BMME has linear convergence; and if $\mathbf a \in (1/2, 1)$, BMME has sublinear convergence. Determining the K{\L}  exponent $\mathbf a$ is out of the scope of this paper.  

\section{Numerical results}
\label{sec:experiment}

In this section, we apply BMME to solve the following penalized orthogonal nonnegative matrix factorization (NMF)~\cite{ahookhosh2019multi,pompili2014two} 
	\begin{equation}
	\label{eq:penelONMF}
	\min_{U\in \mathbb R_+^{\mbfm\times \mbfr}, V\in \mathbb R_+^{\mbfr\times \mbfn}} f(U,V):= \tfrac{1}{2}\|X-UV\|_F^2+\tfrac{\lambda}{2}\|I_r-V V^\top\|_F^2, 
		\end{equation}
where $X\in \mathbb R^{\mbfm\times \mbfn}_+$ is a given input nonnegative data matrix and  $\lambda>0$ is a penalty parameter.
We implement all of the algorithms in MATLAB R2018a and run the experiments on a laptop with 1.8 GHz Intel Core i7 CPU and 16 GB RAM. The codes are available at \url{https://github.com/LeThiKhanhHien/BMME}.

\label{sec:experiment1}


\subsection{Kernel functions and block updates of BMME} \label{kernelONMF}

To implement BMME (Algorithm~\ref{algo:iMM_multiblock}), 
we use the following  kernel functions
\begin{equation}
\label{eq:h1h2}
\begin{split}
    \varphi_{1}(U,V)&=\frac{1}{2}\| U\|_{F}^{2},\\
    \varphi_{2}(U,V)&=\tfrac{6\lambda}{4}\|V\|_{F}^{4}+\tfrac{1}{2} \varepsilon(U) \|V\|_{F}^{2},
\end{split} 
\end{equation} 
where $\varepsilon(U)>0$ may depend on $U$. 
Let us choose $\varepsilon(U)=\max\{\|U^\top U\|,2\lambda\}$. Note that $U\mapsto \varphi_{1}(U,V)$ and $V \mapsto  \varphi_{2}(U,V)$ are strongly convex.  
Let us show that $f$ is block relative smooth with respect to these kernel functions. 

\begin{proposition} \label{prop1}
 Fixing $V$, the function $f(\cdot,V)$ is  $(L_{1}(V),l_1)$-relatively smooth with respect to $\varphi_{1}(\cdot,V)$, with $L_{1}(V)=\|V V^{\top}\|$ and $l_1=0$. Fixing $U$, $f(U,\cdot)$ is $(L_{2},l_2)$- relatively smooth with respect to $\varphi_{2}(U,\cdot)$, with $L_{2}=1$ and $l_2=1$.
\end{proposition}

\begin{proof}  
The first statement is straightforward. Let us prove the second one. 
From~\cite[Proposition 5.1]{ahookhosh2019multi}, we have
$$
\nabla_V^{2}f (U,V)[Z]=U^{\top}U Z+2\lambda(Z V^{\top}V+V Z^{\top}V+V V^{\top}Z-Z).
$$
Note that 
\begin{equation}
    \label{eq:ieq}
 \left| \iprod{Z V^{\top}V+V Z^{\top}V+V V^{\top}Z }{Z} \right|\leq 3 \|V\|_F^2\|Z\|_F^2.
\end{equation}
Hence 
$$\left\langle \nabla_V^{2}f(U,V)[Z],Z\right\rangle \leq\left\Vert U^{\top}U\right\Vert \left\Vert Z\right\Vert _{F}^{2}+6\lambda\left\Vert V\right\Vert _{F}^{2}\left\Vert Z\right\Vert _{F}^{2}.
$$
Furthermore, we have  $$\nabla_V^{2} \varphi_{2}(U,V)[Z]=6\lambda\left(\|V\|_{F}^{2}Z+2\left\langle V,Z\right\rangle V\right) +
\max\{\| U^{\top}U\|,\varepsilon\} Z,$$
which implies 
\begin{align*}
L_2\left\langle \nabla_V^{2}\varphi_{2}(U,V)[Z],Z\right\rangle 	&=\max\{\left\Vert U^{\top}U\right\Vert,2\lambda\} \left\Vert Z\right\Vert ^{2}+6\lambda\left(\|V\|_{F}^{2}\left\Vert Z\right\Vert _{F}^{2}+2\left\langle V,Z\right\rangle ^{2}\right)\\
&	\geq\left\Vert U^{\top}U\right\Vert \left\Vert Z\right\Vert _{F}^{2}+6\lambda\|V\|_{F}^{2}\left\Vert Z\right\Vert _{F}^{2}	\geq\left\langle \nabla_V^{2}f (U,V)[Z],Z\right\rangle.   
\end{align*}
On the other hand, since   $\max\{\|U^{\top}U\|,2\lambda\}\geq 2\lambda$ we have 
\begin{align*}
&\left\langle \nabla_V^{2}f (U,V)[Z],Z\right\rangle + l_2\left\langle \nabla_V^{2}\varphi_{2}(U,V)[Z],Z\right\rangle \\
&\geq \left\langle U^{\top}U Z+2\lambda(Z V^{\top}V+V Z^{\top}V+V V^{\top}Z),Z\right\rangle +  6\lambda\|V\|_{F}^{2}\|Z\|_{F}^{2}\geq 0,
\end{align*}
where we have used~\eqref{eq:ieq} for the last inequality. The result follows, see \cite[Proposition 1.1]{lu2018relatively}, \cite[Proposition 2.6]{ahookhosh2019bregman}. \qed
\end{proof}


Proposition~\ref{prop1} shows that the kernel functions in~\eqref{eq:h1h2} allow $f$ to satisfy 
Definition~\ref{def:block_relative_smooth}, that is, $f$ is block relative smooth with respect to these kernels. 
This would not hold for the block relative smoothness definitions from \cite[Definition 2]{ahookhosh2020_inertial} and \cite[Definition 3.4]{ahookhosh2019multi}. 
In fact, $L_1(V)$ depends on $V$ so \cite[Definition 2]{ahookhosh2020_inertial} does not apply, while  $\varphi_1$ and $\varphi_2$ are two different functions so \cite[Definition 3.4]{ahookhosh2019multi} does not apply either as it requires a sole multi-block kernel function to define block relative smoothness.


In the following we provide closed-form solutions of the sub-problems in~\eqref{eq:iMM_update} for the penalized ONMF problem.  
\begin{proposition} \label{prop:ONMFBMME} Let $\varphi_1$ and $\varphi_2$ be defined in~\eqref{eq:h1h2}. 
Given $\bar U$, $V$, and $L_1$, we have 
$$
 \arg\min_{U\geq0}\left\langle \nabla_{U}f(\bar U,V),U\right\rangle + L_1 D_{\varphi_{1}(\cdot,V)}(U,\bar U)
	=\max\left(\bar U-\frac{1}{L_1}\left(\bar U V V^{\top}-XV^{\top}\right),0\right).$$
Given $\bar V$, $ U$, and $L_2$  we have 
	$$
	\arg\min_{V\geq0}\left\langle \nabla_{U}f( U,\bar V),V\right\rangle + L_2 D_{\varphi_{2} (U,\cdot)}(V,\bar V)
	\; = \;  \frac{1}{\rho}\max(G(\bar V),0),
$$
where 
\begin{align*}
    G(\bar V)&=\nabla_V \varphi_{2}(U,\bar V)-\frac{1}{L_2}\nabla_{V}f(U,\bar V)\\
    &= (6\lambda \|\bar V\|^2_F + \varepsilon(U)) \bar V -\frac{1}{L_2} \big( U^\top U \bar V - U^\top X + 2\lambda(\bar V \bar{V}^\top \bar V - \bar V)\big),
\end{align*}
and $\rho$ is the unique real solution of the equation $\rho^2(\rho-   a)=c$, where $a= \varepsilon(U)$ and $c=6\lambda \|\max(G(\bar V),0) \|^2$, so that $\rho$ has the following closed form 
$$\rho=\frac{a}{3} + \sqrt[3]{\frac{c+\sqrt{\Delta}}{2}+\frac{a^3}{27}} + \sqrt[3]{\frac{c-\sqrt{\Delta}}{2}+\frac{a^3}{27}},
$$
where $\Delta=c^2 + \frac{4}{27} c a^3$.
\end{proposition}
\begin{proof}
For the update of $U$, we have 
\begin{align*}
&\arg\min_{U\geq0}\left\langle \nabla_{U} f(\bar U,V),U\right\rangle + L_1 D_{\varphi_{1}(\cdot,V)}(U,\bar U)\\
&	=\arg\min_{U\geq0}\left\langle \nabla_{U}f(\bar U,V),U\right\rangle + L_1\left(\varphi_{1}(U,V)-\varphi_{1}(\bar U,V)-\left\langle \nabla_U\varphi_{1}(\bar U,V),U-\bar U\right\rangle \right)\\
&	=\arg\min_{U\geq0}\varphi_{1}(U,V)-\left\langle \nabla_U \varphi_{1}(\bar U,V)-\frac{1}{L_1}\nabla_{U}f(\bar U,V),U\right\rangle\\ 
&	=\max\left(\nabla_U \varphi_{1}(\bar U)-\frac{1}{L_1}\nabla_{U}f(\bar U,V),0\right)\\
&	=\max\left(\bar U-\frac{1}{L_1}\left(\bar U V V^{\top}-X V^{\top}\right),0\right).
	\end{align*}
	For the update of $V$, we have
\begin{align*}
    &\arg\min_{V\geq0}\left\langle \nabla_{V} f(U,\bar V),V\right\rangle +L_2 D_{\varphi_{2}(U,\cdot)}(V,\bar V)\\
    &=\arg\min_{V\geq0}\left\langle \nabla_{V}f(U,\bar V),V\right\rangle +L_2\left(\varphi_{2}(U,V)-\varphi_{2}(U,\bar V)-\left\langle \nabla_V \varphi_{2}(U,\bar V),V\right\rangle \right)\\
    &=\arg\min_{V\geq0}\varphi_{2}(U,V)-\left\langle G(\bar V),V\right\rangle. 
\end{align*}	
 Using the same technique as in the proof of \cite[Theorem 5.2]{ahookhosh2019multi}, we obtain the result. \qed
\end{proof}

\paragraph{Computational cost of BMME for Penalized ONMF} 

The updates of BMME for~\eqref{eq:penelONMF} are given by Proposition~\ref{prop:ONMFBMME}. 
The main cost of the update of $U$ is to compute  $U(VV^\top)$ and 
$XV^\top$ which require 
$O( (\mbfm+\mbfn)\mbfr^2)$ and 
$O( \mbfm\mbfn\mbfr )$
operations, respectively. Since $\mbfr \ll \min(m,n)$, the update of $U$ costs $O( \mbfm\mbfn\mbfr )$
operations, and is linear in the dimensions of the input matrix, as most NMF algorithms. 
The main cost of the update of $V$ is to compute  $(U^\top U) \bar V$, 
$U^\top X$, and $\bar V (\bar{V}^\top \bar V)$ which require 
$O( (\mbfm+\mbfn)\mbfr^2)$,
$O( \mbfm\mbfn\mbfr )$ and  $O(\mbfn\mbfr^2)$ operations, respectively. 
Evaluating $D(\cdot,\cdot)$ in the backtracking line search to compute the extrapolation parameter costs $O(\mbfn\mbfr)$ operations.  
In summary, BMME requires $O( \mbfm\mbfn\mbfr )$ operations per iteration. 
Note that, if $X$ is sparse, the cost per iteration reduces to $O( \text{nnz}(X) \mbfr )$ operations where $\text{nnz}(X)$ is the number of nonzero entries of $X$. 

In summary, BMME has the same computational cost per iteration as most  NMF algorithms, requiring $O( \text{nnz}(X) \mbfr )$ operations per iteration, the main cost being the computation of $U(VV^\top)$, 
$XV^\top$, $(U^\top U)V$, and $U^\top X$; see~\cite[Chapter 8]{gillis2020book} for a discussion.

\subsection{Experiments on synthetic and real data sets}  \label{exp:synthreal}

In the following, we compare the following algorithms on the penalized ONMF problem: 
\begin{itemize}
\item BMME with the kernel functions defined in~\eqref{eq:h1h2}. At iteration $k$ and for updating the blocks $U$ and $V$, we find the extrapolation parameter $\beta_i^k$ for BMME by starting from $\nu_k$, where $\nu_0=1$ and $\nu_k = 1/2(1+\sqrt{1+4\nu_{k-1}})$ for $k \geq 1$, and then reducing it by a factor $0.9$ until the condition~\eqref{eq:beta} is satisfied. 


\item BMM, the non-extrapolated version of BMME.  

\item A-BPALM proposed in \cite{ahookhosh2019multi}. 

\item BIBPA proposed in \cite{ahookhosh2020_inertial}.
\end{itemize} 

We use the kernel function $\varphi (U,V)=\frac{\alpha}{2}\|U\|_F^2 \|V\|_F^2 + \frac{\beta}{4} \|V\|_F^4 + \frac{\varepsilon_1}{2}\|U\|_F^2 + \frac{\varepsilon_2}{2}\|V\|_F^2$, where $\alpha, \beta,\varepsilon_1$ and $\varepsilon_2$ are positive constants, for A-BPALM and BIBPA as proposed in \cite[Proposition 5.1]{ahookhosh2019multi}, and choose the default values for the parameters of A-BPALM and BIBPA in the upcoming experiments.  All of these algorithms have convergence guarantee for solving the penalized ONMF problem~\eqref{eq:penelONMF}, 
 and have roughly the same computational cost per iteration, requiring $O( \text{nnz}(X) \mbfr )$ 
 operations.  
 We will display the evolution of the objective function values with respect to time; the evolution with respect to the iterations being very similar. 
 
 In the following four sections, we compare the four algorithms above on four types of data sets: 
 synthetic data sets (Section~\ref{sec:synth}), 
 facial images (Section~\ref{sec:facim}), 
 and 
  document data sets (Section~\ref{sec:docu}).

\subsubsection{Synthetic data sets} \label{sec:synth}

Let us compare the algorithms
on synthetic data sets, 
as done in~\cite{ahookhosh2019multi}. 
We use $(\mbfm,\mbfn,\mbfr)=(500,500,10)$ and $(\mbfm,\mbfn,\mbfr)=(500,2000,10)$. For each choice of $(\mbfm,\mbfn,\mbfr)$, we generate 30 synthetic data sets; each data set is generated as in. Specifically,  we generate randomly the factor $U \in \mathbb R_+^{\mbfm\times \mbfr}$  and the noise matrix $R\in \mathbb R_+^{\mbfm\times \mbfn}$ using the MATLAB command  $\mathsf{rand}$. We generate an  orthogonal nonnegative matrix $V\in \mathbb R_+^{\mbfr\times \mbfn}$ with a single nonzero entry in each column of $V$ as follows.  The  position of the nonzero entry is picked at random (with probability $1/r$ for each position), 
then the nonzero entry is generated using the uniform distribution in the interval $[0,1]$, and finally we normalize each row of $V$.  
Then we construct $X$, adding 5\% of noise, as follows 
\[
X \; = \;  UV + 0.05 \frac{\| X\|_F}{\|R\|_F } R . 
\] 
For each data set, we run
each algorithm for 15 seconds and use the same initialization for all algorithms, namely the successive projection algorithm (SPA)~\cite{Araujo01,gillis2013fast} as done in \cite{ahookhosh2019multi}. 
We set the penalty parameter $\lambda=1000$ in our experiments. We report the evolution of the objective function with respect to time in Figure~\ref{fig:synthetic}.
\begin{figure}
    \centering
    \includegraphics[scale=0.38]{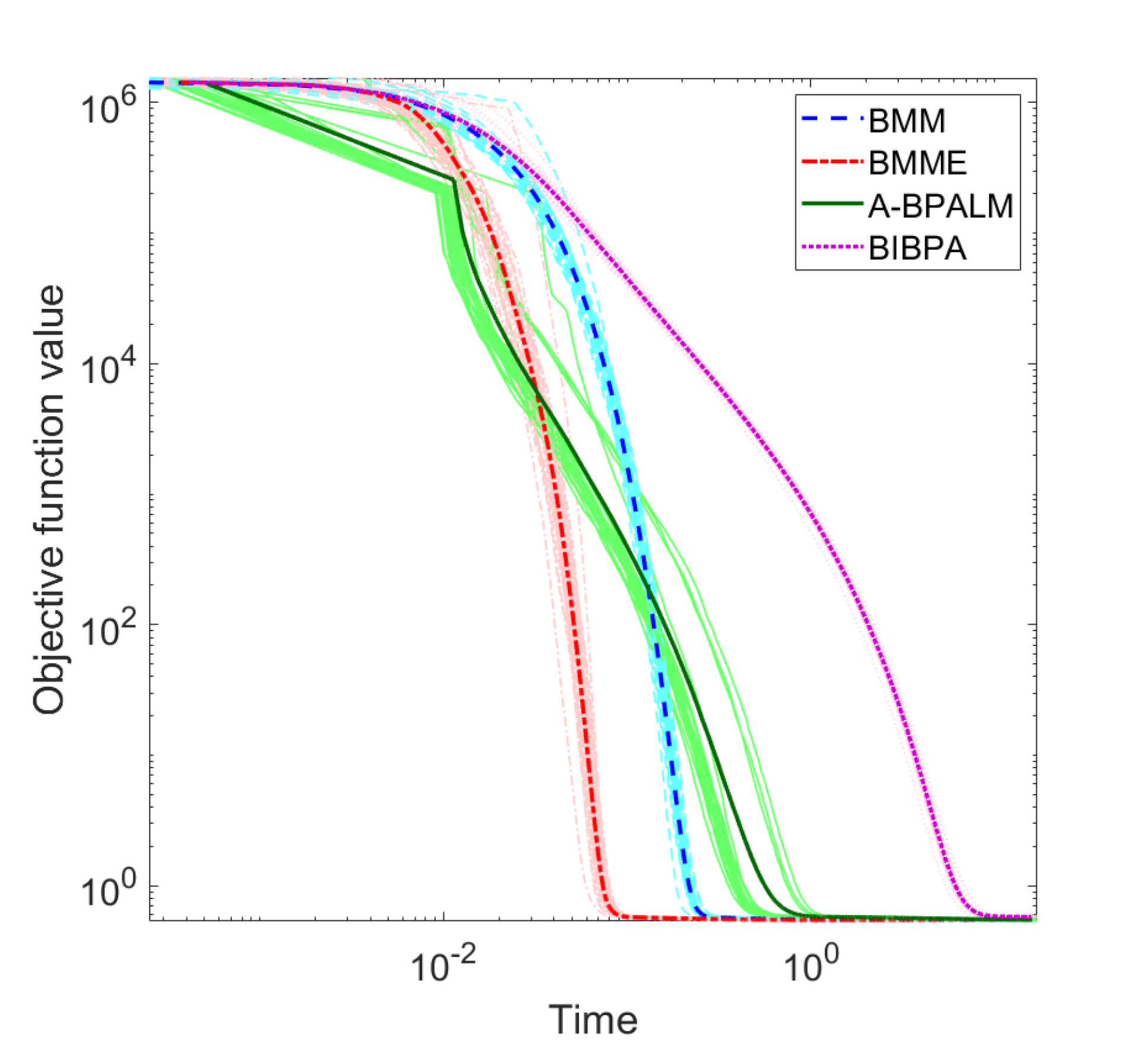} 
     \includegraphics[scale=0.38]{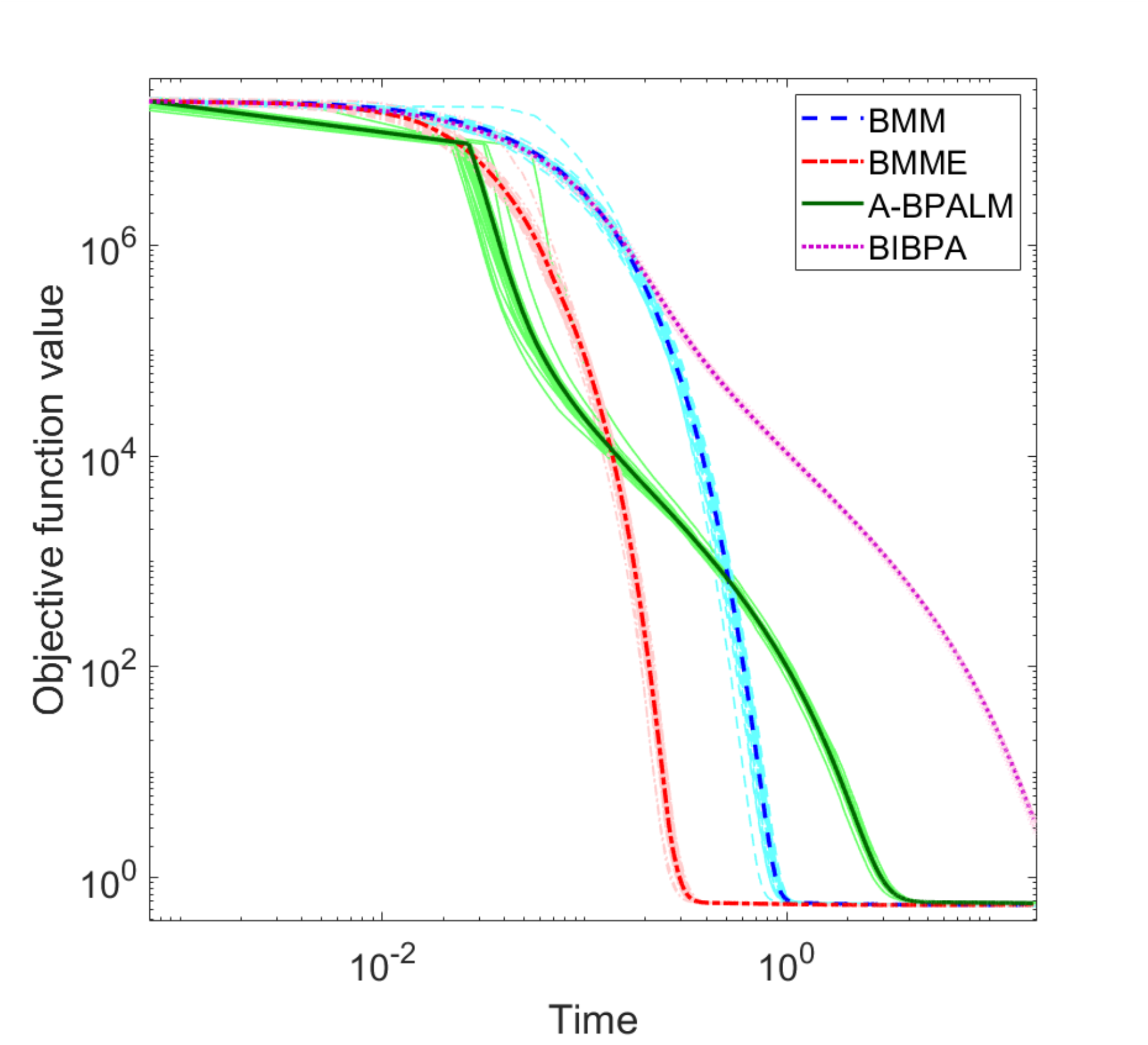} 
    \caption{Evolution of the objective functions with respect to the running time for 30 synthetic data sets with $(\mbfm,\mbfn,\mbfr)=(500,500,10)$ (left) and $(\mbfm,\mbfn,\mbfr)=(500,2000,10)$ (right), in a log-log scale. The average curve is plotted in bold.
    \label{fig:synthetic}}
    
\end{figure}
We observe that BMME consistently outperforms the other algorithms in  term of convergence speed, followed by BMM and A-BPALM. Note that the results are very consistent among various runs on different input matrices.   

These experiments illustrate two facts: 
\begin{enumerate}
    \item Using extrapolation in BMME is useful and accelerates the convergence, as BMME outperforms BMM. 
    
    \item  The flexibility of Definition~\ref{def:block_relative_smooth} allows BMME to choose the kernel functions in ~\eqref{eq:h1h2} that also leads to a significant speedup as BMME outperforms A-BPALM and BIBPA. 
    This illustrates our arguments in the paragraph ``Flexibility of Definition~\ref{def:block_relative_smooth}'' at the end of Section~\ref{sec:Bregman} (page~\pageref{paragraph:flex}). 
\end{enumerate}

In the next sections, we perform numerical experiments on real data sets to further validate these two key observations.

\subsubsection{Facial images}  \label{sec:facim}

In the second experiment, we compare the algorithms on four facial image data sets widely used in the NMF community: CBCL\footnote{\url{http://cbcl.mit.edu/software-datasets/heisele/facerecognition-database.html}}  
(2429 images of dimension 19 $\times$ 19), Frey\footnote{\url{https://cs.nyu.edu/~roweis/data.html}} (1965 images of dimension  $28 \times 20$), ORL\footnote{\url{ https://cam-orl.co.uk/facedatabase.html }} (400 images of dimension  $92 \times 112$),  and Umist\footnote{\url{https://cs.nyu.edu/~roweis/data.html}} (575 images of dimension  $92 \times 112$). 
We construct $X$ as an image-by-pixel matrix, that is, each row of $X$ is a vectorized facial image. 
As we will see, this allows ONMF to extract disjoint facial features as the rows of $V$. 
We set $\mbfr=25$ and use SPA initialization in all runs. We choose the penalty parameter $\lambda= \|X-U_0 V_0\|_F^2 / \mbfr$, where $(U_0,V_0)$ is the SPA initialization. We run each algorithm 100 seconds for each data set. 
The evolution of the scaled objective function values, which equal the objective function values divided by $\|X\|_F^2$, with respect to time is reported in Figure~\ref{fig:image}. 
\begin{figure}
   \centering
   \includegraphics[scale=0.4]{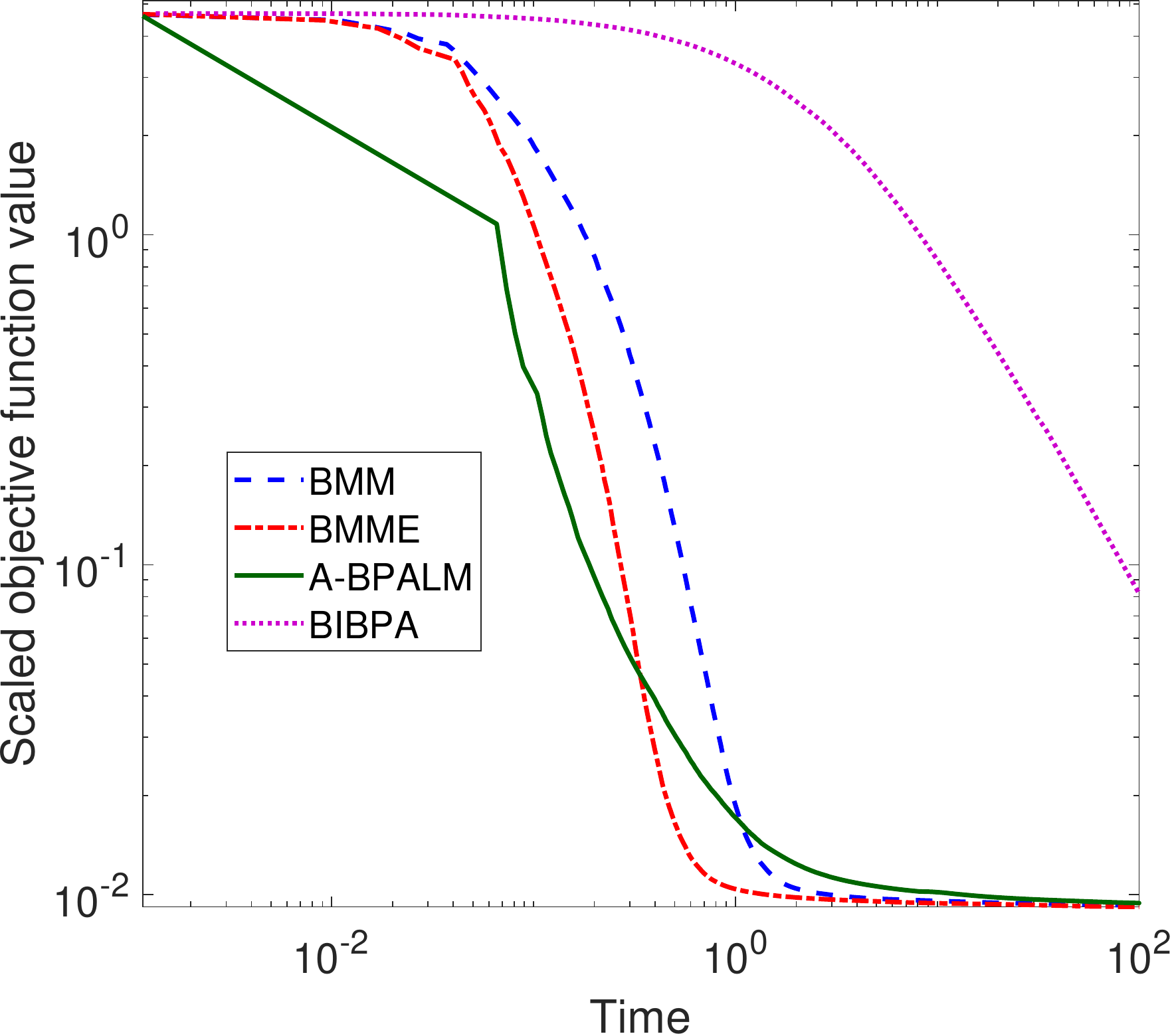}  \includegraphics[scale=0.4]{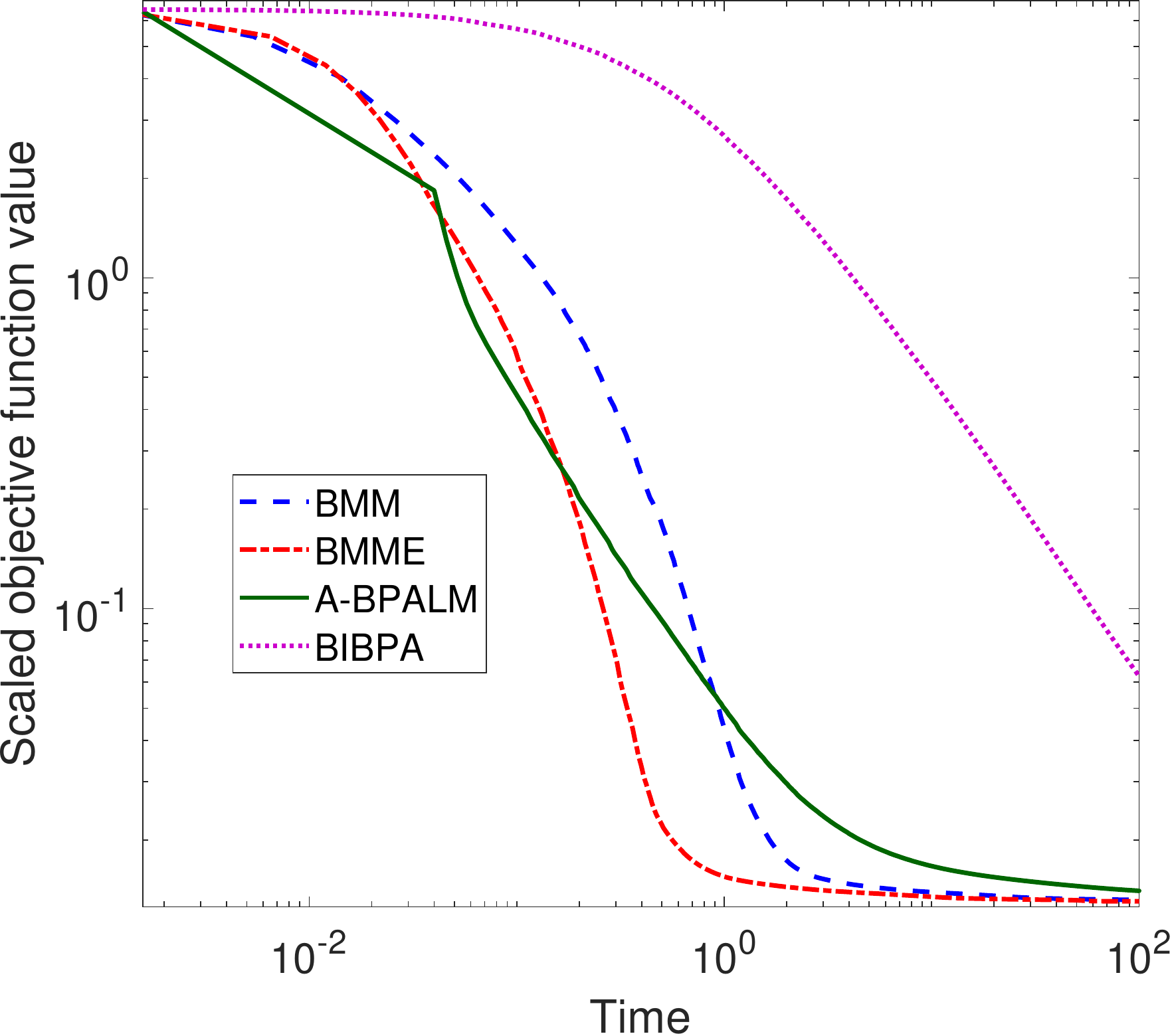}
    \includegraphics[scale=0.4]{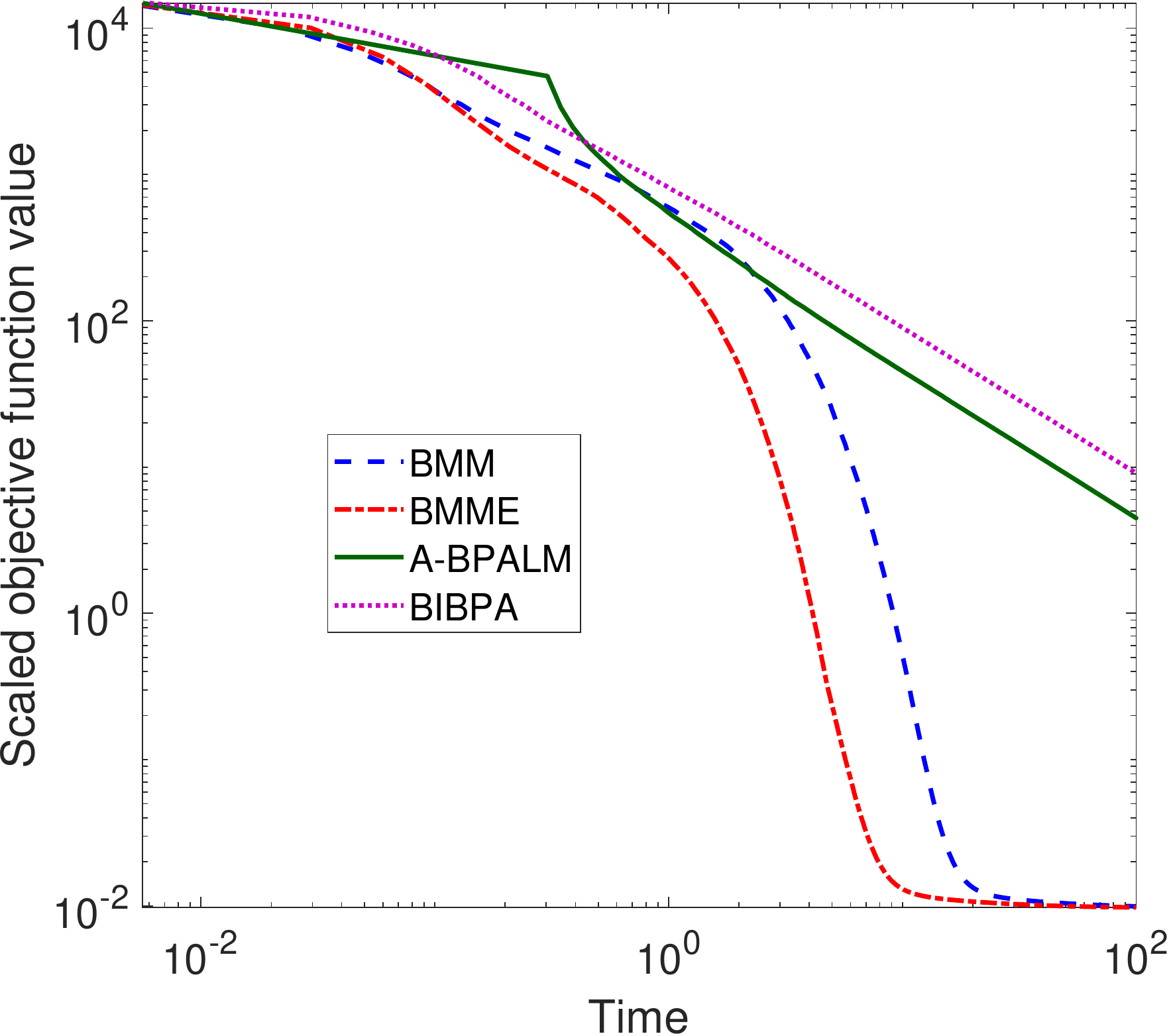}  \includegraphics[scale=0.4]{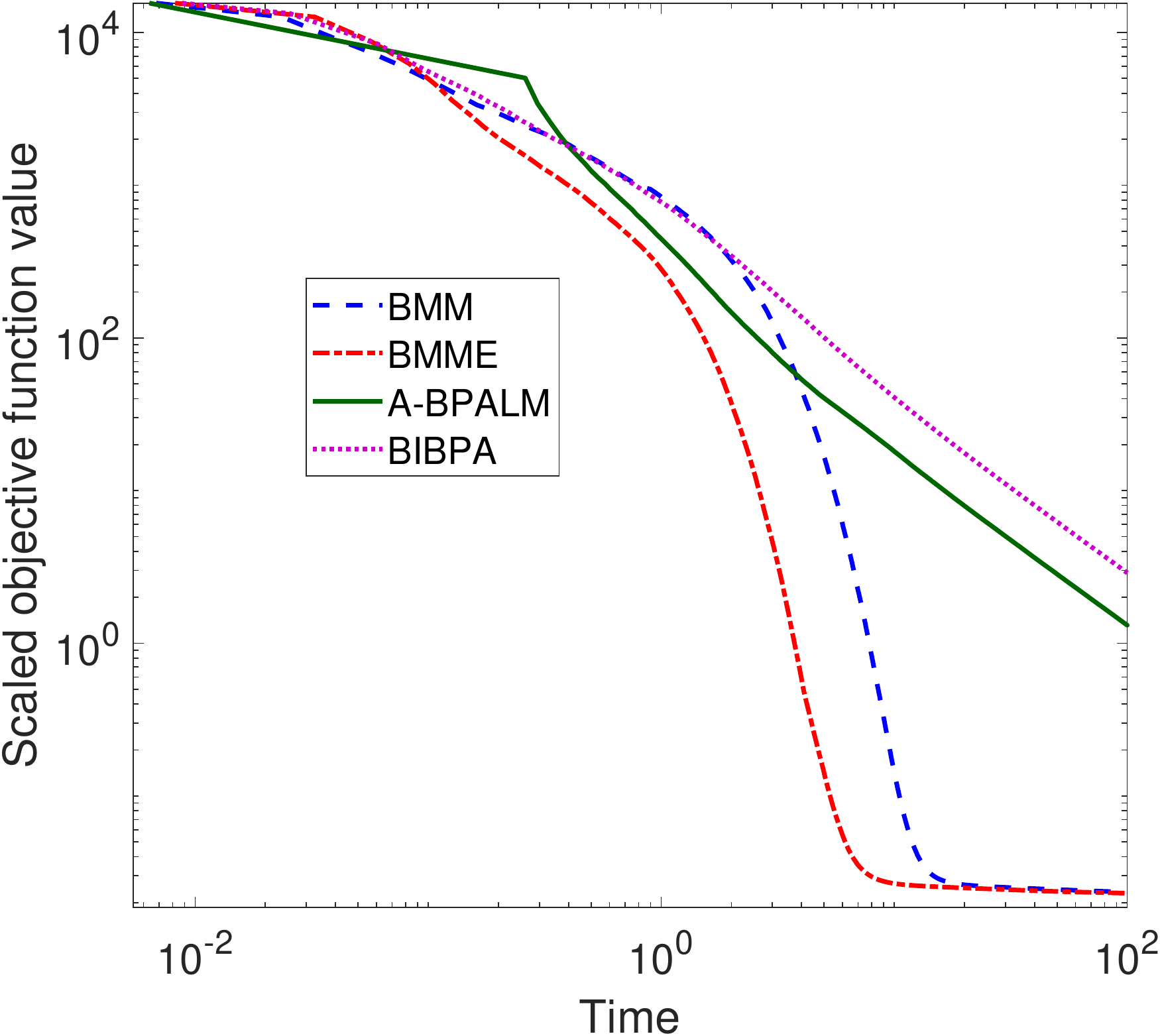}
        \caption{Evolution of the scaled objective function values with respect to running time  on the CBCL (top left), Frey (top right), Umist (bottom left) and ORL (bottom right) data sets, in loglog scale. 
        \label{fig:image}
        }
\end{figure}

We observe a similar behavior as 
for the synthetic data sets: 
 BMME converges the fastest, followed by BMM. 
Figures~\ref{fig:cbclim} and~\ref{fig:ORL} display the reshaped rows of  $V$, corresponding to facial features, obtained by the different algorithms for the CBCL and ORL data sets, respectively.  
For the other data sets, Frey and Umist; 
see \Cref{appendixfaces}.  

\begin{figure}
   \centering
  \begin{tabular}{cc}
      \includegraphics[width=0.35\textwidth]{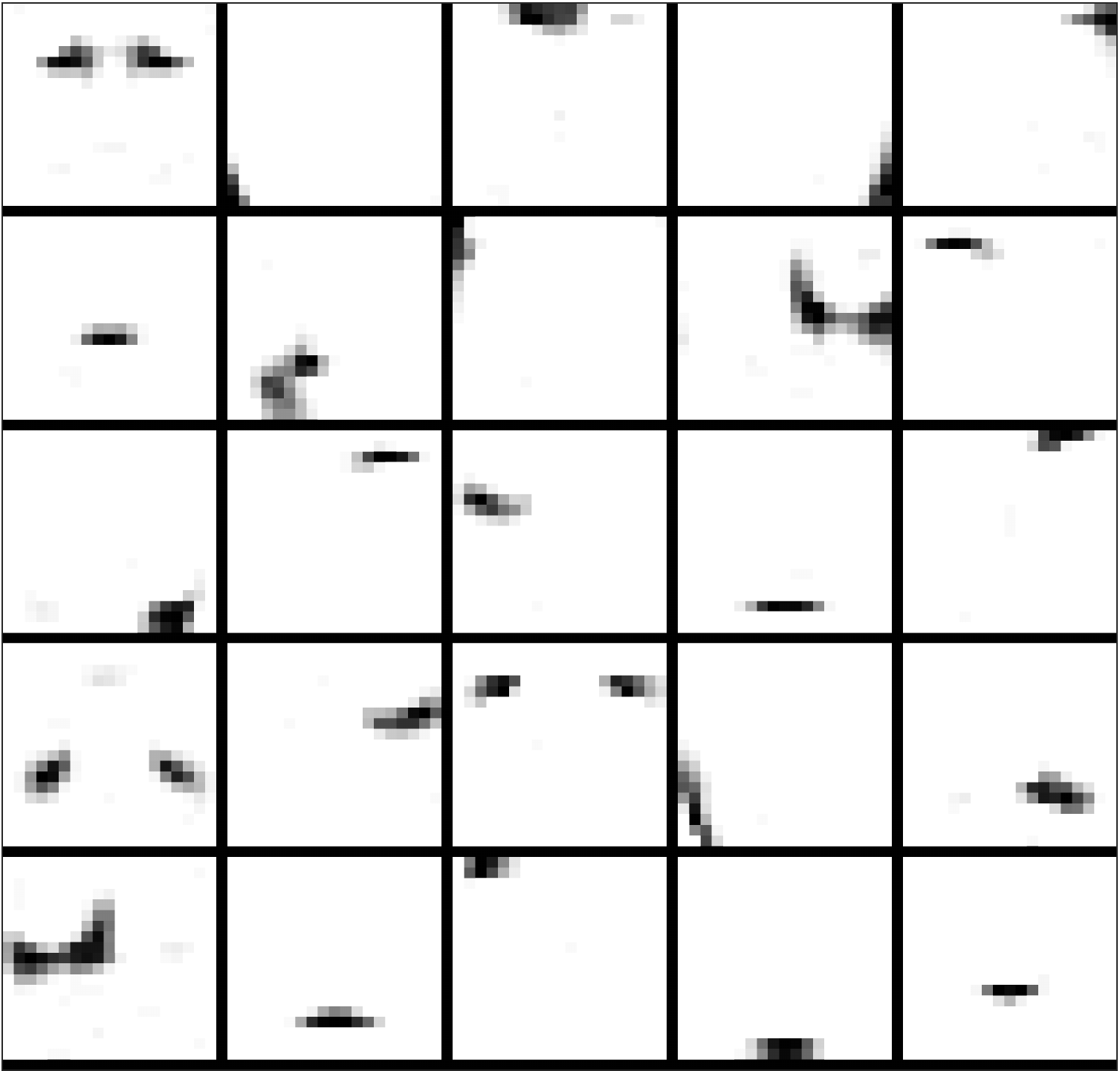}  & \includegraphics[width=0.35\textwidth]{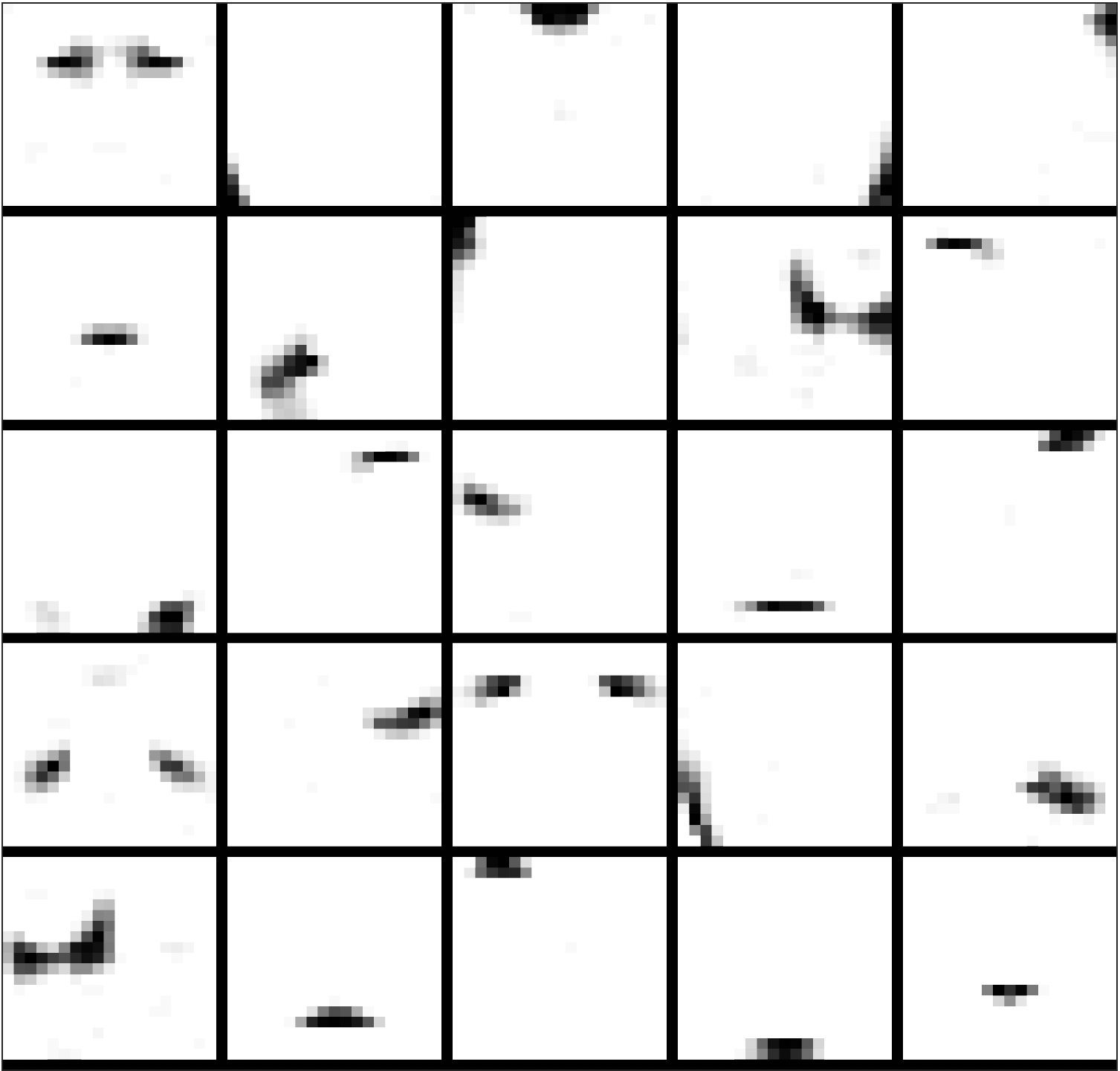}    \\
      BMM & BMME  \vspace{0.2cm} \\
          \includegraphics[width=0.35\textwidth]{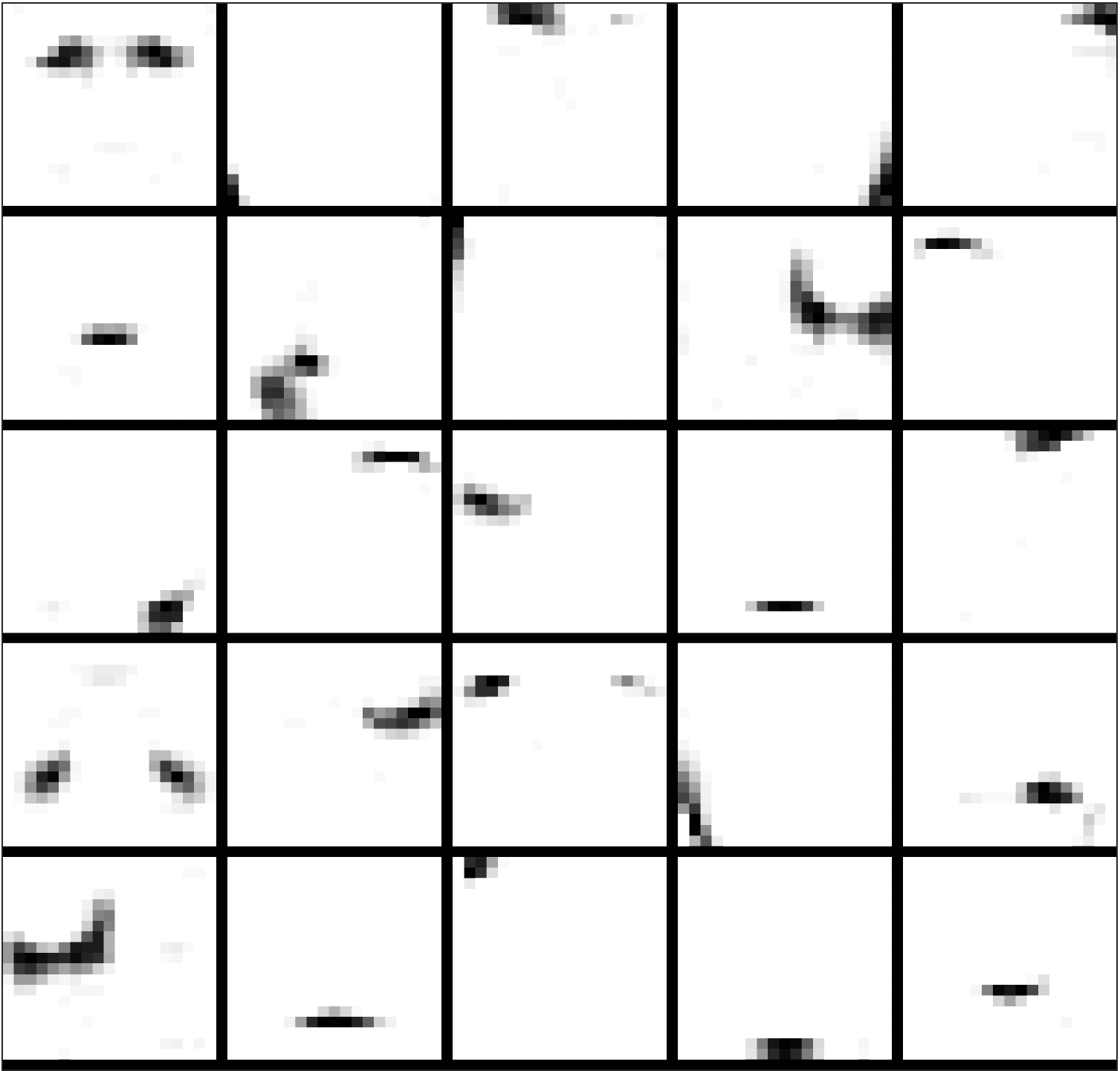}  & \includegraphics[width=0.35\textwidth]{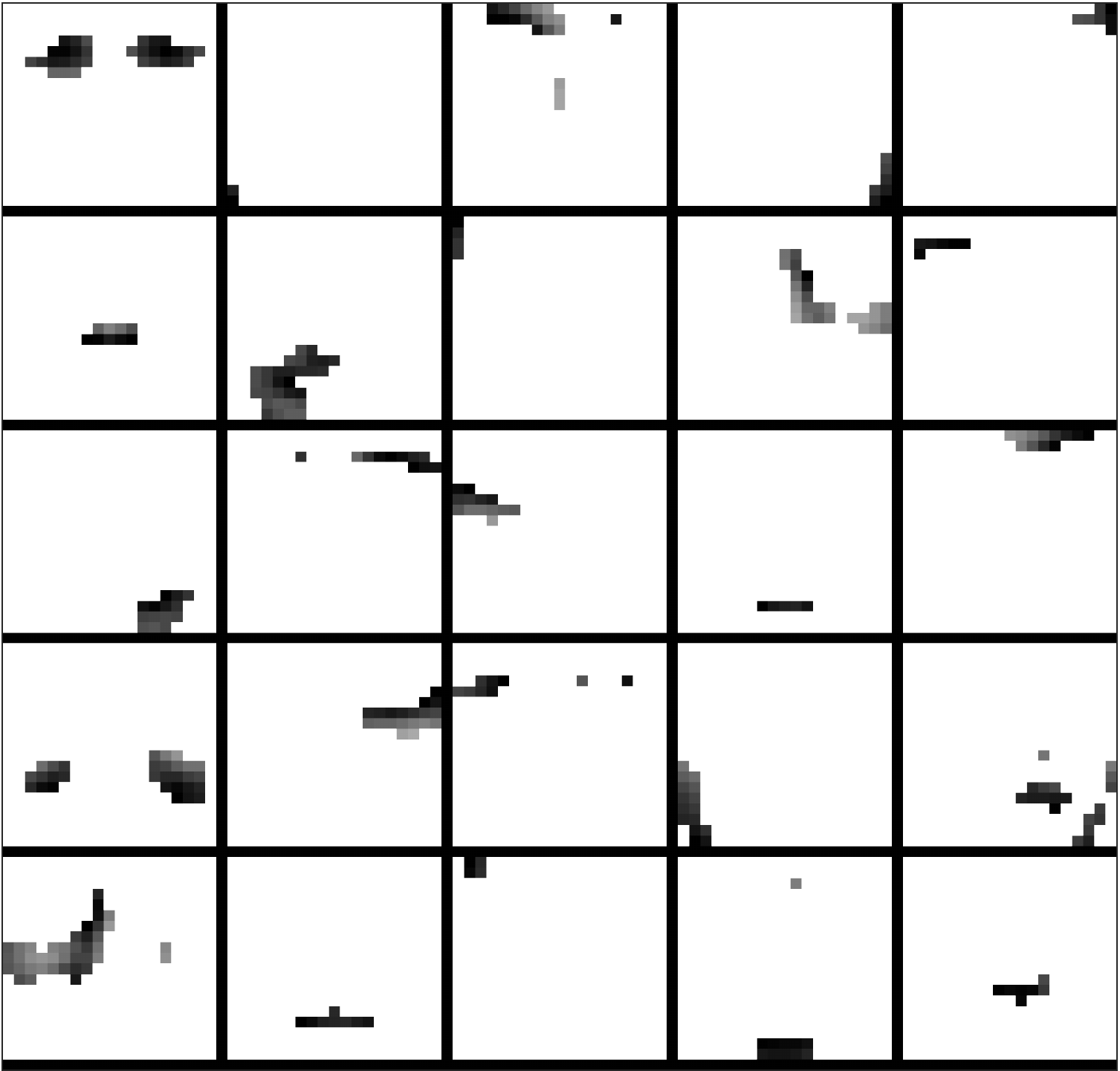} \\ 
          A-BPALM & BIBPA 
  \end{tabular}
        \caption{Display of the rows of the matrix $V$ as facial features, computed by BMM, BMME, A-BPALM, and BIBPA on the CBCL facial images.} 
    \label{fig:cbclim}
\end{figure}
In Figure~\ref{fig:cbclim}, we observe that the solutions obtained by the 4 algorithms are similar. 
Because BIBPA did not have time to converge (see Figure~\ref{fig:image}), it generates slightly worse facial features, with some isolated pixels, and edges of the facial features being sharper.

\begin{figure}
   \centering
   \begin{tabular}{cc}
       \includegraphics[width=0.4\textwidth]{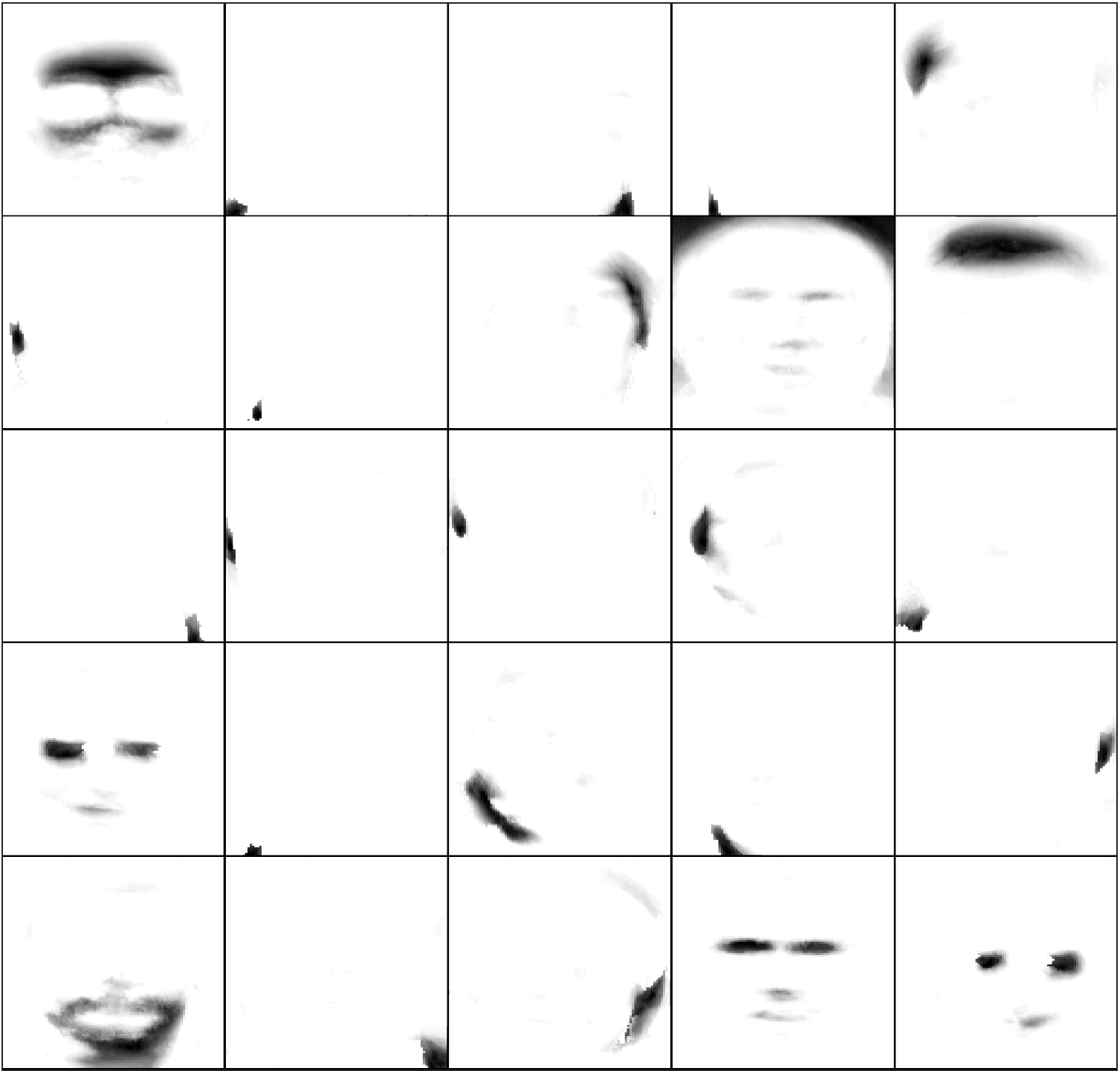}  & \includegraphics[width=0.4\textwidth]{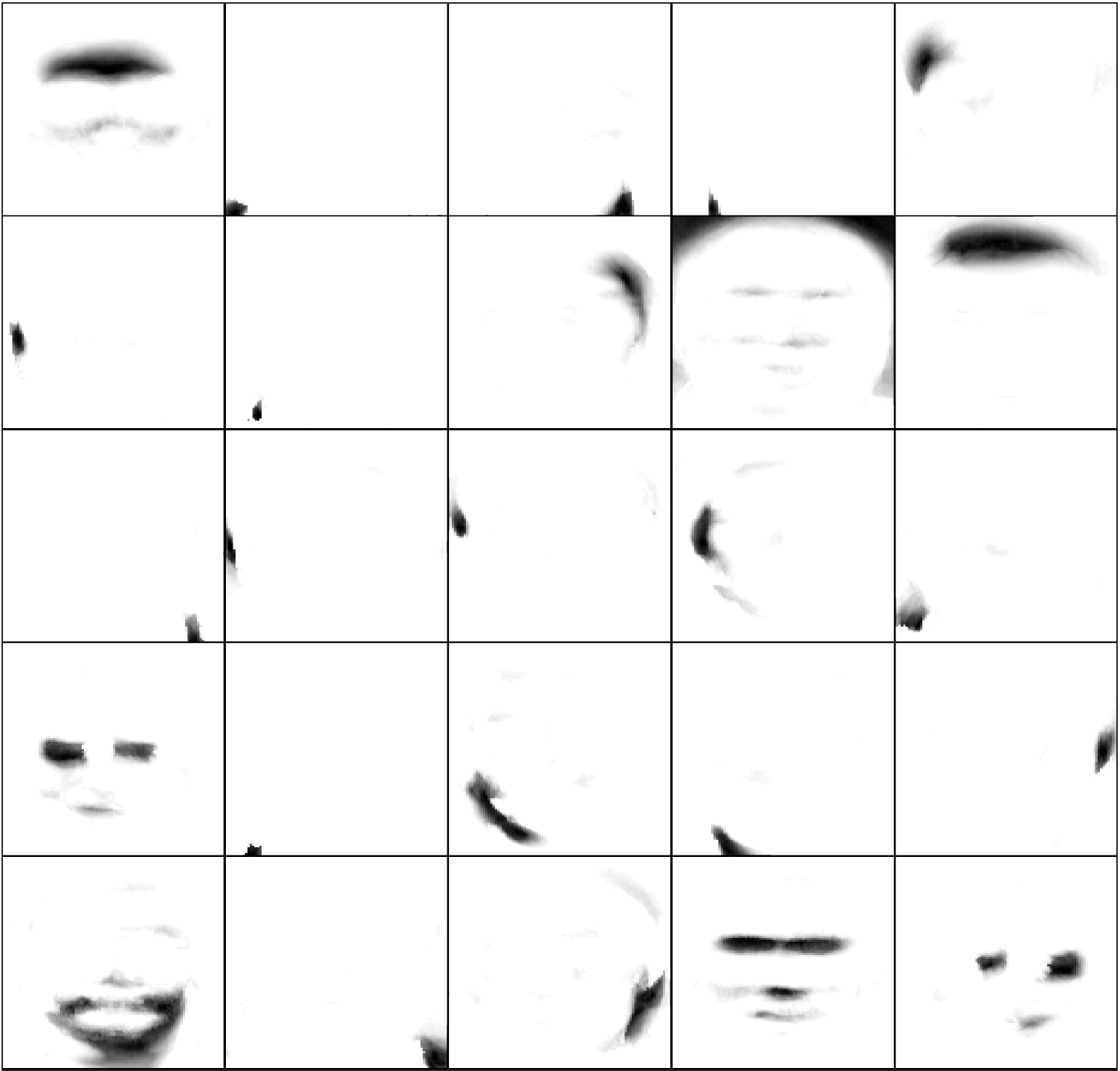}    \\
       BMM & BMME  \vspace{0.2cm} \\
           \includegraphics[width=0.4\textwidth]{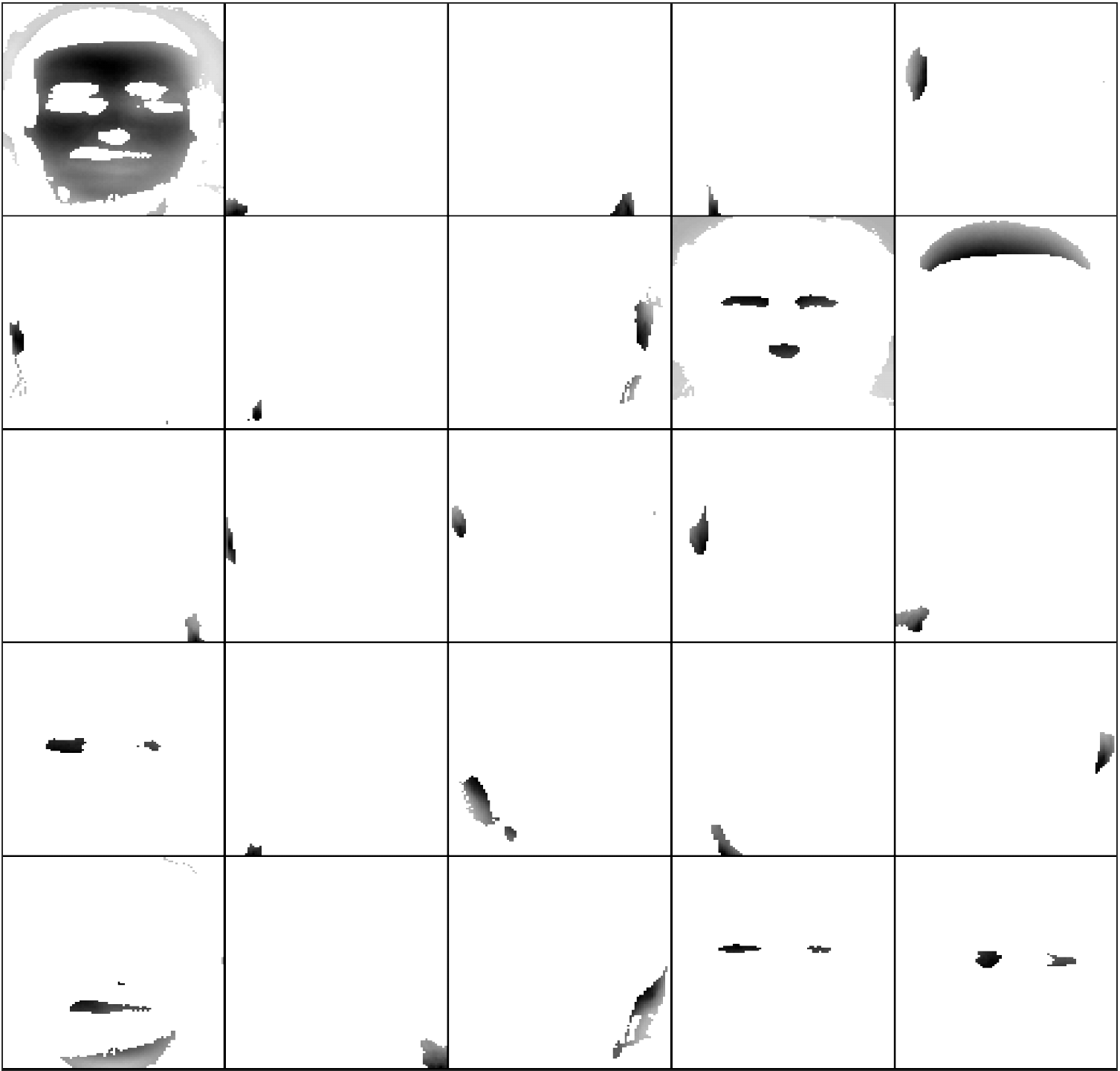}  & \includegraphics[width=0.4\textwidth]{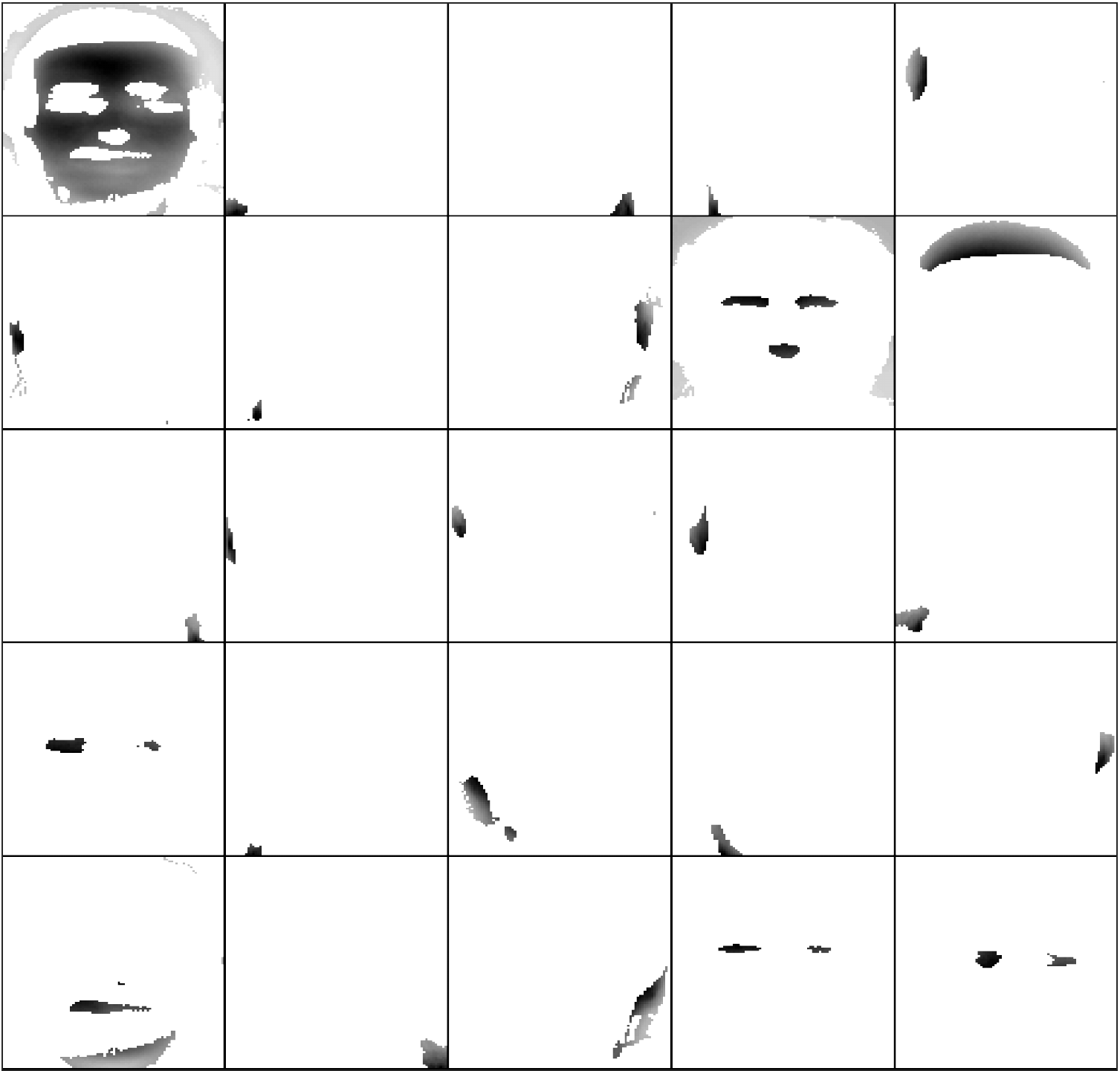} \\ 
           A-BPALM & BIBPA 
   \end{tabular}
        \caption{Display of the rows of the matrix $V$ as facial features, computed by BMM, BMME, A-BPALM, and BIBPA on the ORL facial images.}
    \label{fig:ORL}
\end{figure}
In Figure~\ref{fig:ORL}, as BMM and BMME both converged to similar objective function values (see Figure~\ref{fig:image}), they provide very similar facial features; although slightly different. For example, the first facial feature of BMME is sparser than that of BMM. 
A-BPALM and BIBPA were not able to converge within the 100 seconds, and hence provide worse facial features. For example, the first facial feature is much denser than for BMM and BMME, overlapping with other facial features (meaning that the orthogonality constraints is not well satisfied). (A very similar observation holds for the Umist data set; see \Cref{appendixfaces}.)

\subsubsection{Document data sets}   \label{sec:docu}

In the fourth experiment, we compare the algorithms on 12 sparse document 
data sets from~\cite{ZG05}, as in~\cite{pompili2014two}.  
For such data sets, SPA does not provide a good initialization, because of outliers and gross corruptions. Hence we initialize  $U_0$ with the procedure provided by H2NMF from~\cite{gillis2015hierarchical}, while $V_0$ is initialized by minimizing 
$\|X- U_0V_0\|_F^2$ while imposing $V_0$ to have a single nonzero entry per column.
The penalty parameter $\lambda$ is chosen as before, namely  $\lambda = \|X- U_0V_0\|_F^2/\mbfr$. 
We run each algorithm 200 seconds for each data set. 
Table~\ref{tab:document} reports the clustering accuracy obtained by the algorithms, which is defined as follows. Given the true clusters, $C_k$ for $k=1,2,\dots,\mbfr$, and the clusters computed by an algorithm, $C'_k$ for $k=1,2,\dots,\mbfr$ (in ONMF, a data point is assigned to the cluster corresponding to the largest entry in the corresponding column of $V$), the accuracy of the algorithm is defined as 
\[
\text{Accuracy} \quad =  \max_{\pi, \text{ a permutation of $[\mbfr]$} } 
\frac{1}{ \mbfn }
\left( 
\sum_{k=1}^{\mbfr} \left| C_k \cap C'_{\pi(k)}\right|
\right). 
\] 
\begin{center}  
 \begin{table}[h!] 
 \caption{Accuracy in percent obtained by the different algorithms on 12 document data sets. 
 The best accuracy is highlighted in bold.  \label{tab:document}}
 \begin{center}  
 \begin{tabular}{|c|c|c|c|c|c|} 
 \hline Data set & rank $\mbfr$& BMM & BMME & A-BPALM & BIBPA \\ 
 \hline 
hitech&  6 &  \textbf{39.94}  &  {39.93}  &  38.98  &  37.07 \\ 
  reviews&  5  &  \textbf{73.56}  &  73.53  &  66.70  &  66.31 \\ 
  sports&  7  &  50.09  &  \textbf{50.13}  &  42.93  &  42.93 \\ 
  ohscal&  10  & \textbf{31.70}  &  31.52  &  27.25  &  27.25 \\ 
  la1&  6  &  49.86  &  \textbf{53.37} &  41.32  &  41.32 \\ 
  la2&  6  &  53.43  &  52.46  &  \textbf{54.83}  &  50.34 \\ 
  classic&  4  &  60.74  &  \textbf{61.43}  &  50.70  &  50.10 \\ 
  k1b&  6  &  \textbf{79.19}  &  \textbf{79.19}  & 71.41  &  71.41 \\ 
  tr11&  9 & 37.44  &  37.44  &  37.44  &  \textbf{41.30} \\ 
  tr23&  6  &  \textbf{41.67}  & \textbf{ 41.67}  &  \textbf{41.67}  &  40.20 
  \\ tr41&  10  &  \textbf{38.61}  &  \textbf{38.61}  &  \textbf{38.61}  &  35.08 \\ 
  tr45&  10  & 35.51  &  35.51  &  35.51  & \textbf{37.82} \\ \hline 
     average & & 49.31   & \textbf{49.57}  &  45.61 &   45.09 \\
  \hline 
  \end{tabular} 
  \end{center} 
 \end{table} 
 \end{center} 
 
 We observe on Table~\ref{tab:document} that BMM and BMME provide, on average, better clustering accuracies than A-BPALM and BIBPA. 
 In fact, in terms of accuracy, BMM performs similarly as BMME as both algorithms were able to converge within the allotted time (see Figure~\ref{fig:document} and \Cref{appendix_document}). 
 When A-BPALM or BIBPA have a better clustering accuracy, it is only by a small margin (less than 4\% in all cases), while 
 BMM and/or BMME sometimes outperform  A-BPALM and BIBPA; in particular, by 
  6.8\% for reviews, 
 7.2\% for sports, 
  12\% for la1, 
  10.7\% for classic, and 
  7.8\% for k1b.

The scaled objective function values with respect to time for the hitech and reviews data set are reported in Figure~\ref{fig:document}; the results for the other data sets are similar, and can be found in \Cref{appendix_document}. As before, BMME is the fastest, followed by BMM, A-BPALM and BIBPA (in that order). 
\begin{figure}
   \centering
   \begin{tabular}{cc}
           \includegraphics[width=0.45\textwidth]{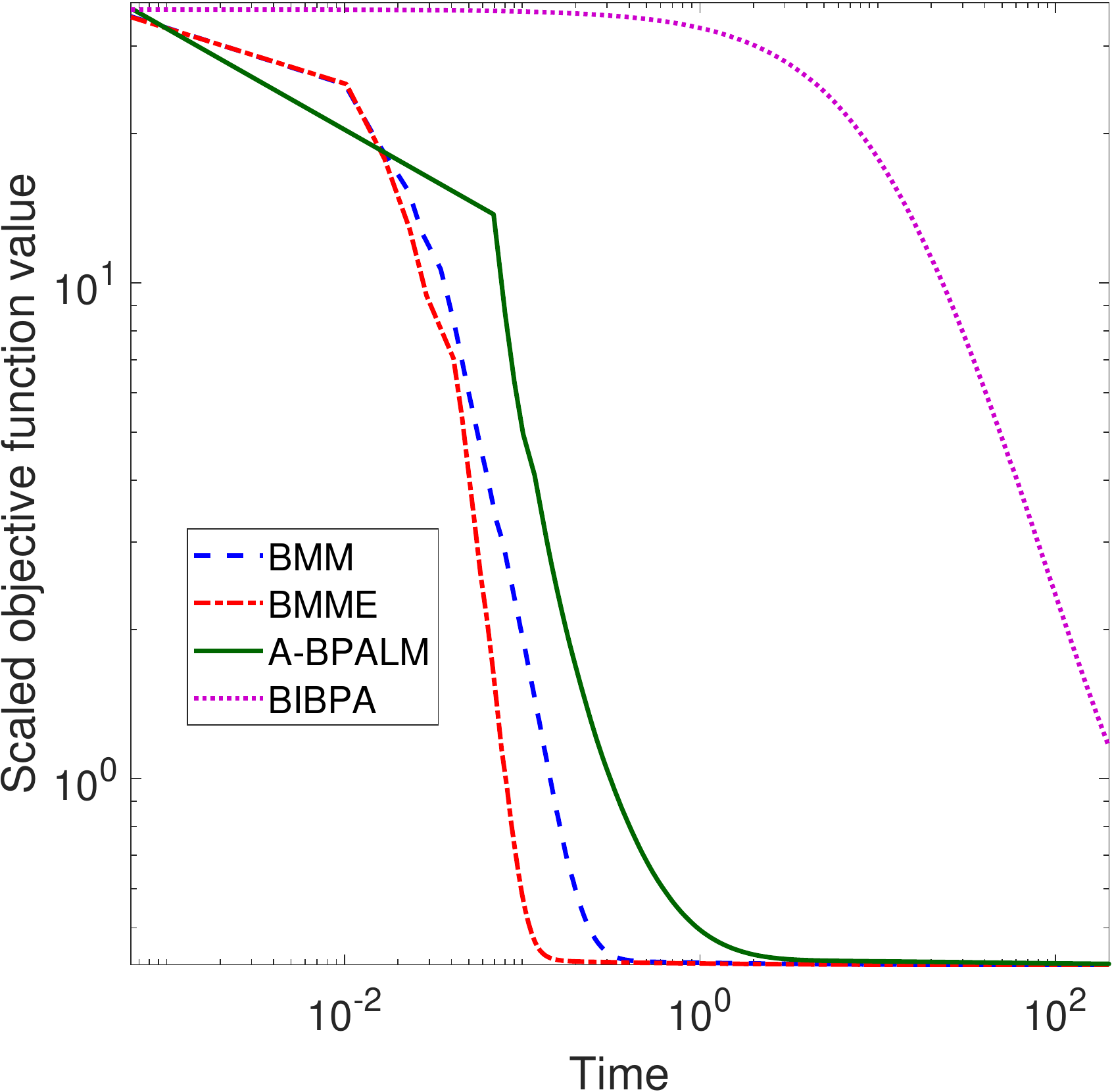}  & \includegraphics[width=0.45\textwidth]{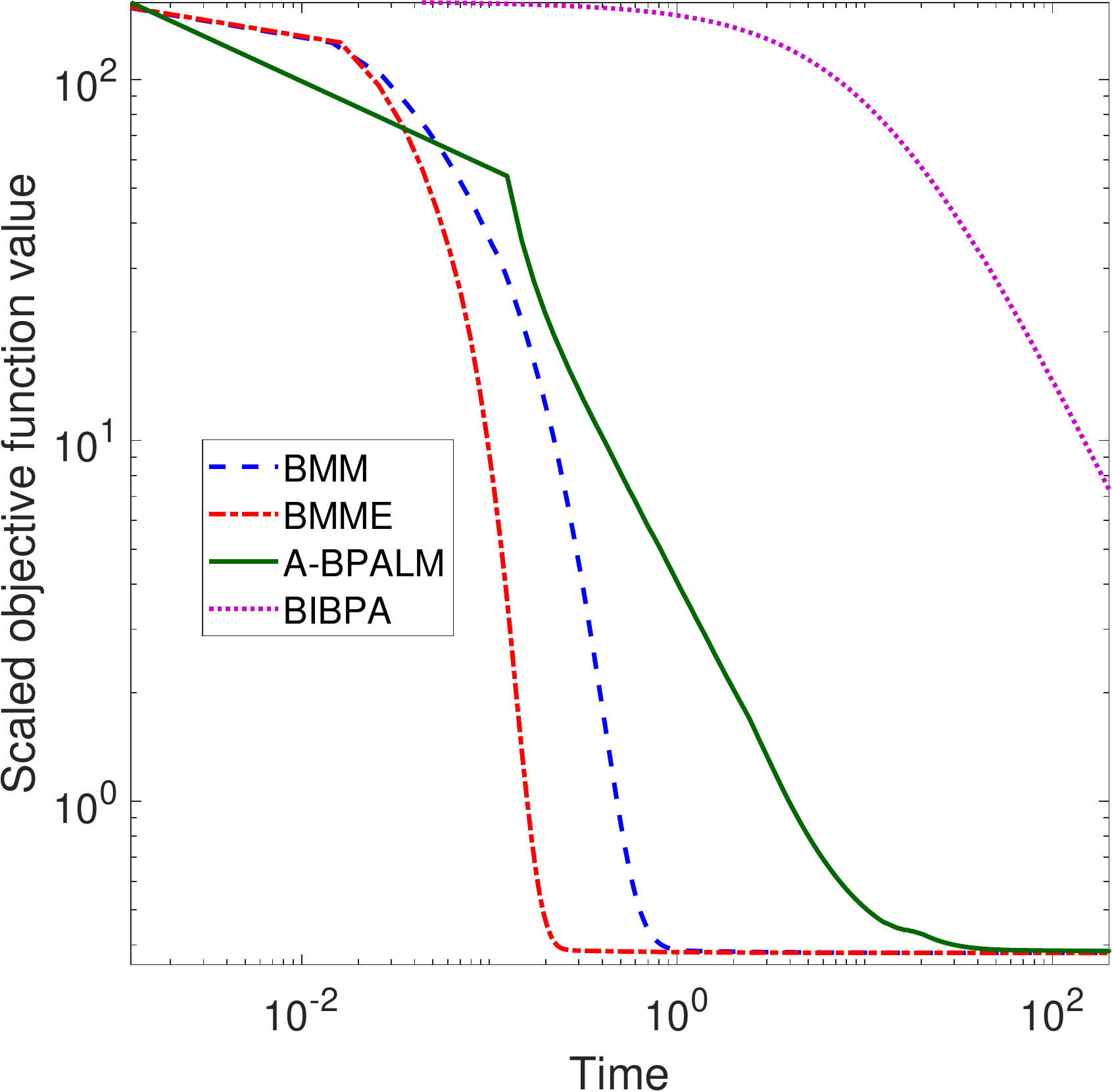} 
   \end{tabular}
           \caption{Evolution of the scaled objective function values with respect to running time  on the hitech (left) and review (right) data sets, in loglog scale. 
        \label{fig:document}
        }
\end{figure}

\section{Conclusion}
\label{sec:conclusion}

In this paper, we have developed BMME, a block alternating Bregman Majorization Minimization framework with Extrapolation that uses the Nesterov acceleration technique, for a class of nonsmooth nonconvex optimization problems that does not require the global Lipschitz gradient continuity. 
We have proved the subsequential and global convergence of BMME 
to first-order stationary points; see Theorems~\ref{theorem:subsconvergence}  and~\ref{theorem:globalconvergence}, respectively.  
We have evaluated the performance of BMME on the penalized orthogonal nonnegative matrix factorization problem on synthetic data sets, facial images, and documents. 
The numerical results have shown that 
(1)~Using extrapolation improves the convergence of BMME, and (2)~BMME converges faster than previously 
introduced the Bregman BPG methods, A-BPALM~\cite{ahookhosh2019multi} and BIBPA~\cite{ahookhosh2020_inertial}, 
because BMME allows a much more flexible choice of the kernel functions and uses Nesterov-type extrapolation. 


\appendix

\section{Preliminaries of nonconvex nonsmooth optimization}
 \label{sec:prelnnopt}

Let $g: \bbE\to \bbR\cup \{+\infty\} $ be a proper lower semicontinuous function.  
\begin{definition}
\label{def:dd}
\begin{itemize}
\item[(i)] For each $x\in{\rm dom}\,g,$ we denote $\hat{\partial}g(x)$ as
the Frechet subdifferential of $g$ at $x$ which contains vectors
$v\in\mathbb{E}$ satisfying 
\[
\liminf_{y\ne x,y\to x}\frac{1}{\left\Vert y-x\right\Vert }\left(g(y)-g(x)-\left\langle v,y-x\right\rangle \right)\geq 0.
\]
If $x\not\in{\rm dom}\:g,$ then we set $\hat{\partial}g(x)=\emptyset.$  
\item[(ii)] The limiting-subdifferential $\partial g(x)$ of $g$ at $x\in{\rm dom}\:g$
is defined as follows. 
\[
\partial g(x) := \left\{ v\in\mathbb{E}:\exists x^{k}\to x,\,g\big(x^{k}\big)\to g(x),\,v^{k}\in\hat{\partial}g\big(x^{k}\big),\,v^{k}\to v\right\} .
\]
\end{itemize}
\end{definition}
\begin{definition}
\label{def:type2}
We call $x^{*}\in \rm{dom}\,F$ a critical point of $F$ if $0\in\partial F\left(x^{*}\right).$ 
\end{definition}
We note that if $x^{*}$ is a local minimizer of $F$ then $x^{*}$ is a critical point of $F$. 
\begin{definition}
\label{def:KL}
A function $\phi(x)$ is said to have the KL property
at $\bar{x}\in{\rm dom}\,\partial\, \phi$ if there exists $\eta\in(0,+\infty]$,
a neighborhood $U$ of $\bar{x}$ and a concave function $\xi:[0,\eta)\to\mathbb{R}_{+}$
that is continuously differentiable on $(0,\eta)$, continuous at
$0$, $\xi(0)=0$, and $\xi'(s)>0$ for all $s\in(0,\eta),$ such that for all
$x\in U\cap[\phi(\bar{x})<\phi(x)<\phi(\bar{x})+\eta],$ we have
\begin{equation}
\label{ieq:KL}
\xi'\left(\phi(x)-\phi(\bar{x})\right) \, \dist\left(0,\partial\phi(x)\right)\geq1.
\end{equation}
$\dist\left(0,\partial\phi(x)\right)=\min\left\{ \|y\|:y\in\partial\phi(x)\right\}$.
If $\phi(x)$ has the KL property at each point of ${\rm dom}\, \partial\phi$ then $\phi$ is a KL function. 
\end{definition}
Many nonconvex nonsmooth functions in practical applications belong to the class of KL functions, for examples, real analytic functions, semi-algebraic functions, and locally strongly convex functions~\cite{Bochnak1998,Bolte2014}.




\section{Facial features extracted by the ONMF algorithms on the Frey and Umist facial images} \label{appendixfaces}

Figures~\ref{fig:Frey}, and~\ref{fig:Umist} display the facial features extracted by BMM, BMME, A-BPALM and BIBPA for the Frey and Umist facial images, respectively. 
\begin{figure}[ht!]
   \centering
   \begin{tabular}{cccc}
       \includegraphics[width=0.22\textwidth]{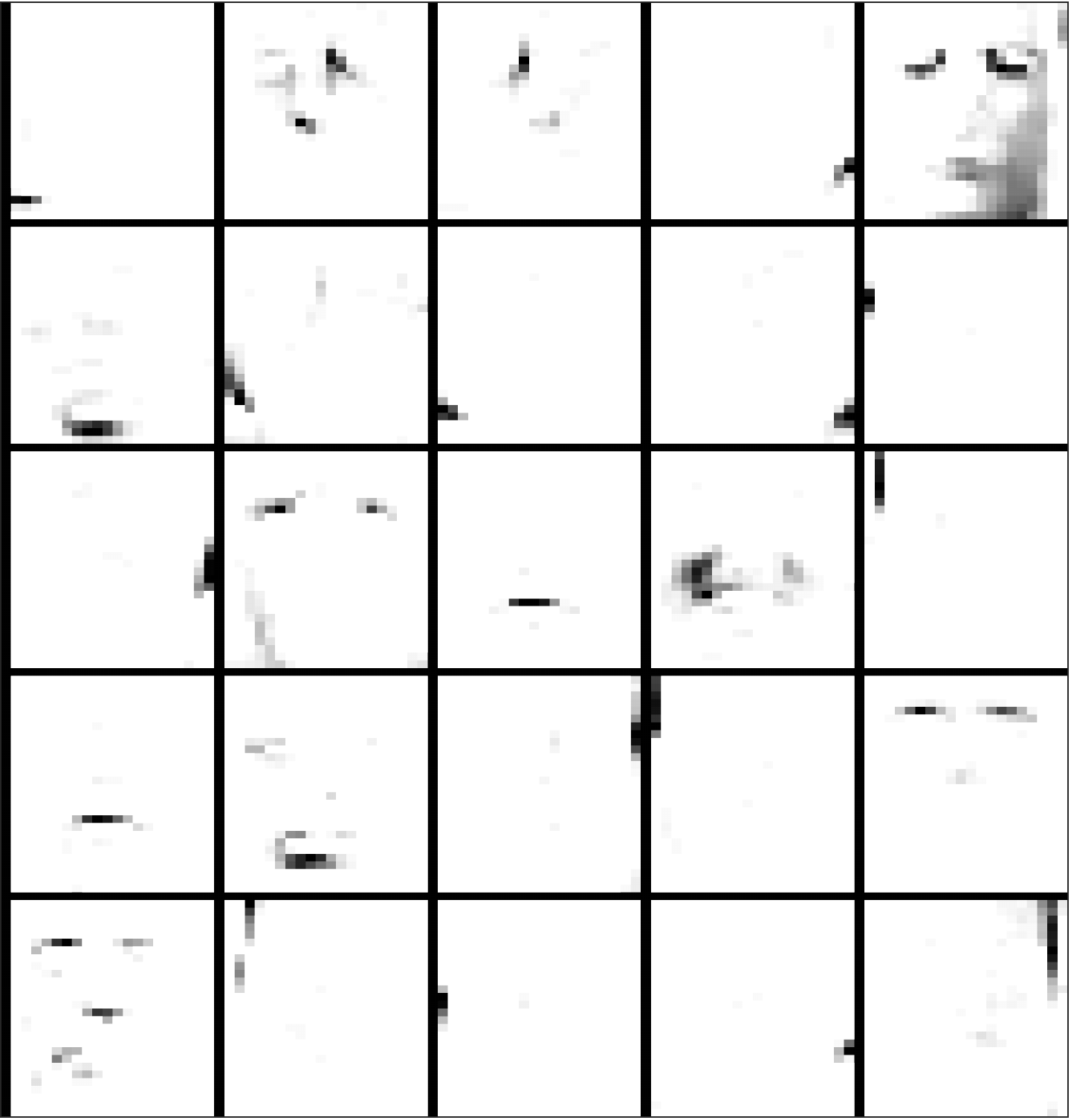}  & \includegraphics[width=0.22\textwidth]{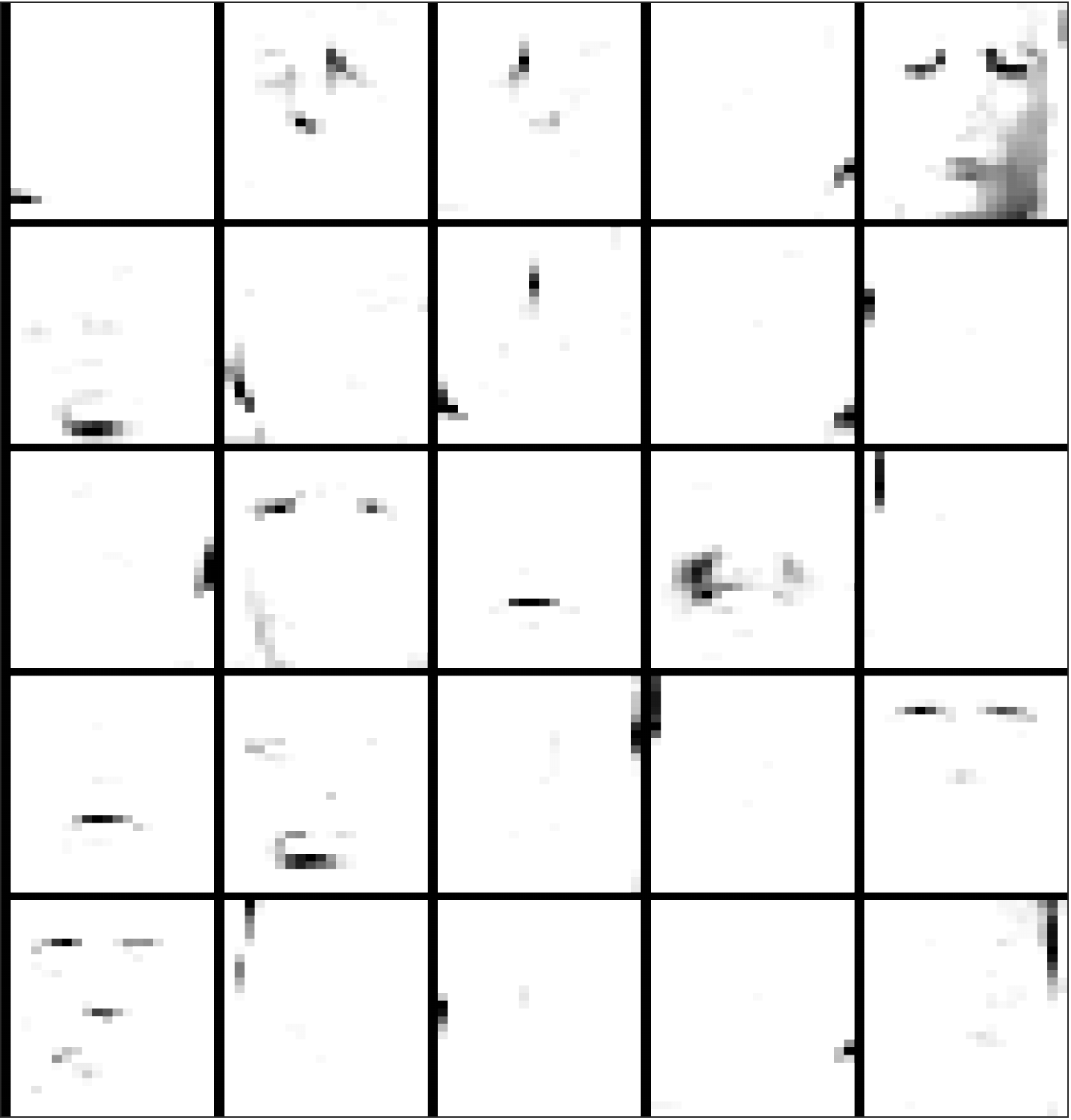}    
     & 
           \includegraphics[width=0.22\textwidth]{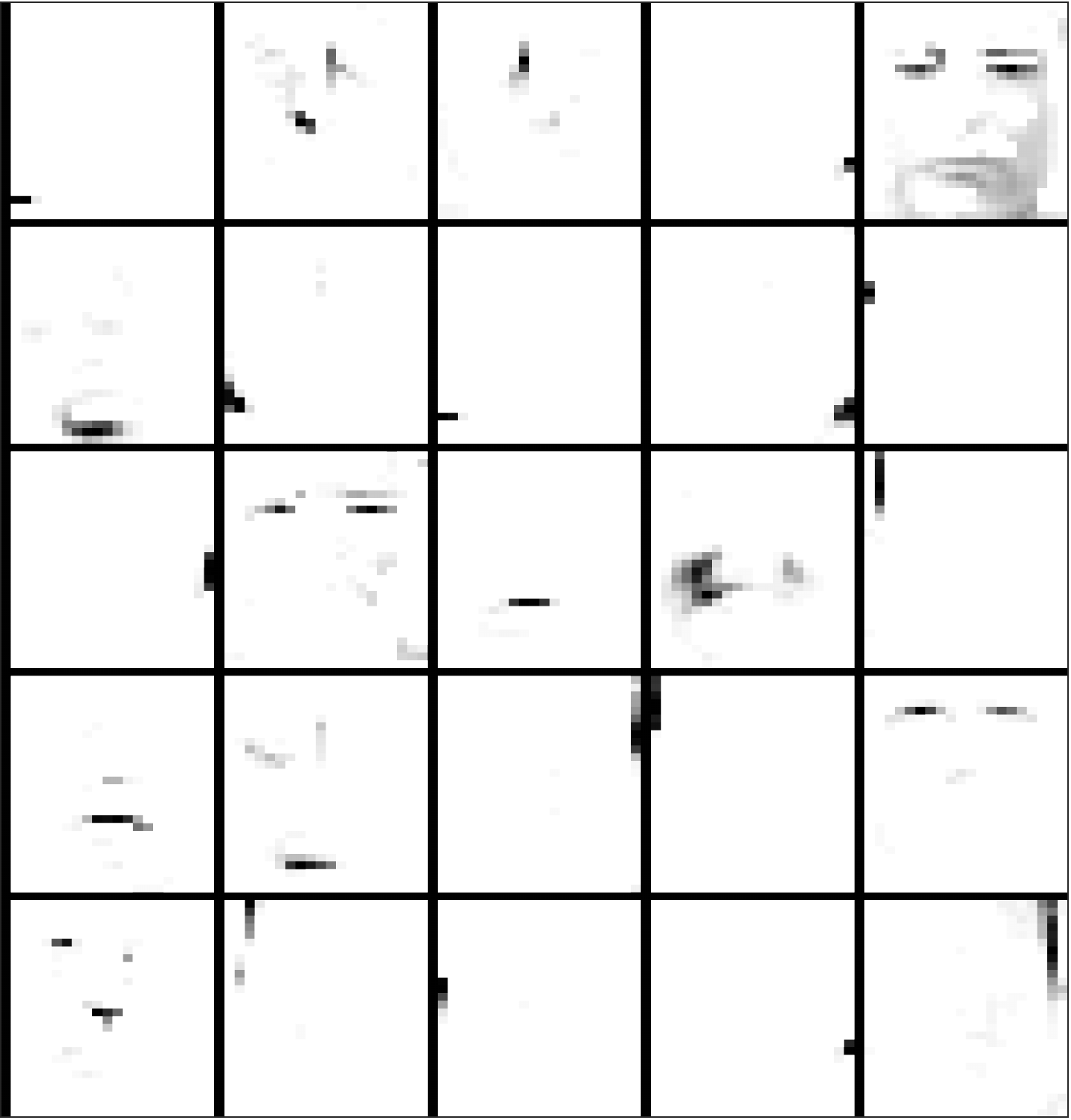}  & \includegraphics[width=0.22\textwidth]{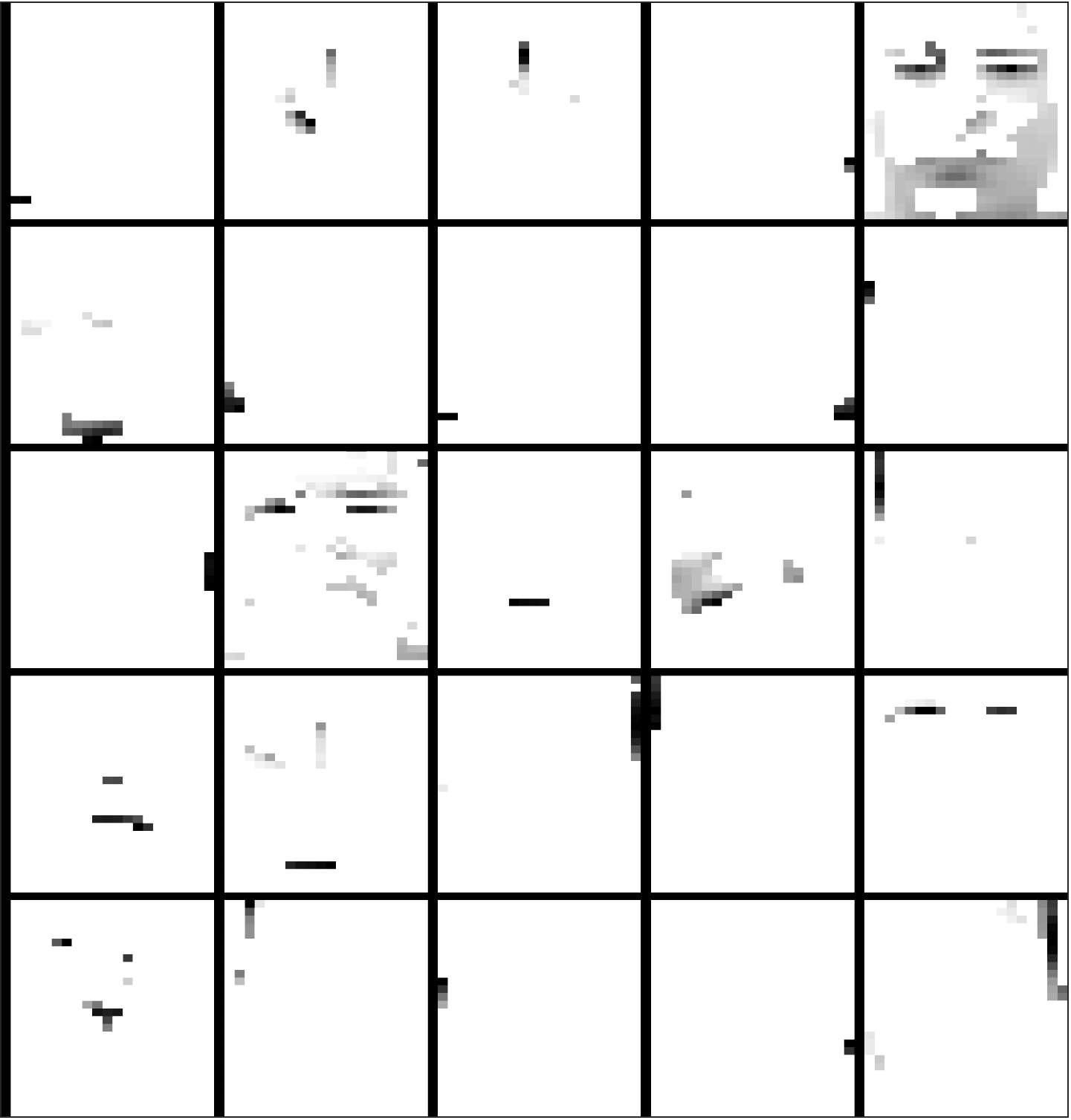} \\ 
           BMM & BMME & A-BPALM & BIBPA 
   \end{tabular}
        \caption{Display of the rows of the matrix $V$ as facial features, computed by BMM, BMME, A-BPALM, and BIBPA on the Frey facial images.}
    \label{fig:Frey}
\end{figure}
In Figure~\ref{fig:Frey}, facial features are rather similar, although BMM and BMME obtained smaller objective function values. 

\begin{figure}[ht!] 
   \centering
   \begin{tabular}{cccc}
       \includegraphics[width=0.22\textwidth]{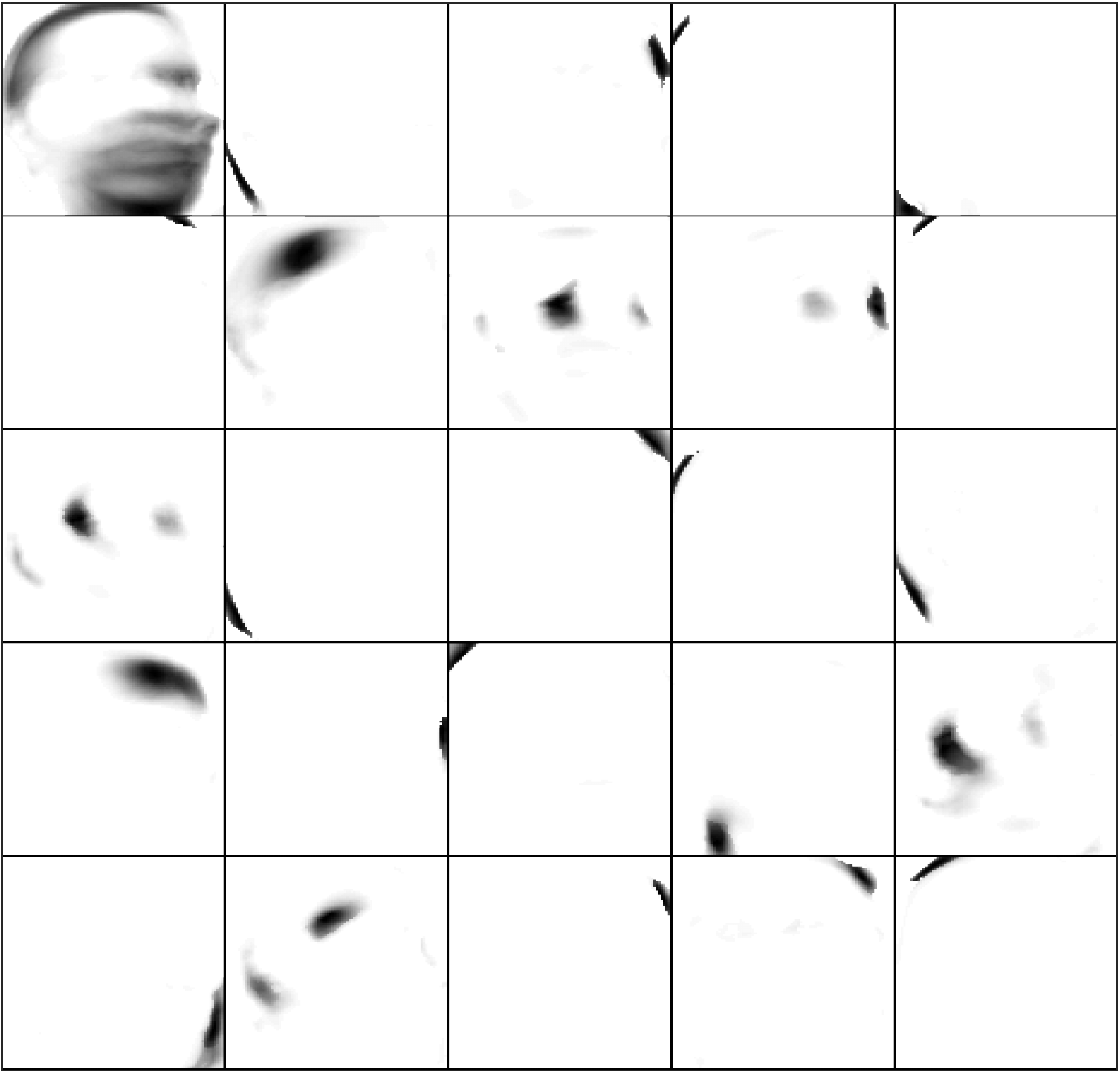} 
       & \includegraphics[width=0.22\textwidth]{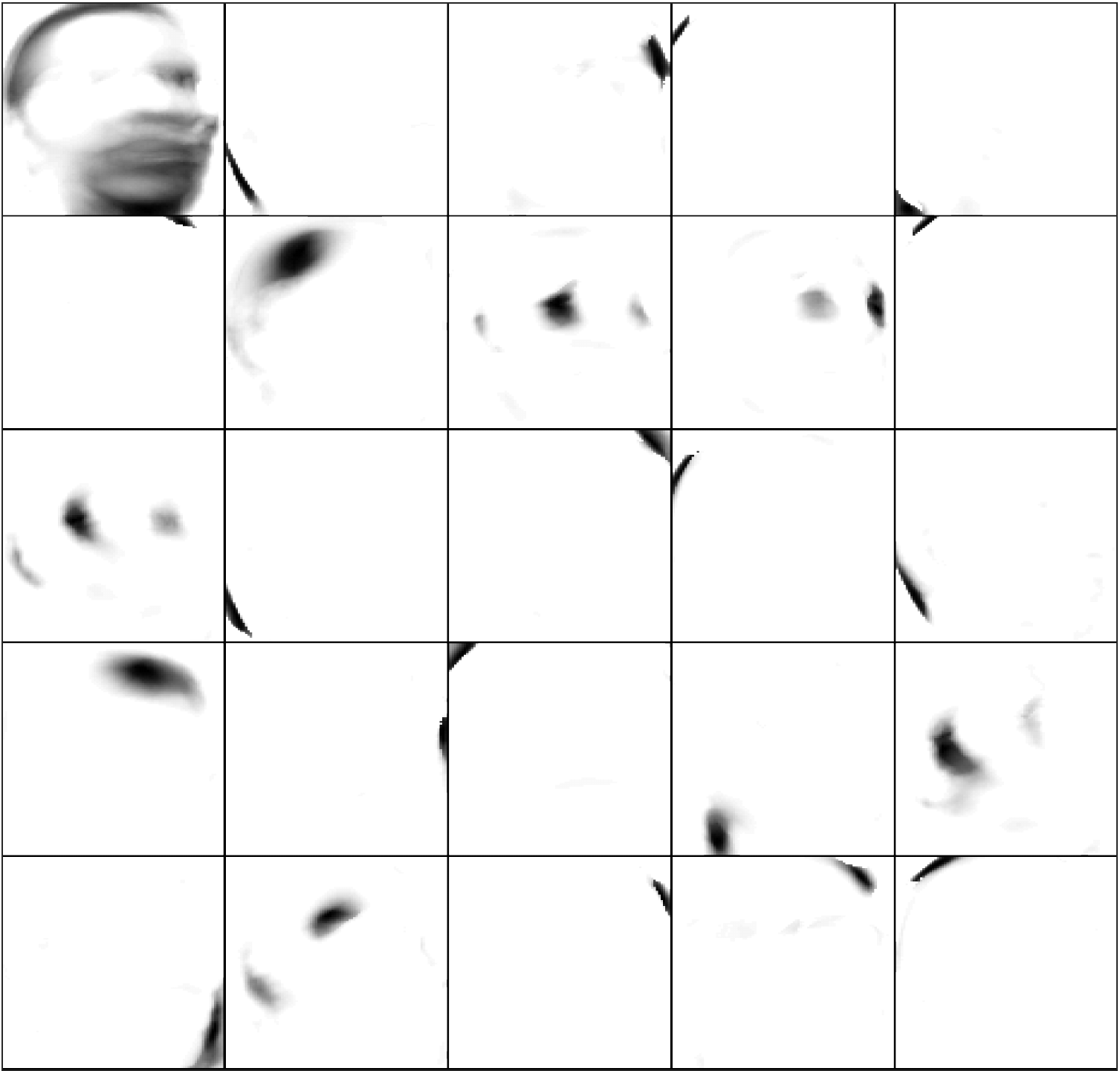}   
     &  \includegraphics[width=0.22\textwidth]{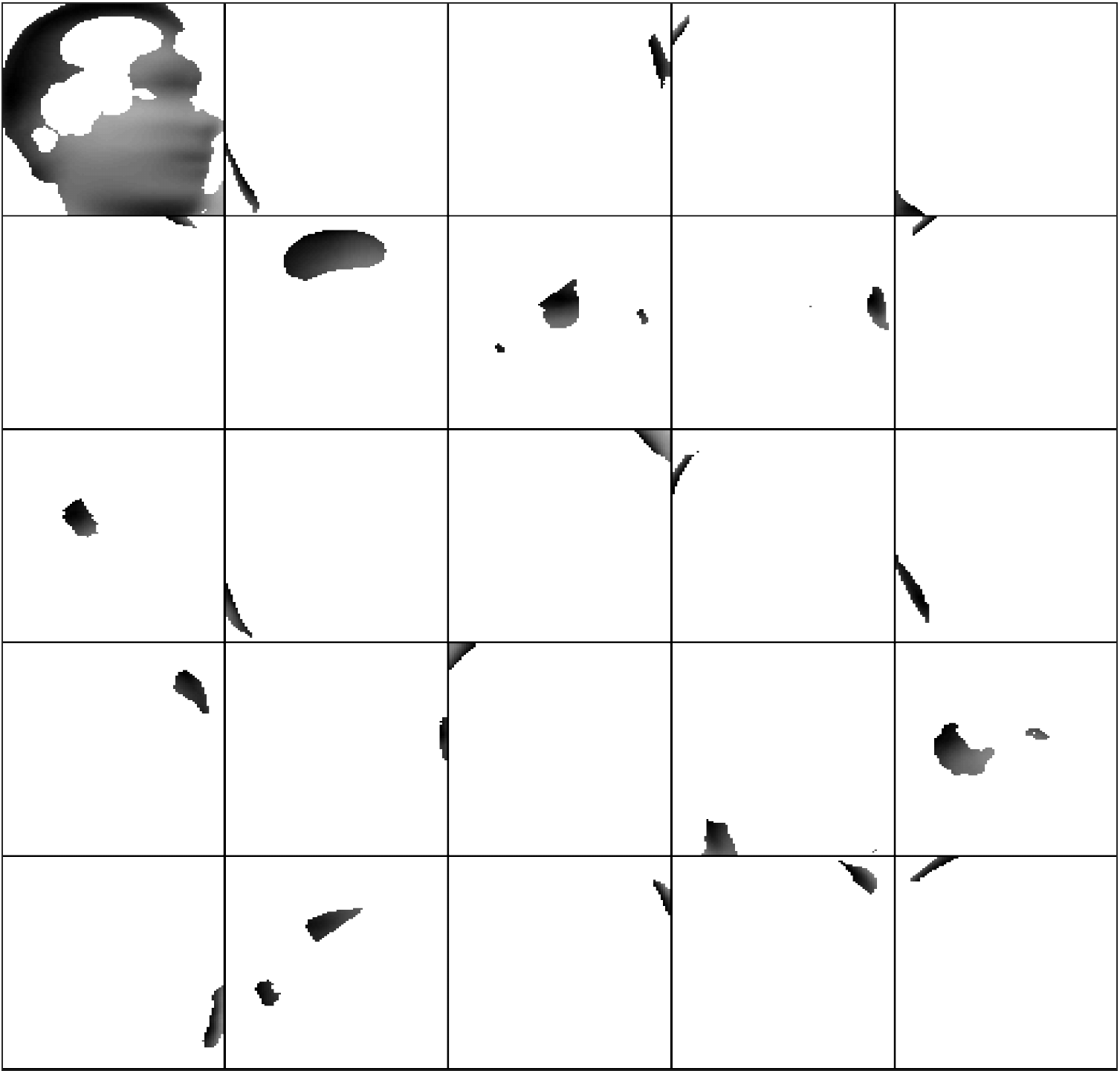}  & \includegraphics[width=0.22\textwidth]{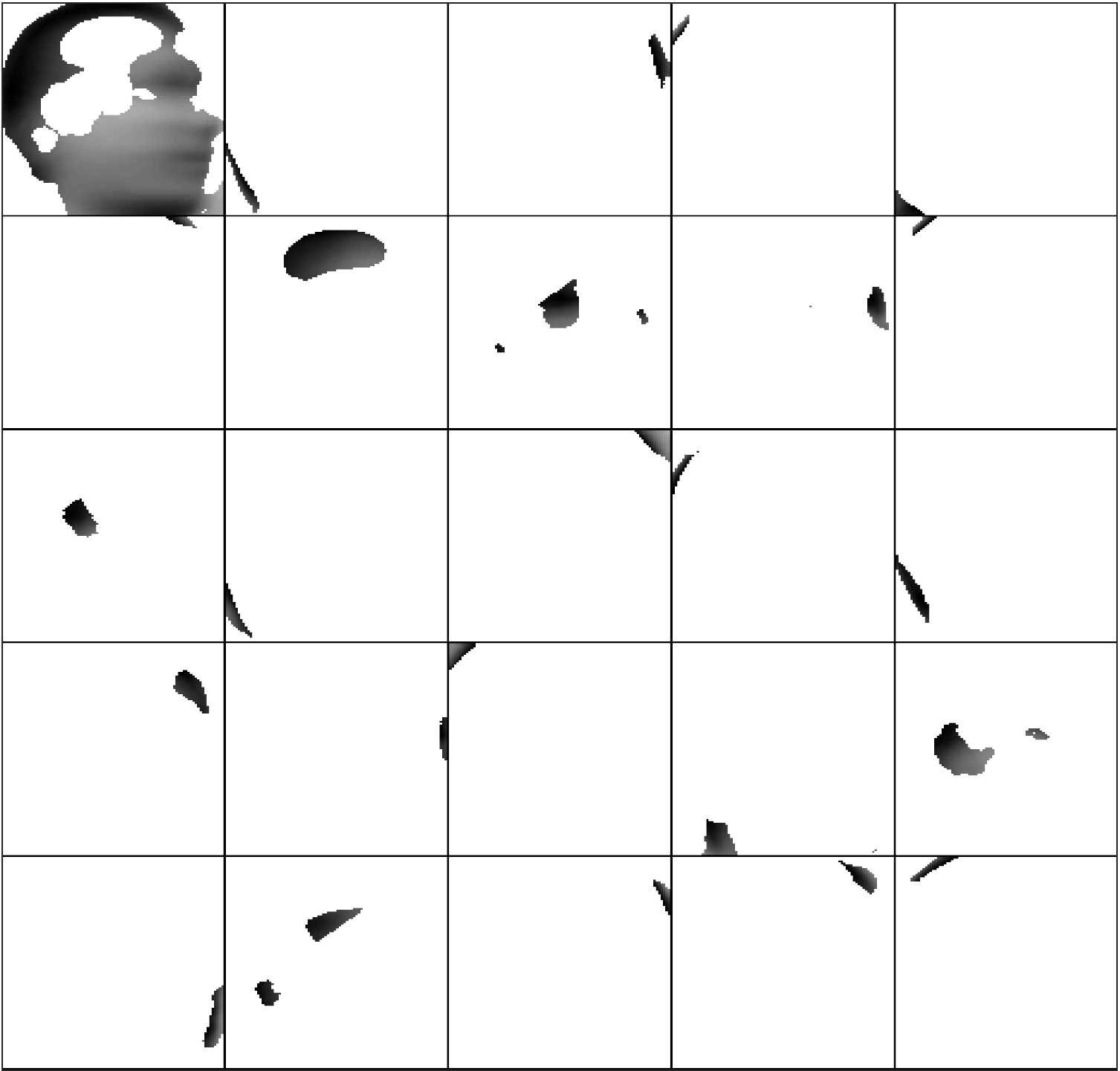} \\ 
          BMM & BMME & A-BPALM & BIBPA 
   \end{tabular}
        \caption{Display of the rows of the matrix $V$ as facial features, computed by BMM, BMME, A-BPALM, and BIBPA on the Umist facial images.}
    \label{fig:Umist}
\end{figure}
In Figure~\ref{fig:Umist}, as BMM and BMME both converged to similar objective function values (see Figure~\ref{fig:image}), they provide very similar facial features.  
A-BPALM and BIBPA were not able to converge within the 100 seconds, and hence provide worse facial features. For example, the first facial feature is much denser than for  BMM and BMME, overlapping with other facial features (meaning that the orthogonality constraints is not well satisfied).

\section{Scaled objective function values for document data sets} \label{appendix_document}

Figures~\ref{fig:document2} and~\ref{fig:document3} display the scaled objective function values for the document data sets on the penalized ONMF problem. We observe that, except for tr41 and tr45 where A-BPALM is able to compete with BMM and BMME, 
BMM and BMME outperform A-BPALM and BIBPA which performs particularly badly on these sparse data sets. 
\begin{figure}
   \centering
   \begin{tabular}{cc}
   sports & ohscal \\ 
       \includegraphics[width=0.45\textwidth]{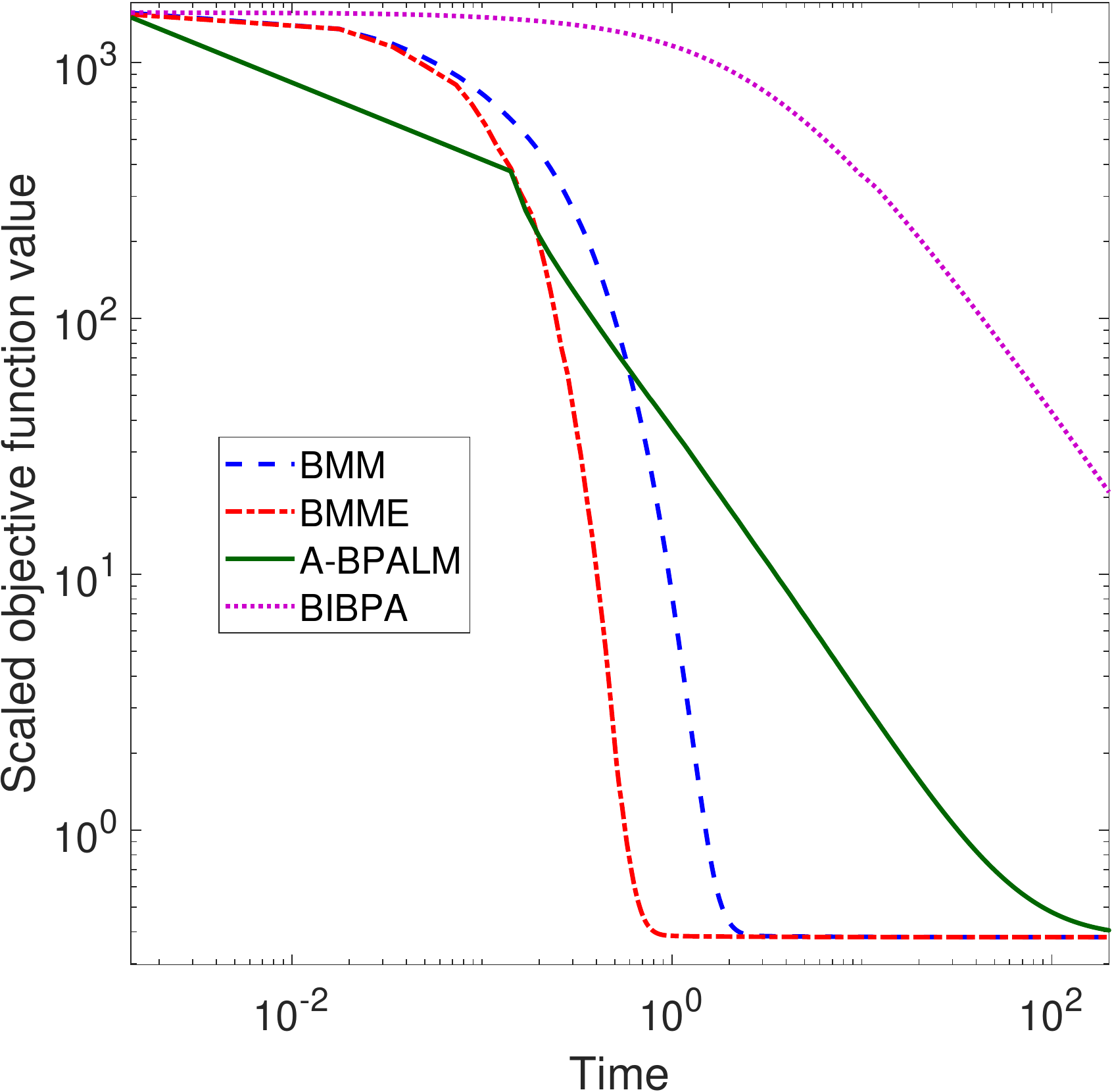}  & \includegraphics[width=0.45\textwidth]{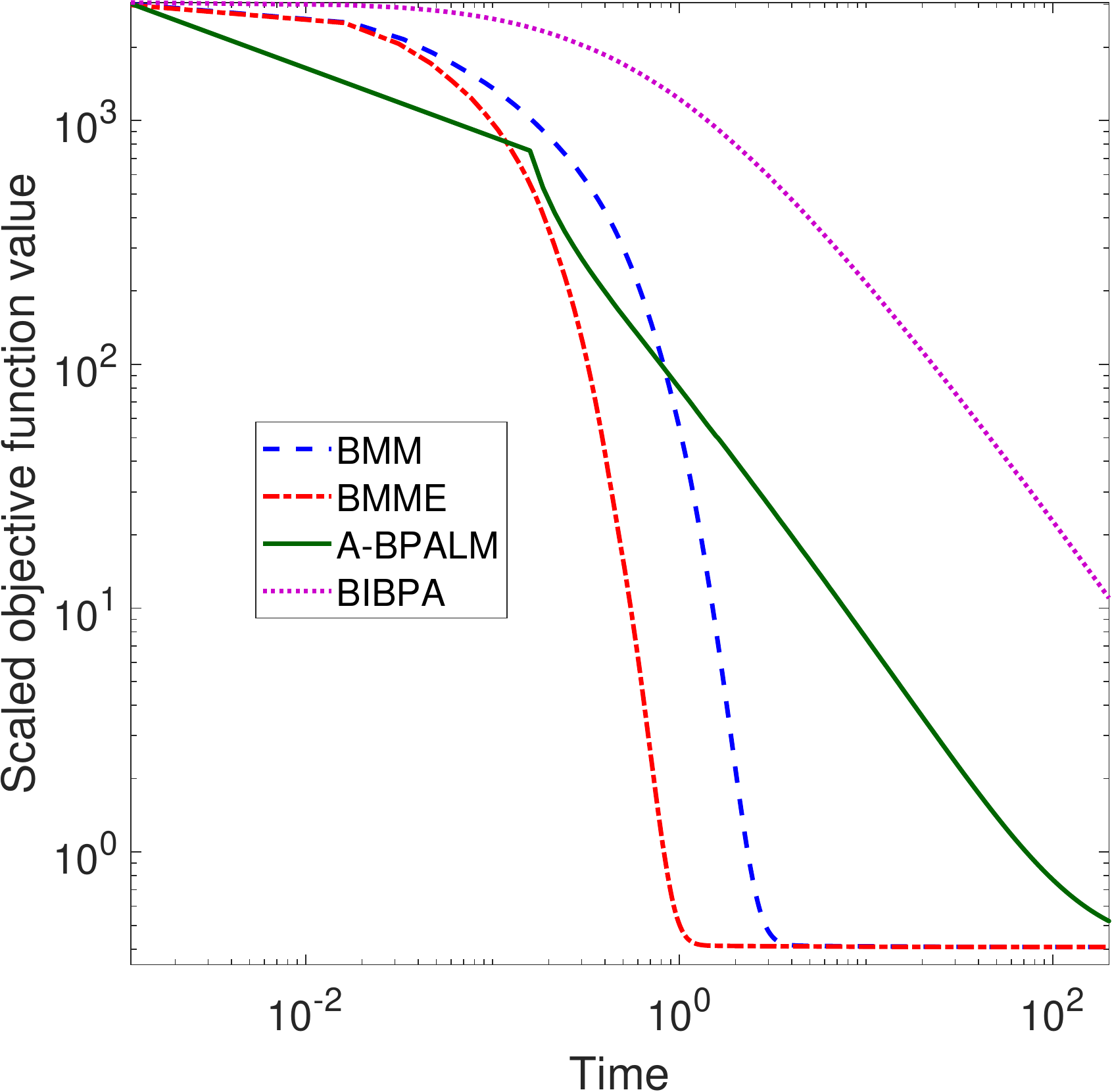}    
         \vspace{0.2cm} \\
         la1 & la2 \\ 
           \includegraphics[width=0.45\textwidth]{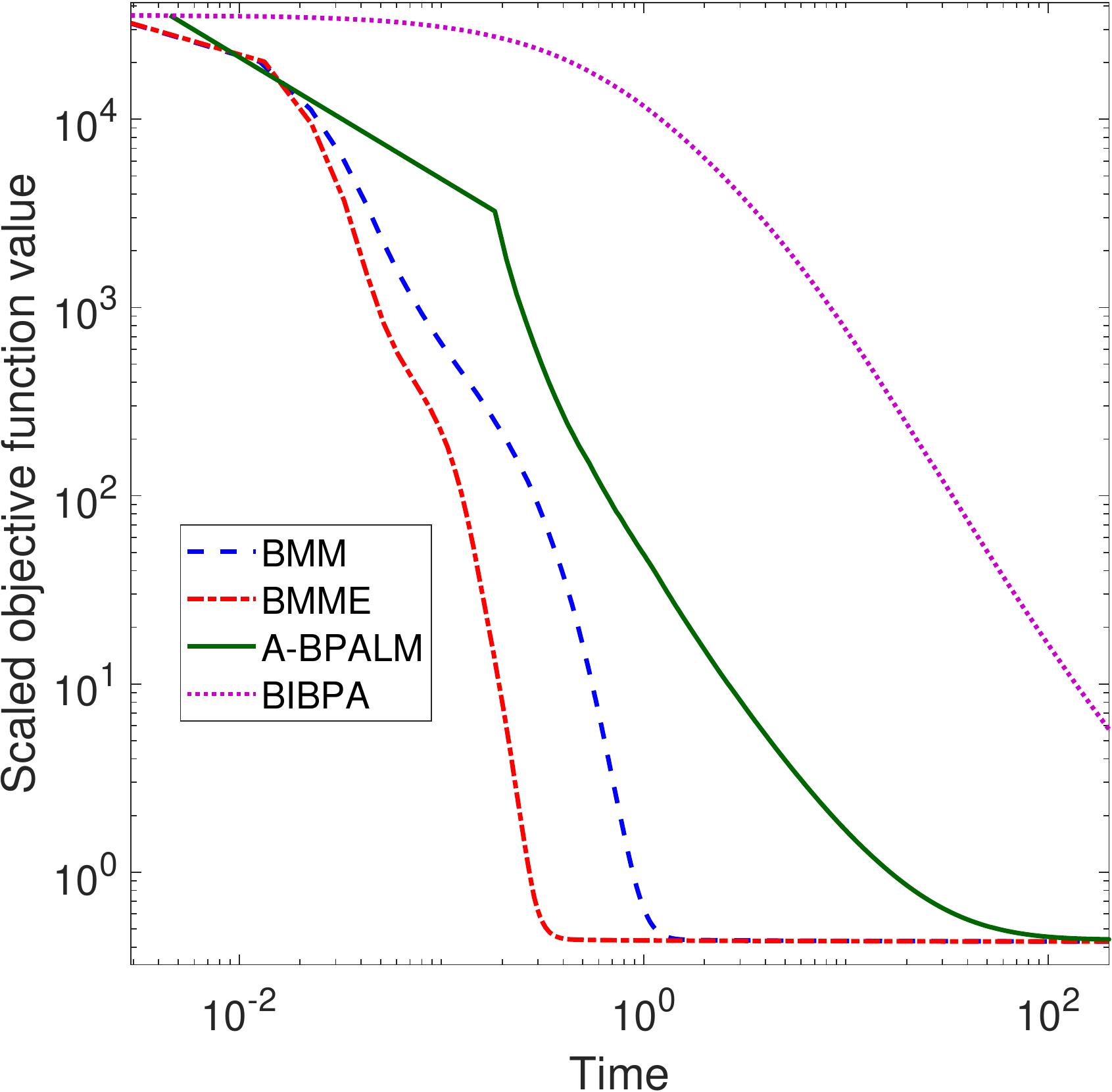}  & \includegraphics[width=0.45\textwidth]{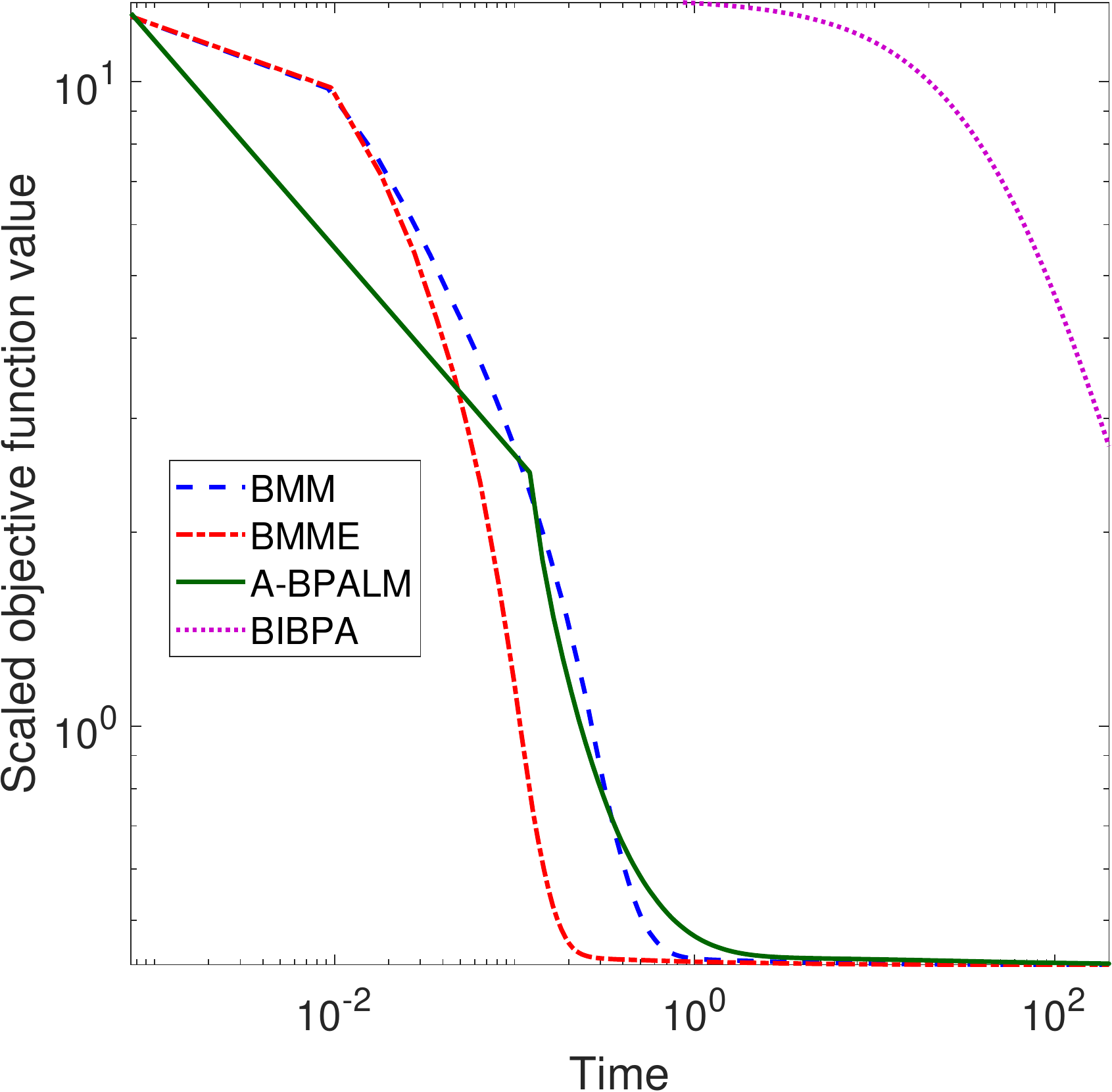}   
   \end{tabular}
        \caption{Evolution of scaled objective function values with respect to time on document data sets.} 
    \label{fig:document2}
\end{figure}

\begin{figure}
   \centering
   \begin{tabular}{cc}
              classic & k1bl \\ 
           \includegraphics[width=0.45\textwidth]{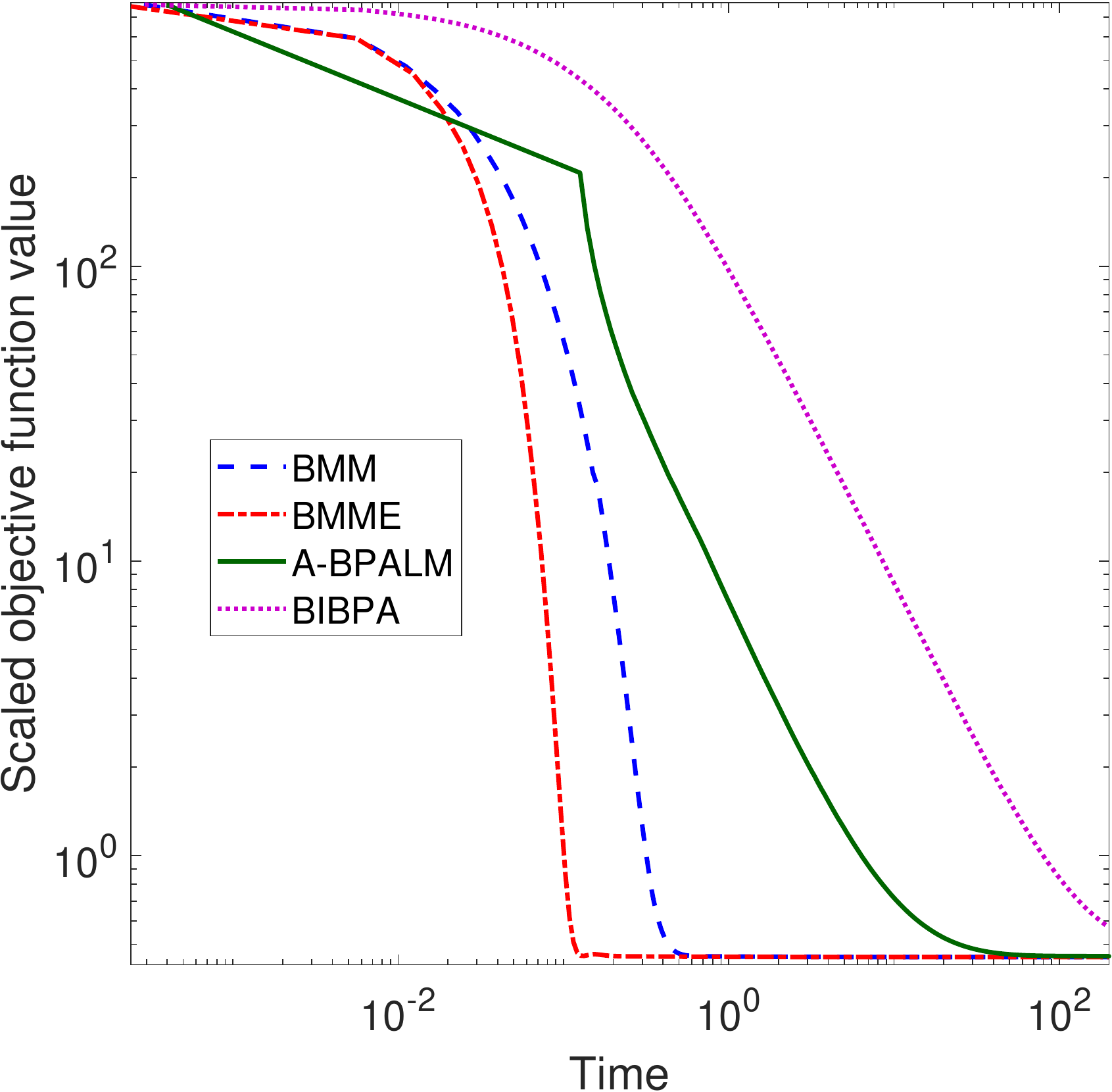}  & \includegraphics[width=0.45\textwidth]{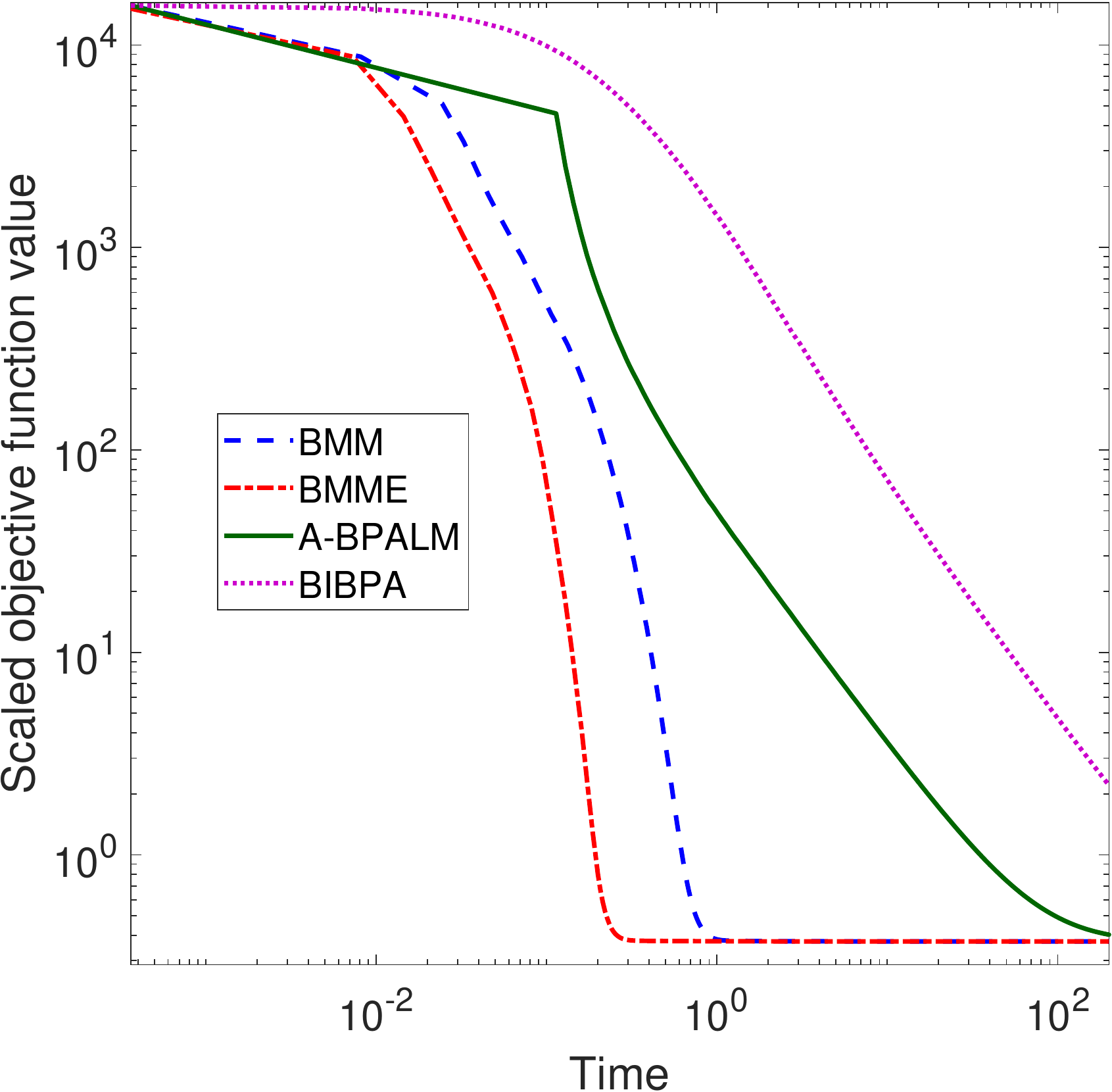}   
            \vspace{0.2cm}  \\
            tr11 & tr23 \\
           \includegraphics[width=0.45\textwidth]{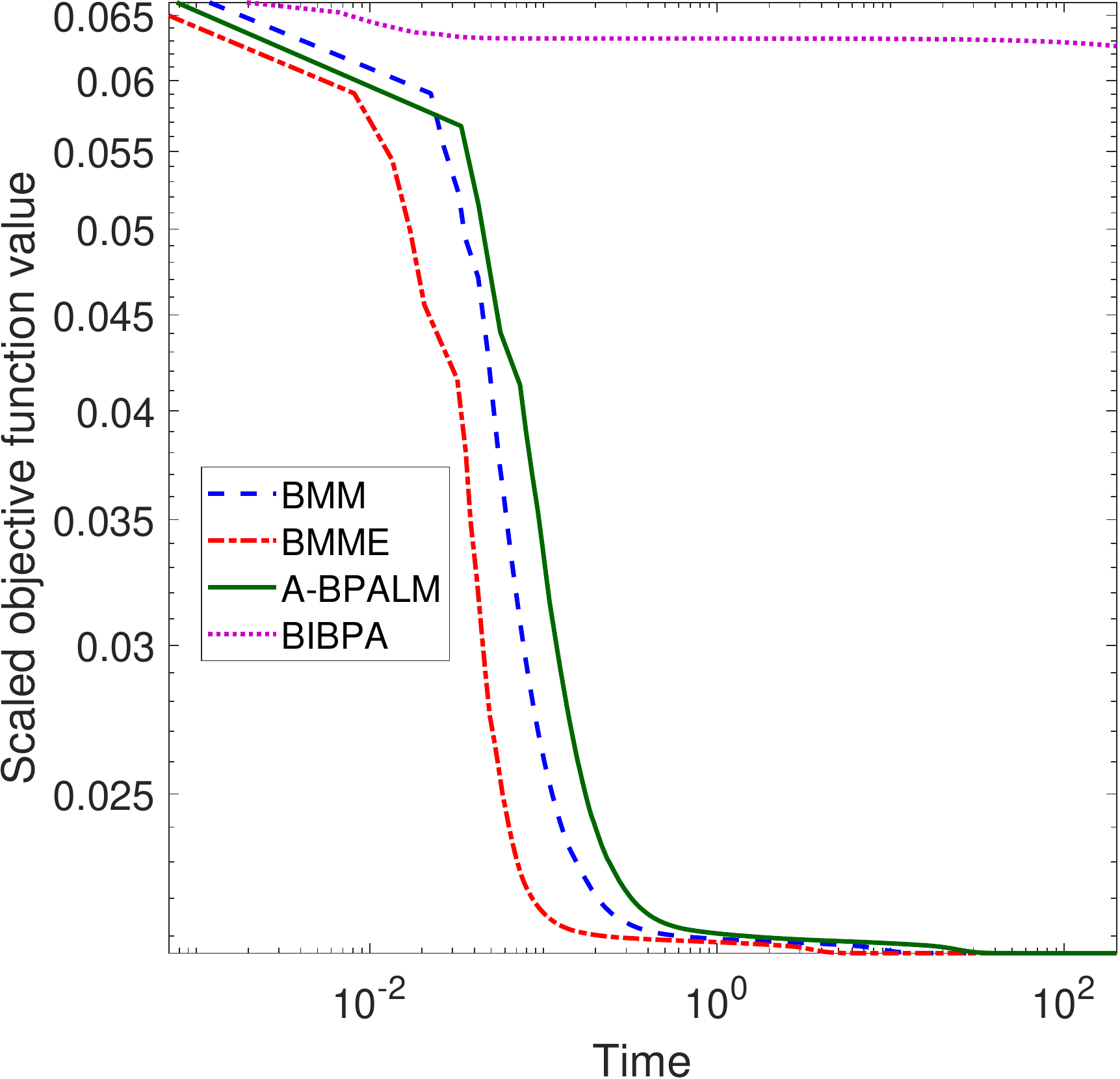}  & \includegraphics[width=0.45\textwidth]{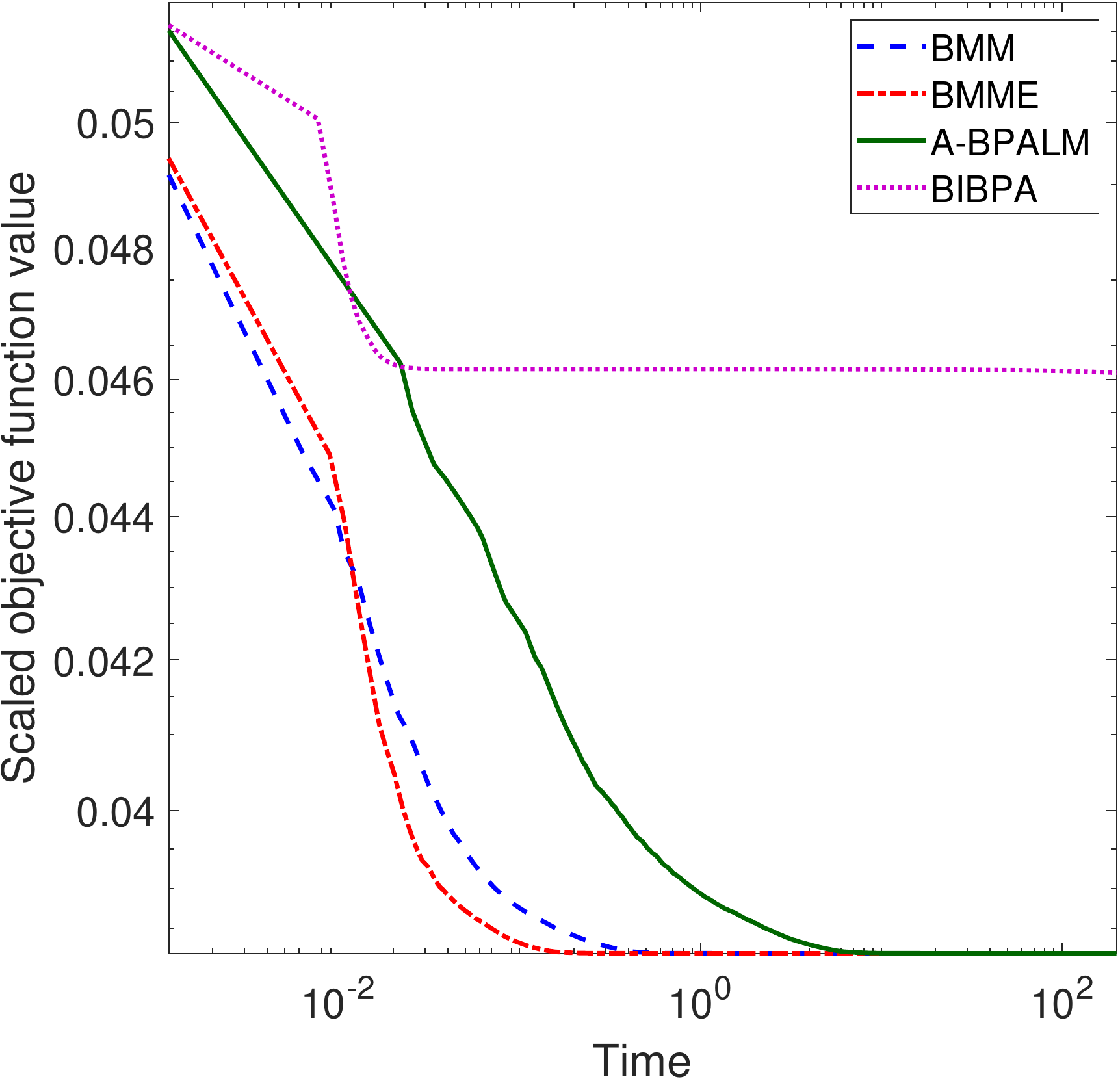}    
         \vspace{0.2cm}\\ 
          tr41 & tr45 \\ 
           \includegraphics[width=0.45\textwidth]{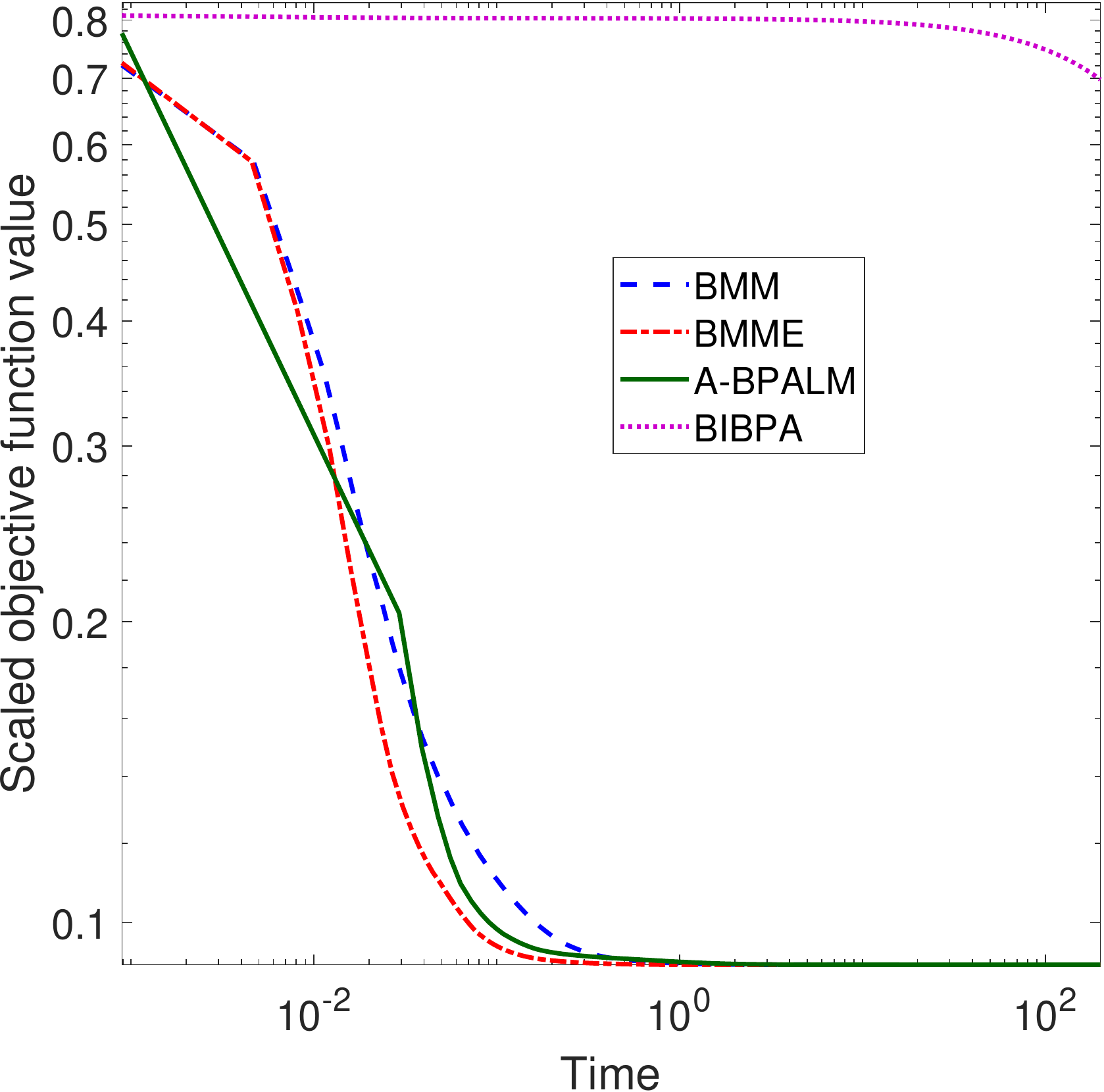}  & \includegraphics[width=0.45\textwidth]{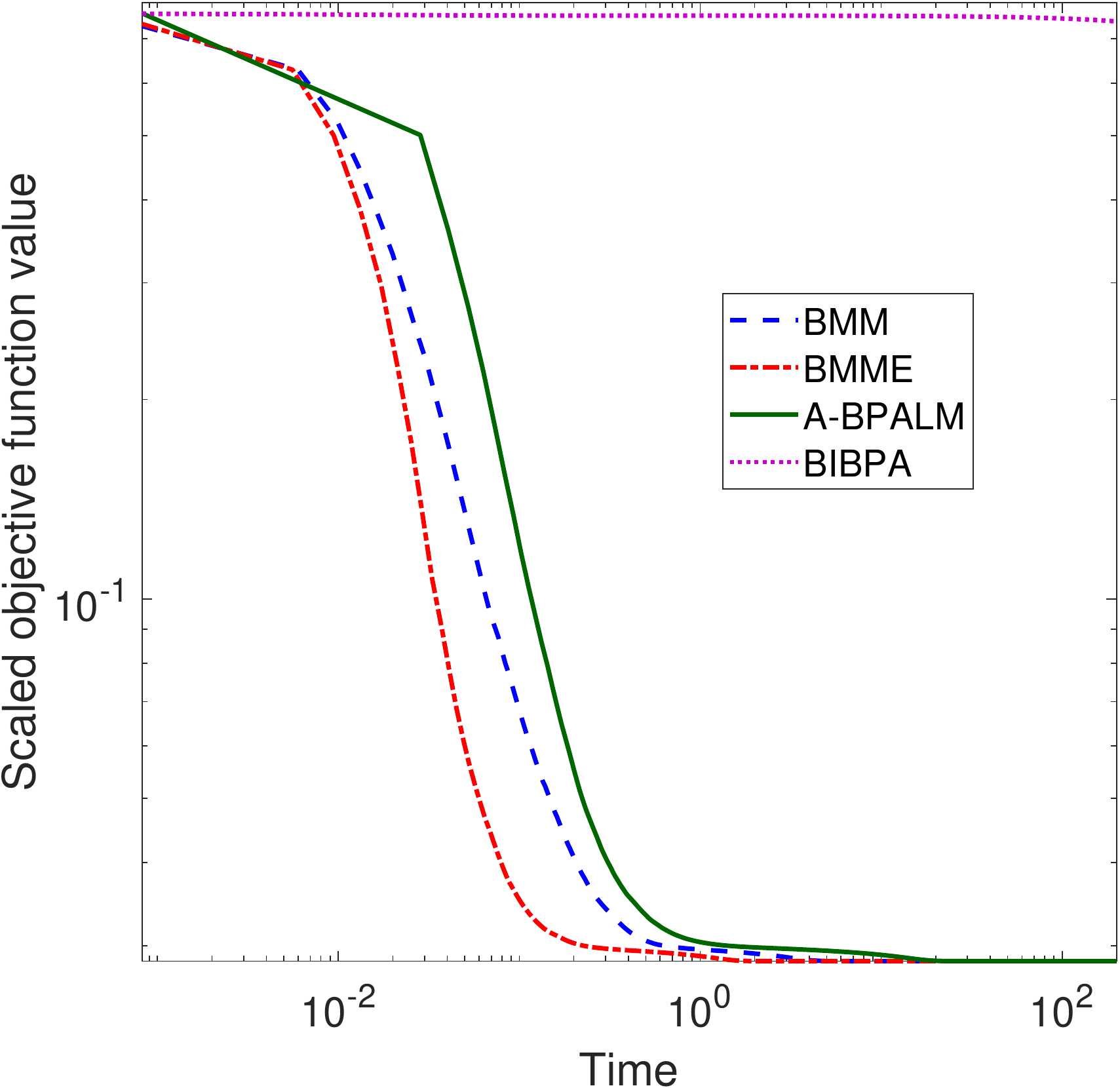}   
   \end{tabular}
        \caption{Evolution of scaled objective function values with respect to time on document data sets.} 
    \label{fig:document3}
\end{figure}

\section{Comparison between BMME and CoCaIn on the matrix completion problem}\label{appendix_mcp}

In this section, we consider BMME for solving the CSOP~\eqref{model} with $m=1$. As $m=1$, we can omit the index $i$. In addition, the relative smooth parameters and the kernel generating distance do not depend on $k$. Therefore, the condition \eqref{eq:beta} can be rewritten as follows
\begin{equation*}
        D_{\varphi}(x^k,\bar{x}^k) \leq \frac{\delta L}{L + l} D_{\varphi}(x^{k-1},x^k),
   \end{equation*}
and the update \eqref{eq:iMM_update} becomes
\begin{equation*}
       \begin{split}
           x^{k+1}\in\argmin_{x\in\mathcal X}\biggl\{LD_{\varphi}(x,\bar{x}^k) + \biggl\langle \nabla f(\bar{x}^k),x\biggr\rangle + u(x,x^k) \biggr\}.
       \end{split}
   \end{equation*}
In some applications, the constants $L,l$ might be very large, leading to a slow convergence. 
Hence, like CoCaIn~\cite{Mukkamala2020} we incorporate a backtracking line search for $L$ and $l$ into BMME. In particular, BMME with backtracking computes the extrapolation point $\bar{x}^k = x^k + \beta^k(x^k - x^{k-1}) \in {\rm int \, dom\,} \varphi$, where 
$\beta^k$ satisfies the following condition 
\begin{equation*}
     D_{\varphi}(x^k,\bar{x}^k) \leq \frac{\delta L^{k-1}}{L^{k-1} + l^k} D_{\varphi}(x^{k-1},x^k),
\end{equation*}
where $l^k$ is updated via backtracking such that
\begin{equation*}
    D_f(x^k,\bar{x}^k) \geq -l^k D_{\varphi}(x^k,\bar{x}^k).
\end{equation*}
The update $x^{k+1}$ is computed by solving the following convex nonsmooth sub-problem
\begin{equation*}
    \min_{x\in\mathcal X}\biggl\{L^kD_{\varphi}(x,\bar{x}^k) + \biggl\langle \nabla f(\bar{x}^k),x\biggr\rangle + u(x,x^k) \biggr\},
\end{equation*}
where $L^k$ is chosen via backtracking such that $L^k\geq L^{k-1}$ and 
\begin{equation*}
    D_f(x^{k+1},\bar{x}^k) \leq L^k D_{\varphi}(x^{k+1},\bar{x}^k).
\end{equation*}

We now conduct an additional experiment on the following matrix completion problem (MCP) to demonstrate the advantages of using proper convex surrogate functions:
\begin{equation}
\label{MF}
    \min_{U\in\mathbb{R}^{\mbfm\times \mbfr},V\in\mathbb{R}^{\mbfr\times \mbfn}} \biggl\{ \frac{1}{2}\|\mathcal P(A - UV)\|_F^2 + g(U,V)\biggr\},
\end{equation}
where $A\in\mathbb R^{\mathbf m \times \mathbf n}$ is a given data matrix, $\mathcal P(Z)_{ij}$ is equal to $Z_{ij}$ if $A_{ij}$ is observed and is equal to $0$ otherwise, and $g$ is a regularization term. Here, we are interested in an exponential regularization $g$ defined by
\begin{equation*}
    g(U,V) =  \lambda\Big(\sum_{ij}\big(1-\exp(-\theta |U_{ij}|)\big) + \sum_{ij}\big(1-\exp(-\theta |U_{ij}| )\big) \Big),
\end{equation*}
where $\lambda$ and $\theta$ are tuning parameters. We consider the problem \eqref{MF} as the form of \eqref{model} with $m=1$, $\mathcal X_1 = \mathcal X = \mathbb{R}^{\mbfm\times \mbfr}\times\mathbb{R}^{\mbfr\times \mbfn}$, $f(U,V) = \frac{1}{2}\|\mathcal P(A - UV)\|_F^2$, and $g_1(U,V) = g(U,V)$. 
We now investigate BMME for solving the problem \eqref{MF} by choosing a kernel generating distance $\varphi$ given by
\begin{equation*}
    \varphi(U,V) = c_1\biggl(\frac{\|U\|^2_F + \|V\|^2_F}{2}\biggr)^2 + c_2\frac{\|U\|^2_F + \|V\|^2_F}{2},
\end{equation*}
where $c_1 = 3$ and $c_2 = \|\mathcal P(A)\|_F$. In \cite{MukkamalaO19}, the authors showed that $f$ is $(L,l)$-relative smooth to $\varphi$ for all $L,l\geq 1$. BMME iteratively chooses a convex surrogate function $u$ of $g$ as follows:
\begin{equation*}
    u(U,V,U^k,V^k) =g(U^k,V^k) + \langle W^{U^k}, |U|\rangle + \langle W^{V^k}, |V|\rangle,
\end{equation*}
where $W^{U^k}_{ij} = \lambda\theta\exp(-\theta|U_{ij}^k|)$. 
BMME with backtracking updates $(U^{k+1}, V^{k+1})$ by solving the following convex nonsmooth sub-problem
\begin{equation}\label{sub:MF}
    \min_{U,V}\biggl\{u(U,V,U^k,V^k) + \langle P^k,U\rangle + \langle Q^k, V\rangle + L^k\varphi(U,V)\biggr\},
\end{equation}
where $P^k = \nabla_Uf(\bar{U}^k,\bar{V}^k) - L^k\nabla_U\varphi(\bar{U}^k,\bar{V}^k)$, $Q^k = \nabla_Vf(\bar{U}^k,\bar{V}^k) - L^k\nabla_V\varphi(\bar{U}^k,\bar{V}^k)$, and $L^k$ is chosen via a backtracking line search.
The solution to the problem \eqref{sub:MF} is defined by $U^{k+1} = -\tau^*\mathcal{S}(P^k,W^{U^k})/L^k$ and $V^{k+1} = -\tau^*\mathcal{S}(Q^k,W^{V^k})/L^k$, where $\mathcal S(A,B)_{ij} = [|A_{ij}|-B_{ij}]_+\text{sign}(A_{ij})$, and $\tau^*$ is the unique positive real root of
\begin{equation*}
    c_1\biggl(\|\mathcal{S}(P^k,W^{U^k})/L^k\|^2_F + \|\mathcal{S}(Q^k,W^{V^k})/L^k\|^2_F\biggr)\tau^3 + c_2\tau - 1 = 0.
\end{equation*}

Since CoCaIn \cite{Mukkamala2020} does not use the MM step and requires the weakly convexity of $g$, it is different from BMME for updating $(U^{k+1}, V^{k+1})$ and initializing $L^0$. In particular, CoCaIn iteratively solves the following nonconvex sub-problem
\begin{equation*}
    \min_{U,V}\biggl\{g(U,V) + \langle P^k,U\rangle + \langle Q^k, V\rangle + L^k\varphi(U,V)\biggr\},
\end{equation*}
which does not have closed-form solutions. We therefore employ an MM scheme to solve this sub-problem. For initializing the step-size, CoCaIn requires $L^0>\frac{\lambda\theta^2}{(1-\delta - \epsilon)c_2}$ that might be quite large, where $\epsilon\in(0,1)$ such that $\delta+\epsilon<1$. Unlike CoCaIn, our BMME with backtracking can use any $L^0$. The flexibility of the initialization $L^0$ may lead to a faster convergence. 

In the experiment, we set $\lambda = 0.1$, $\theta = 5$, $\delta = 0.99$, and $l^0 = 0.001$. We initialize $L^0 = 1.0001\frac{\lambda\theta^2}{(1-\delta - \epsilon)c_2}$ with $\epsilon = 0.009$ for CoCaIn while we choose $L^0 = 0.01$ for BMME with backtracking. We carry out the experiment on MovieLens 1M that contains $999,714$ ratings of $6,040$ different users. 
We choose $\mbfr = 5$ and randomly use 70\% of the observed ratings for training and the rest for testing. The process is repeated twenty times. We run each algorithm 20 seconds. 
We are interested in the root mean squared error on the test set: $RMSE = \sqrt{\|\mathcal P_T(A - UV)\|^2/N_T}$, where $\mathcal P_T(Z)_{ij} = Z_{ij}$ if $A_{ij}$ belongs to the test set and $0$ otherwise, $N_T$ is the number of ratings in the test set.
 We plotted the curves of the average value of RMSE and the objective function value versus training time in Figure \ref{fig:MCP}. 
 
\begin{figure}[ht]
\begin{center}
\begin{tabular}{cc}
\includegraphics[width =0.45\textwidth]{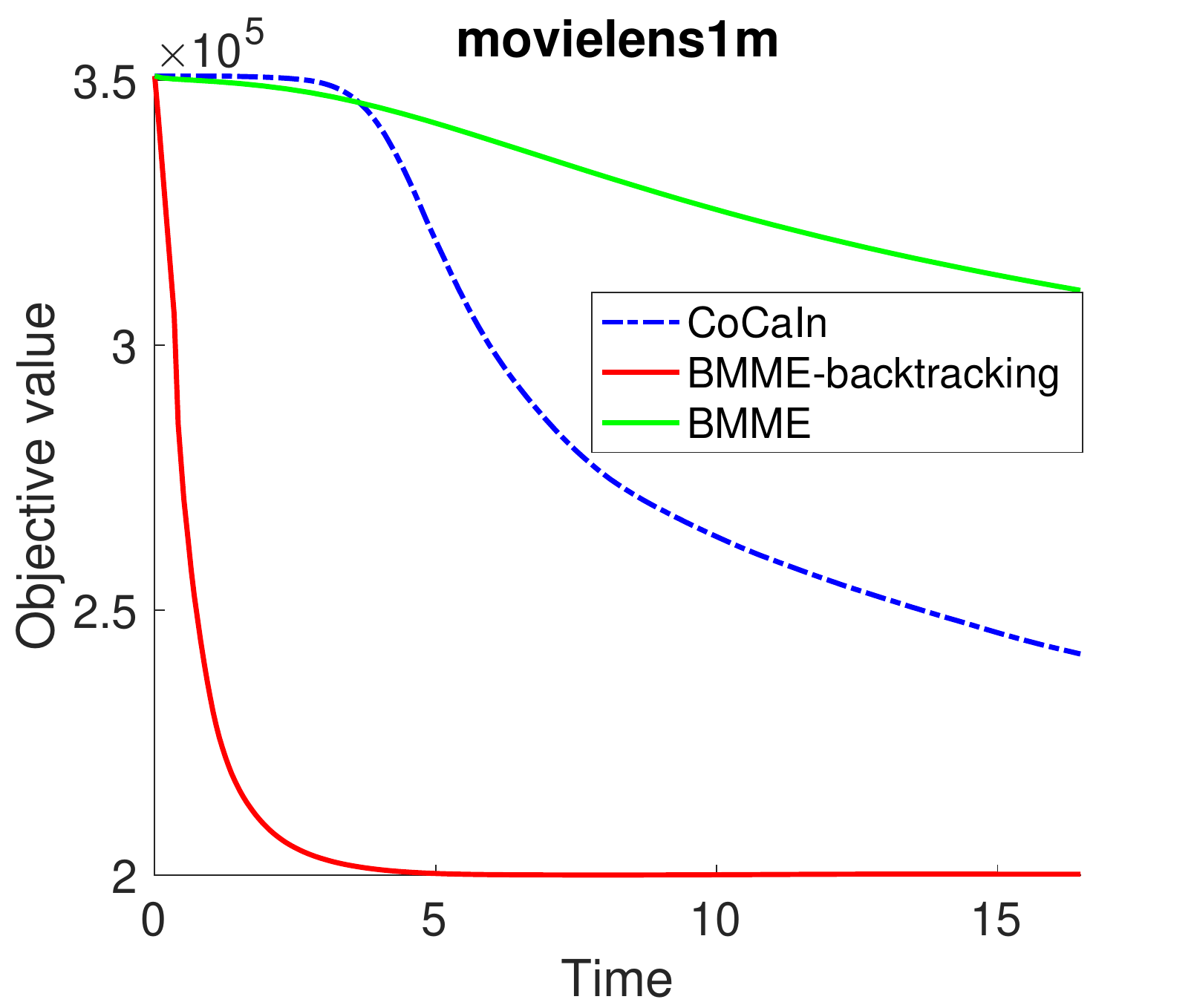}  &
\includegraphics[width=0.47\textwidth]{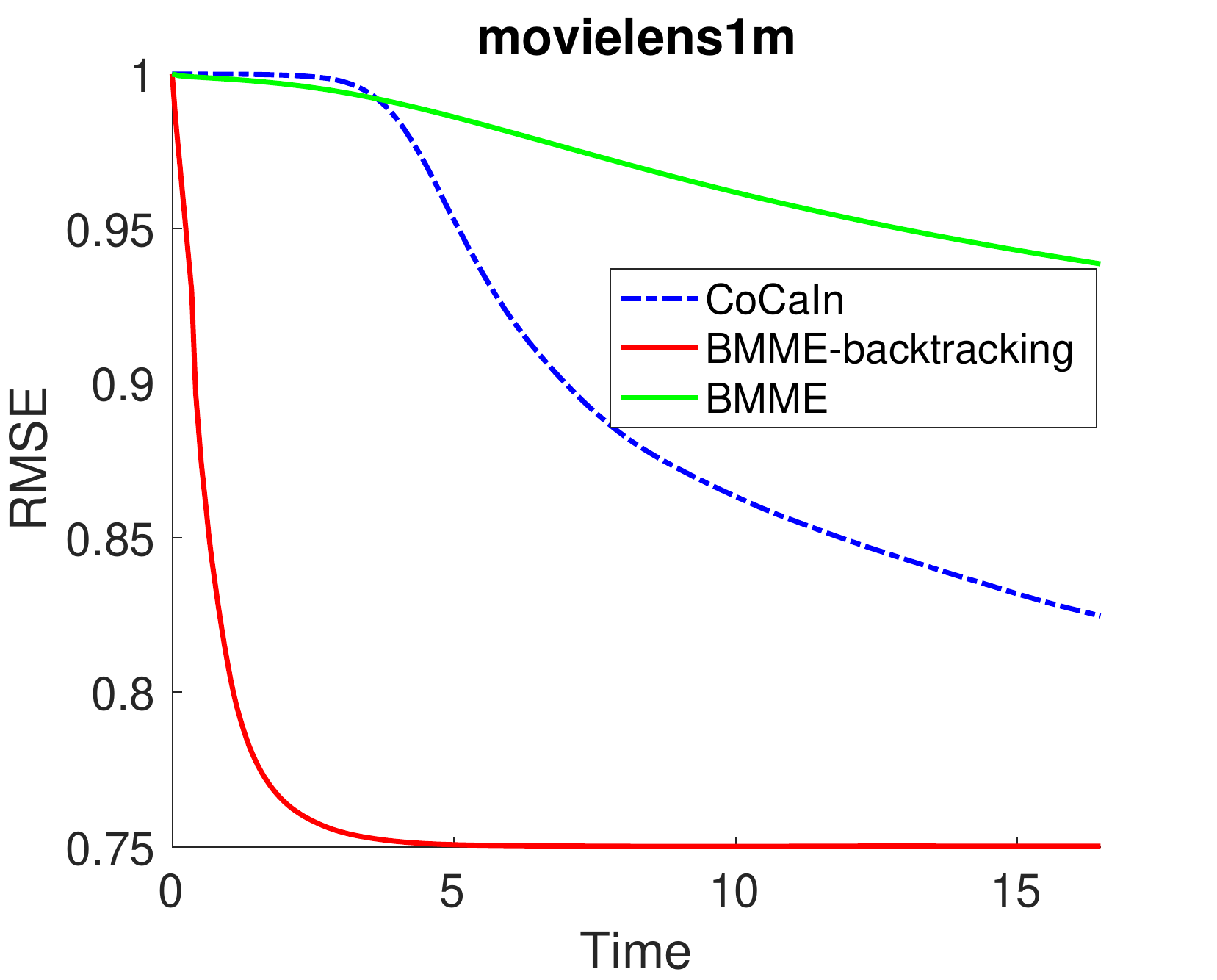}
\end{tabular}
\caption{BMME and CoCaIn applied on the MCP~\eqref{MF}. Evolution of the average value of the RMSE on the test set and the objective function value with respect to time.
\label{fig:MCP}} 
\end{center}
\end{figure} 

 We observe that BMME with backtracking converges much faster than CoCaIn and BMME without backtracking (BMME). 
This illustrates the usefulness of properly choosing the convex surrogate function and the backtracking line search for the relative smooth constants.


\bibliographystyle{siamplain}
\bibliography{iMM}

\begin{thebibliography}{10}

\bibitem{aharon2006k}
{\sc M.~Aharon, M.~Elad, A.~Bruckstein, et~al.}, {\em K-svd: An algorithm for
  designing overcomplete dictionaries for sparse representation}, IEEE
  Transactions on Signal Processing, 54 (2006), p.~4311.

\bibitem{ahookhosh2020_inertial}
{\sc M.~Ahookhosh, L.~T.~K. Hien, N.~Gillis, and P.~Patrinos}, {\em A block
  inertial {B}regman proximal algorithm for nonsmooth nonconvex problems with
  application to symmetric nonnegative matrix tri-factorization}, Journal of
  Optimization Theory and Applications,  (2021).

\bibitem{ahookhosh2019multi}
{\sc M.~Ahookhosh, L.~T.~K. Hien, N.~Gillis, and P.~Patrinos}, {\em Multi-block
  {B}regman proximal alternating linearized minimization and its application to
  sparse orthogonal nonnegative matrix factorization}, Computational
  Optimization and Application, 79 (2021), p.~681–715.

\bibitem{ahookhosh2019bregman}
{\sc M.~Ahookhosh, A.~Themelis, and P.~Patrinos}, {\em Bregman forward-backward
  splitting for nonconvex composite optimization: superlinear convergence to
  nonisolated critical points}, SIAM Journal on Optimization, 31 (2021),
  pp.~653--685.

\bibitem{Araujo01}
{\sc U.~Ara\'ujo, B.~Saldanha, R.~Galv\~ao, T.~Yoneyama, H.~Chame, and
  V.~Visani}, {\em The successive projections algorithm for variable selection
  in spectroscopic multicomponent analysis}, Chemometrics and Intelligent
  Laboratory Systems, 57 (2001), pp.~65--73.

\bibitem{Attouch2009}
{\sc H.~Attouch and J.~Bolte}, {\em On the convergence of the proximal
  algorithm for nonsmooth functions involving analytic features}, Mathematical
  Programming, 116 (2009), pp.~5--16,
  \url{https://doi.org/10.1007/s10107-007-0133-5}.

\bibitem{Attouch2010}
{\sc H.~Attouch, J.~Bolte, P.~Redont, and A.~Soubeyran}, {\em Proximal
  alternating minimization and projection methods for nonconvex problems: An
  approach based on the {K}urdyka-{{\L}}ojasiewicz inequality}, Mathematics of
  Operations Research, 35 (2010), pp.~438--457,
  \url{https://doi.org/10.1287/moor.1100.0449}.

\bibitem{Attouch2013}
{\sc H.~Attouch, J.~Bolte, and B.~F. Svaiter}, {\em Convergence of descent
  methods for semi-algebraic and tame problems: proximal algorithms,
  forward--backward splitting, and regularized gauss--seidel methods},
  Mathematical Programming, 137 (2013), pp.~91--129.

\bibitem{Bauschke2017}
{\sc H.~H. Bauschke, J.~Bolte, and M.~Teboulle}, {\em A descent lemma beyond
  {L}ipschitz gradient continuity: First-order methods revisited and
  applications}, Mathematics of Operations Research, 42 (2017), pp.~330--348,
  \url{https://doi.org/10.1287/moor.2016.0817}.

\bibitem{Beck2013}
{\sc A.~Beck and L.~Tetruashvili}, {\em On the convergence of block coordinate
  descent type methods}, SIAM Journal on Optimization, 23 (2013),
  pp.~2037--2060.

\bibitem{BLUMENSATH2009}
{\sc T.~Blumensath and M.~E. Davies}, {\em Iterative hard thresholding for
  compressed sensing}, Applied and Computational Harmonic Analysis, 27 (2009),
  pp.~265 -- 274, \url{https://doi.org/10.1016/j.acha.2009.04.002}.

\bibitem{Bochnak1998}
{\sc J.~Bochnak, M.~Coste, and M.-F. Roy}, {\em Real Algebraic Geometry},
  Springer, 1998.

\bibitem{Bolte2014}
{\sc J.~Bolte, S.~Sabach, and M.~Teboulle}, {\em Proximal alternating
  linearized minimization for nonconvex and nonsmooth problems}, Mathematical
  Programming, 146 (2014), pp.~459--494.

\bibitem{Bolte2018}
{\sc J.~Bolte, S.~Sabach, M.~Teboulle, and Y.~Vaisbourd}, {\em First order
  methods beyond convexity and {L}ipschitz gradient continuity with
  applications to quadratic inverse problems}, SIAM Journal on Optimization, 28
  (2018), pp.~2131--2151, \url{https://doi.org/10.1137/17M1138558}.

\bibitem{gillis2020book}
{\sc N.~Gillis}, {\em Nonnegative Matrix Factorization}, SIAM, Philadelphia,
  2020.

\bibitem{gillis2015hierarchical}
{\sc N.~Gillis, D.~Kuang, and H.~Park}, {\em Hierarchical clustering of
  hyperspectral images using rank-two nonnegative matrix factorization}, IEEE
  Transactions on Geoscience and Remote Sensing, 53 (2015), pp.~2066--2078.

\bibitem{gillis2013fast}
{\sc N.~Gillis and S.~A. Vavasis}, {\em Fast and robust recursive algorithms
  for separable nonnegative matrix factorization}, IEEE Transactions on Pattern
  Analysis and Machine Intelligence, 36 (2013), pp.~698--714.

\bibitem{HienNicolas_KLNMF}
{\sc L.~T.~K. Hien and N.~Gillis}, {\em Algorithms for nonnegative matrix
  factorization with the {K}ullback-{L}eibler divergence}, Journal of
  Scientific Computing,  (2021),
  \url{https://doi.org/10.1007/s10915-021-01504-0}.

\bibitem{Hien_ICML2020}
{\sc L.~T.~K. Hien, N.~Gillis, and P.~Patrinos}, {\em Inertial block proximal
  method for non-convex non-smooth optimization}, in Thirty-seventh
  International Conference on Machine Learning (ICML), 2020.

\bibitem{Titan2020}
{\sc L.~T.~K. Hien, D.~N. Phan, and N.~Gillis}, {\em Inertial block
  majorization minimization framework for nonconvex nonsmooth optimization}.
\newblock arXiv:2010.12133, 2020.

\bibitem{Kurdyka1998}
{\sc K.~Kurdyka}, {\em On gradients of functions definable in o-minimal
  structures}, Annales de l'Institut Fourier, 48 (1998), pp.~769--783,
  \url{https://doi.org/10.5802/aif.1638}.

\bibitem{lu2018relatively}
{\sc H.~Lu, R.~M. Freund, and Y.~Nesterov}, {\em Relatively smooth convex
  optimization by first-order methods, and applications}, SIAM Journal on
  Optimization, 28 (2018), pp.~333--354.

\bibitem{Mairal_ICML13}
{\sc J.~Mairal}, {\em Optimization with first-order surrogate functions}, in
  Proceedings of the 30th International Conference on International Conference
  on Machine Learning - Volume 28, ICML’13, JMLR.org, 2013, pp.~783--791.

\bibitem{MukkamalaO19}
{\sc M.~C. Mukkamala and P.~Ochs}, {\em Beyond alternating updates for matrix
  factorization with inertial {B}regman proximal gradient algorithms}, in
  Advances in Neural Information Processing Systems 32: Annual Conference on
  Neural Information Processing Systems 2019, NeurIPS 2019, December 8-14,
  2019, Vancouver, BC, Canada, H.~M. Wallach, H.~Larochelle, A.~Beygelzimer,
  F.~d'Alch{\'{e}}{-}Buc, E.~B. Fox, and R.~Garnett, eds., 2019,
  pp.~4268--4278.

\bibitem{Mukkamala2020}
{\sc M.~C. Mukkamala, P.~Ochs, T.~Pock, and S.~Sabach}, {\em Convex-concave
  backtracking for inertial {B}regman proximal gradient algorithms in nonconvex
  optimization}, SIAM Journal on Mathematics of Data Science, 2 (2020),
  pp.~658--682, \url{https://doi.org/10.1137/19M1298007}.

\bibitem{Natarajan1995}
{\sc B.~Natarajan}, {\em Sparse approximate solutions to linear systems}, SIAM
  Journal on Computing, 24 (1995), pp.~227--234,
  \url{https://doi.org/10.1137/S0097539792240406}.

\bibitem{Nesterov1983}
{\sc Y.~Nesterov}, {\em A method of solving a convex programming problem with
  convergence rate {O}$(1/k^2)$}, Soviet Mathematics Doklady, 27 (1983).

\bibitem{Nesterov2018}
{\sc Y.~Nesterov}, {\em Lectures on Convex Optimization}, Springer Publishing
  Company, Incorporated, 2nd~ed., 2018.

\bibitem{Ochs2019}
{\sc P.~Ochs}, {\em Unifying abstract inexact convergence theorems and block
  coordinate variable metric ipiano}, SIAM Journal on Optimization, 29 (2019),
  pp.~541--570, \url{https://doi.org/10.1137/17M1124085}.

\bibitem{Pock2016}
{\sc T.~Pock and S.~Sabach}, {\em Inertial proximal alternating linearized
  minimization (i{PALM}) for nonconvex and nonsmooth problems}, SIAM Journal on
  Imaging Sciences, 9 (2016), pp.~1756--1787,
  \url{https://doi.org/10.1137/16M1064064}.

\bibitem{pompili2014two}
{\sc F.~Pompili, N.~Gillis, P.-A. Absil, and F.~Glineur}, {\em Two algorithms
  for orthogonal nonnegative matrix factorization with application to
  clustering}, Neurocomputing, 141 (2014), pp.~15--25.

\bibitem{Razaviyayn2013}
{\sc M.~Razaviyayn, M.~Hong, and Z.~Luo}, {\em A unified convergence analysis
  of block successive minimization methods for nonsmooth optimization}, SIAM
  Journal on Optimization, 23 (2013), pp.~1126--1153,
  \url{https://doi.org/10.1137/120891009}.

\bibitem{Teboulle2020}
{\sc M.~Teboulle and Y.~Vaisbourd}, {\em Novel proximal gradient methods for
  nonnegative matrix factorization with sparsity constraints}, SIAM Journal on
  Imaging Sciences, 13 (2020), pp.~381--421,
  \url{https://doi.org/10.1137/19M1271750}.

\bibitem{Tseng2009}
{\sc P.~Tseng and S.~Yun}, {\em A coordinate gradient descent method for
  nonsmooth separable minimization}, Mathematical Programming, 117 (2009),
  pp.~387--423.

\bibitem{Xu2013}
{\sc Y.~Xu and W.~Yin}, {\em A block coordinate descent method for regularized
  multiconvex optimization with applications to nonnegative tensor
  factorization and completion}, SIAM Journal on Imaging Sciences, 6 (2013),
  pp.~1758--1789, \url{https://doi.org/10.1137/120887795},
  \url{https://doi.org/10.1137/120887795}.

\bibitem{XuYin2016}
{\sc Y.~Xu and W.~Yin}, {\em A fast patch-dictionary method for whole image
  recovery}, Inverse Problems \& Imaging, 10 (2016), p.~563,
  \url{https://doi.org/10.3934/ipi.2016012}.

\bibitem{Xu2017}
{\sc Y.~Xu and W.~Yin}, {\em A globally convergent algorithm for nonconvex
  optimization based on block coordinate update}, Journal of Scientific
  Computing, 72 (2017), pp.~700--734.

\bibitem{ZG05}
{\sc S.~Zhong and J.~Ghosh}, {\em Generative model-based document clustering: a
  comparative study}, Knowledge and Information Systems, 8 (2005),
  pp.~374--384.

\end{thebibliography}

\end{document}